\RequirePackage{iftex}
\ifXeTeX
  \RequirePackage{scrlfile}
  \PreventPackageFromLoading{breakurl}
\fi
\documentclass[final,nohypdvips,onefignum,onetabnum,oneeqnum,onealgnum,onethmnum]{siamart251216}
\usepackage{amssymb,graphicx,booktabs,algorithm,algpseudocode}
\newsiamremark{assumption}{Assumption}
\newsiamremark{remark}{Remark}
\newsiamremark{example}{Example}
\newsiamthm{property}{Property}
\theoremstyle{plain}
\theoremheaderfont{\normalfont\sc}
\theorembodyfont{\normalfont\itshape}

\expandafter\let\csname lemma*\endcsname\relax
\expandafter\let\csname endlemma*\endcsname\relax
\newtheorem{lemma}{Lemma}

\expandafter\let\csname corollary*\endcsname\relax
\expandafter\let\csname endcorollary*\endcsname\relax
\newtheorem{corollary}{Corollary}

\expandafter\let\csname proposition*\endcsname\relax
\expandafter\let\csname endproposition*\endcsname\relax
\newtheorem{proposition}{Proposition}

\expandafter\let\csname definition*\endcsname\relax
\expandafter\let\csname enddefinition*\endcsname\relax
\newtheorem{definition}{Definition}

\expandafter\let\csname property*\endcsname\relax
\expandafter\let\csname endproperty*\endcsname\relax
\newtheorem{property}{Property}
\theoremheaderfont{\normalfont\itshape}
\theorembodyfont{\normalfont}

\expandafter\let\csname assumption*\endcsname\relax
\expandafter\let\csname endassumption*\endcsname\relax
\newtheorem{assumption}{Assumption}

\expandafter\let\csname remark*\endcsname\relax
\expandafter\let\csname endremark*\endcsname\relax
\newtheorem{remark}{Remark}
\DeclareMathOperator*{\argmin}{arg\,min}
\DeclareMathOperator{\prox}{prox}
\DeclareMathOperator{\dom}{dom}
\DeclareMathOperator{\dist}{dist}
\DeclareMathOperator{\range}{range}
\DeclareMathOperator{\rank}{rank}
\DeclareMathOperator{\supp}{supp}
\newcommand{\R}{\mathbb R}
\newcommand{\E}{\mathcal E}
\newcommand{\V}{\mathcal V}
\newcommand{\Sscale}{\mathcal S}
\newcommand{\Heff}{\mathcal H_0}
\newcommand{\Qeff}{\mathcal Q_{\mathrm{eff}}}
\newcommand{\reff}{r_{\mathrm{eff}}}
\newcommand{\norm}[1]{\left\|#1\right\|}
\newcommand{\ip}[2]{\left\langle #1,#2\right\rangle}
\title{Controlled Block Proximal Linearized Method for Nonconvex and Nonsmooth Optimization}
\author{WeiFeng Yang}
\headers{Controlled Block Proximal Linearized Method}{WeiFeng Yang}

\begin{document}

\maketitle

\begin{abstract}
In this paper, we consider multiblock nonconvex and nonsmooth optimization without assuming bounded iterates or coercivity of the original objective.
We propose the Controlled Block Proximal Linearized (CBPL) method, whose control potentials establish boundedness while retaining the original objective, constraints and variables.
We prove finite length and whole-sequence convergence to an effective critical point of the original objective, where one limiting subgradient annihilates all effective feasible directions.
We establish finite total variation of the original objective and block functions, summable original effective residuals, and finite termination and linear or sublinear convergence rates for the original objective values.
Scale orthogonality further gives full original criticality and, for homogeneous balancing controls, summable absolute objective errors with an $o(1/K)$ best-iterate bound.
We establish a complete degeneracy theory for graph-regularized nonnegative factorizations: scale criticality is equivalent to graph collapse, and full criticality eliminates graph action; our controls prevent this collapse while retaining genuine reconstruction changes.
Furthermore, our proposed method provides the practical, simple, and effective schemes for these applications, combining guaranteed convergence with explicit closed-form updates that can be applied directly to solve the original problems. 
We apply CBPL to graph-regularized sparse nonnegative matrix factorization with $\ell_0$-constraints ($\ell_0$-GSNMF), graph-regularized sparse nonnegative canonical polyadic (CP) decomposition with $\ell_0$-constraints ($\ell_0$-GSNCP), and graph-regularized sparse nonnegative independent-core matrix tri-factorization with $\ell_0$-constraints ($\ell_0$-GSNMTF).
Numerical experiments demonstrate effectiveness, convergence and structural control.
\end{abstract}

\begin{keywords}
nonconvex and nonsmooth optimization, block proximal linearization,
scale degeneracy, graph regularization, sparse nonnegative factorization
\end{keywords}
\begin{AMS}
90C26, 90C30, 15A69
\end{AMS}

\section{Introduction}
\label{sec:introduction}
This paper considers a class of nonconvex and nonsmooth optimization problems, which can be expressed as
\begin{equation}
\min_{\{x_i\}_{i=1}^{N}} J(\{x_i\}_{i=1}^{N})
=H(\{x_i\}_{i=1}^{N})+\sum_{i=1}^{N}F_i(x_i),
\label{e11}
\end{equation}
where $x_i\in\R^{d_i}$, $H$ is continuously differentiable and each $F_i:\R^{d_i}\to(-\infty,+\infty]$ is proper and lower semicontinuous, e.g., the indicator of a nonnegativity or $\ell_0$ constraint. No boundedness or coercivity assumption is imposed on Eq.~(\ref{e11}). This formulation covers many applications, such as nonnegative matrix factorization (NMF) \cite{lee1999learning,lee2001algorithms}, sparse tensor decomposition \cite{kolda2009tensor} and graph regularization \cite{cai2011graph}.

Block proximal linearization provides explicit updates for many instances of Eq.~(\ref{e11}). The proximal alternating linearized minimization (PALM) method \cite{bolte2014proximal} combines a linearization of the smooth coupling with a block proximal operator step, thereby establishing a convergence framework. Block coordinate methods \cite{xu2013block,xu2017globally} and abstract descent analyses \cite{attouch2009convergence,attouch2013convergence} also establish convergence guarantees for nonconvex sequences. Boundedness supplies accumulation points and uniform local constants in this analysis. Coercivity gives bounded sublevel sets, so a descent sequence remains bounded when the objective is coercive.

Inertial strategies introduce extrapolation into the block updates to improve numerical behavior. For example, inertial proximal alternating linearized minimization (iPALM) \cite{pock2016ipalm} uses two extrapolation points, the accelerated block proximal linear framework with adaptive momentum (ABPL$^+$) \cite{yang2023abpl} combines adaptive momentum with a comparison of proximal and extrapolated steps, and the inertial block majorization-minimization framework (TITAN) \cite{hien2023inertial} provides general block surrogates and extrapolation operators. These extensions design the updates and descent quantities to satisfy the convergence framework. Their whole-sequence convergence analysis still requires control of the generated iterates.

Splitting methods are also used for nonconvex and nonsmooth problems with linear constraints. Wang et al.~\cite{wang2019global} analyze the multiblock alternating direction method of multipliers (ADMM), deriving boundedness from coercivity on the feasible set and proving whole-sequence convergence under the Kurdyka--\L{}ojasiewicz (KL) property. Bai et al.~\cite{bai2025inexact} develop an inexact ADMM with relative-error criteria and an expansion line search.

For factorization models, controlling the iterates is closely connected with factor scales. Practical formulations introduce normalization, scale constraints or positive regularizers to obtain stable representations. Taslaman and Nilsson \cite{taslaman2012framework} study regularized nonnegative matrix factorization with scale control, and Le Roux et al.~\cite{leroux2015sparse} construct sparse nonnegative factorizations with explicit treatment of factor scaling. Gu et al.~\cite{gu2011trivial} investigate trivial solutions and scale transfer in graph-regularized NMF. Cohen and Leplat \cite{cohen2024efficient} analyze homogeneous regularization and factor balancing in scale-invariant low-rank models. Positive quadratic penalties also provide coercivity for scale-balanced sparse factorization methods \cite{he2026scale}. Alternating proximal techniques have been applied to sparse nonnegative Tucker decomposition \cite{xu2015ntd}.

However, even elementary NMF \cite{lee1999learning,lee2001algorithms} is noncoercive. Multiplying one nonzero factor by $t>0$ and dividing the other by $t$ preserves the reconstruction error and the original hard sparsity constraints, while a factor norm tends to infinity as $t\to\infty$. Thus the same finite objective value occurs at arbitrarily large factor norms, giving an unbounded sublevel set. Adding positive penalties changes the objective and its criticality conditions, and imposing norm constraints changes the feasible set of the factors. The resulting convergence guarantees concern these modified models. Establishing convergence for the original noncoercive problem requires a mechanism that controls the iteration without replacing its modeling terms or constraints.

Moreover, graph regularization introduces a more fundamental structural conflict. A scale change can leave the reconstruction unchanged and drive the graph penalty to zero. Let $A_N$ be the factor carrying the graph term and $L$ be its graph matrix. We prove that criticality along all genuine reconstruction-preserving scales is equivalent to $LA_N=0$, and full criticality of the original objective necessarily implies this condition. For a connected combinatorial graph Laplacian, $LA_N=0$ forces every sample representation to be the same. Thus full criticality eliminates the differences that graph regularization is intended to organize. This structural conflict concerns the convergence target itself. A meaningful convergence theory must retain nonzero graph action and establish a first-order condition for the original objective through directions that retain genuine reconstruction changes.

To tackle these challenges, we propose the Controlled Block Proximal Linearized (CBPL) method for Eq.~(\ref{e11}). The method adds a control potential to the smooth block update and keeps the original functions $F_i$. We give control design conditions that produce compact controlled sublevel sets and transfer the convergence certificate to the original objective. This provides a general convergence framework without assuming bounded iterates or coercivity of $J$, while retaining the original objective, feasible set and variables.

Within this framework, we prove finite length and whole-sequence convergence of the original variables to an effective critical point of the original objective. The original objective and every block function $F_i$ have finite total variation and converge to their values at the limit. The controlled subgradient certificates and original effective residuals are summable. We establish the convergence rates of the controlled objective, the original variables and the original objective in separate results. The original objective values have finite termination and linear or sublinear convergence according to the KL exponent. Orthogonality between the original subgradients and the control gradient further gives full original criticality and a quadratic bound on the remaining control variation, which strengthens the original objective rate. For nonnegative homogeneous balancing controls, the absolute errors of the original objective values are summable and their best-iterate rate is $o(k^{-1})$.

For graph models, the conventional criticality target must also be reconciled with the graph degeneracy proved above. We identify the complete feasible reconstruction fibers, including singular and zero structures, and prove that projection along genuine scale actions retains every reconstruction-changing feasible direction and every change along feasible curves. The control gradient vanishes on the resulting effective directions, so a controlled certificate gives one original limiting subgradient that annihilates them all. The graph controls keep graph action uniformly positive. Their noncollapsed limits also have full original block criticality on every factor without graph regularization and on the independent core, together with exact balance and an explicit graph-scale equilibrium.

The same framework applies to ordinary nonnegative matrix factorization, canonical polyadic (CP) decomposition, Tucker decomposition, independent-core matrix tri-factorization, their $\ell_0$-constrained forms and related problems. The original objective is constant along their feasible scale curves, and its subgradients are orthogonal to the balancing-control gradient. This relation recovers full original criticality and summability of the full original residuals. The homogeneous controls vanish and the balance equations hold exactly at the limit. These models inherit the general convergence guarantees and the stronger original-objective rates without preassigned boundedness or original-objective coercivity.

We apply CBPL to graph-regularized sparse nonnegative matrix factorization with $\ell_0$-constraints ($\ell_0$-GSNMF), graph-regularized sparse nonnegative CP decomposition with $\ell_0$-constraints ($\ell_0$-GSNCP), and graph-regularized sparse nonnegative independent-core matrix tri-factorization with $\ell_0$-constraints ($\ell_0$-GSNMTF). The control constructions also cover the general graph-regularized sparse nonnegative Tucker model with an independent core and $\ell_0$-constraints. They give explicit nonnegative hard-thresholding projections and computable safe parameters with one projected update per block. For the graph Laplacians used in these applications, the control calculations preserve the per-iteration order of PALM with the same original gradient, curvature and projection computations. Thus the applications combine the general guarantees with explicit, simple block updates and additional structure preservation.
\subsection{Contributions}
The main contributions are summarized as follows.

(1) We establish a complete degeneracy theory for $\ell_0$-GSNMF, $\ell_0$-GSNCP, $\ell_0$-GSNMTF and the general graph-regularized sparse nonnegative Tucker model with an independent core and $\ell_0$-constraints. We characterize all component scales, independent-core transformations and exact feasible reconstruction fibers, and prove that scale criticality is equivalent to graph collapse. Full criticality of the original graph objective therefore forces collapse, revealing a structural contradiction between the conventional criticality target and preservation of graph representations. This theory includes zero components, singular cores and changes of active support.

(2) We propose the CBPL method for Eq.~(\ref{e11}) and construct controls for graph-regularized and ordinary nonnegative factorization while retaining the original objective, feasible set and variables. The controls prevent divergence along reconstruction-preserving scales and give compact controlled sublevel sets. For every positive graph-factor budget, they also enforce a uniform positive lower bound on graph action. We derive explicit block projections and finite formulas for safe parameters, giving one projected update per block. The computational-cost property and the three application formulas show that the controls preserve the per-iteration order of the corresponding PALM updates with the same original gradient, curvature and projection computations for graph Laplacians.

(3) We establish a general convergence framework for CBPL under its control design conditions, without assuming bounded iterates or coercivity of the original objective. The entire original sequence has finite length and converges to a finite limit, and the original objective and every block function converge to their values at that limit with finite total variation. We establish finite termination and linear or sublinear rates separately for the controlled objective, the original variables and the original objective values, including stronger original-objective rates under scale orthogonality. The controlled objective gaps are summable and decay as $o(k^{-1})$. Controlled certificates and original effective residuals are summable, with an $o(k^{-1})$ best-iterate bound. The graph-regularized and ordinary constructions verify these conditions for all their original factors and independent cores.

(4) We establish original criticality and exact balance for graph-regularized and ordinary problems. For graph models, every factor without graph regularization and the independent core satisfy full original block criticality. A single original limiting subgradient annihilates all effective comparison directions, which retain reconstruction changes. The noncollapsed limits satisfy exact balance among the factors without graph regularization or between their columns and core slices, and an explicit graph-scale equilibrium. For ordinary nonnegative matrix factorization, CP decomposition, Tucker decomposition, independent-core matrix tri-factorization, their $\ell_0$-constrained forms and related problems, we establish full original criticality and summability of the full original residuals. The homogeneous balancing constructions give vanishing control and exact balance, summable absolute original-objective errors and an $o(k^{-1})$ best-iterate objective-error bound. Both branches retain the original objective and feasible set and require neither bounded iterates nor coercivity of the original objective.

Numerical experiments on $\ell_0$-GSNMF, $\ell_0$-GSNCP and $\ell_0$-GSNMTF examine the convergence and structural properties with direct safe-parameter updates.

\section{Notation and preliminaries}
\label{sec:preliminaries}
In this section, we introduce the notation and definitions used in the construction and convergence analysis of CBPL. All spaces are finite dimensional. We use the Euclidean inner product and the Frobenius norm on matrix and tensor blocks. For a block vector $x_{(N)}=\{x_i\}_{i=1}^{N}$, $\norm{x_{(N)}}^2=\sum_{i=1}^N\norm{x_i}^2$. For a direction $d_i$ in block $i$, we write $x+td_i=(x_1,\ldots,x_{i-1},x_i+td_i,x_{i+1},\ldots,x_N)$. For a set $E$, $\dist(x,E)=\inf_{y\in E}\norm{x-y}$, with $\inf\varnothing=+\infty$. The domain, range and kernel of a mapping are denoted by $\dom$, $\range$ and $\ker$, respectively.
For a matrix or tensor $A$, $\norm A_0$ denotes the number of its nonzero entries, $\supp(A)$ is their index set, and $[A]_+$ is its entrywise positive part.
Tensors are denoted by $\mathcal{X}$ and $\mathcal{G}$, and their mode-$i$
unfoldings by $X^{(i)}$ and $G^{(i)}$. Matrix data and cores are denoted
by $X$ and $G$. Table~\ref{notation} summarizes the frequently used notation.

\begin{table}[!ht]
\centering
\caption{Summary of frequently used notations}
\begin{tabular}{p{0.21\linewidth}|p{0.69\linewidth}} \hline
Notation & Definition \\ \hline
$[n]$ & $\{1,\ldots,n\}$, for a positive integer $n$ \\
$x_{(N)}$ & The complete block vector $\{x_i\}_{i=1}^{N}$ \\
$x_i^k$ & Block $i$ at iteration $k$ \\
$\mathcal{X},\mathcal{G}$ & The data tensor and independent core tensor \\
$X^{(i)},G^{(i)}$ & Mode-$i$ unfoldings of $\mathcal{X}$ and $\mathcal{G}$ \\
$J$ & The original objective $H+\sum_{i=1}^{N}F_i$ \\
$C$ & The control potential used in the block updates \\
$\Phi$ & The controlled objective $J+C$ \\
$h_i(x_i)$ & The smooth block function defined in Eq.~(\ref{eq:block-function}) \\
$\sigma_i^k$ & The step size of block $i$, with $c_i^k=\frac{1}{\sigma_i^k}$ \\
$\prox_{\sigma F}$ & The proximal mapping of $F$ with step size $\sigma$ \\
$\widehat\partial J,\ \partial J$ & The Fr\'echet and limiting subdifferentials of $J$ \\
$\widehat N_E,\ N_E$ & The regular and limiting normal cones of $E$ \\
$\|\cdot\|_F,\ \|\cdot\|_0$ & The Frobenius norm and the number of nonzero entries \\ \hline
\end{tabular}
\label{notation}
\end{table}

\subsection{Functions and proximal mappings}
\label{subsec:function-preliminaries}
We first recall basic definitions for extended-valued functions \cite{rockafellar1998variational} and proximal mappings \cite{bolte2014proximal}.

\begin{definition}\label{def:functions}
Let $F:\R^d\to(-\infty,+\infty]$. Its effective domain is $\dom F=\{x:F(x)<+\infty\}$, and $F$ is proper if $\dom F\ne\varnothing$. It is lower semicontinuous at $x$ if
\begin{equation*}
F(x)\leq\liminf_{\ell\to\infty}F(x^\ell)
\quad\text{whenever }x^\ell\to x.
\end{equation*}
The function is bounded below if $F(x)\geq m$ for some $m\in\R$ and every $x$. It is coercive if $F(x)\to+\infty$ as $\norm{x}\to+\infty$. Its sublevel set at $M\in\R$ is $\{x:F(x)\leq M\}$.
\end{definition}

A lower semicontinuous coercive function has compact finite sublevel sets. In this paper, compactness will instead be obtained from the control construction. For a set $E$, its indicator is $\delta_E(x)=0$ on $E$ and $+\infty$ outside $E$. It is proper and lower semicontinuous when $E$ is nonempty and closed.

\begin{definition}\label{def:prox}
For a proper function $F$, the proximal mapping is
\begin{equation}
\prox_{\sigma F}(u)=\argmin_x\left\{F(x)+\frac{1}{2\sigma}\norm{x-u}^2\right\},\qquad \sigma>0.
\label{eq:prox}
\end{equation}
For $F=\delta_E$, it is the Euclidean projection $P_E(u)=\argmin_{x\in E}\norm{x-u}^2$.
\end{definition}
The proximal mapping can be set valued. For the nonnegative sparse sets used below, its values are obtained by explicit projection formulas.

\begin{definition}\label{def:smoothness}
Let $f$ be continuously differentiable on an open set $U$. Its gradient is Lipschitz continuous on a set $K\subset U$ if there is $L_K<+\infty$ such that
\begin{equation*}
\norm{\nabla f(x)-\nabla f(y)}\leq L_K\norm{x-y}
\quad(x,y\in K).
\end{equation*}
It is locally Lipschitz continuous on $U$ if each point of $U$ has a neighborhood on which this inequality holds. A block-gradient Lipschitz constant is defined by fixing the other blocks and applying this inequality to the selected block.
\end{definition}

If $\nabla f$ is $L$-Lipschitz continuous on a set containing the segment from $x$ to $x+d$, the descent inequality \cite{bolte2014proximal} gives
\begin{equation}
f(x+d)\leq f(x)+\ip{\nabla f(x)}d+\frac L2\norm d^2.
\label{eq:descentlemma}
\end{equation}

\subsection{Preliminaries for nonconvex analysis}
\label{subsec:nonconvex-preliminaries}
The subdifferential and normal-cone definitions \cite{rockafellar1998variational} describe the original first-order conditions, including changes of active support. The KL property then relates these first-order quantities to function decrease in the convergence analysis.

\begin{definition}\label{def:subdifferential}
For a proper lower semicontinuous function $J$ and $x\in\dom J$, the Fr\'echet subdifferential is
\begin{equation*}
\widehat\partial J(x)=\left\{g:\liminf_{y\to x,\,y\ne x}
\frac{J(y)-J(x)-\ip{g}{y-x}}{\norm{y-x}}\geq0\right\}.
\end{equation*}
The limiting subdifferential $\partial J(x)$ consists of all limits $g^\ell\to g$ with $g^\ell\in\widehat\partial J(x^\ell)$, $x^\ell\to x$ and $J(x^\ell)\to J(x)$. Both subdifferentials are empty outside $\dom J$. A full critical point satisfies $0\in\partial J(x)$.
\end{definition}

\begin{definition}\label{def:normals}
For a closed set $E$ and $x\in E$, the regular normal cone is
\begin{equation*}
\widehat N_E(x)=\left\{n:\limsup_{\substack{y\to x,\ y\in E\\y\ne x}}
\frac{\ip n{y-x}}{\norm{y-x}}\leq0\right\}.
\end{equation*}
The limiting normal cone $N_E(x)$ consists of limits $n^\ell\to n$ with $x^\ell\in E$, $x^\ell\to x$ and $n^\ell\in\widehat N_E(x^\ell)$. Equivalently, $\widehat N_E=\widehat\partial\delta_E$ and $N_E=\partial\delta_E$.
\end{definition}

\begin{proposition}\label{prop:subdifferential-calculus}
If a proper lower semicontinuous function $J$ has a local minimum at $x\in\dom J$, then $0\in\widehat\partial J(x)\subseteq\partial J(x)$. For a continuously differentiable function $C$, smooth addition gives
\begin{equation}
\begin{aligned}
\widehat\partial(J+C)(x)=\widehat\partial J(x)+\nabla C(x),
\\
\partial(J+C)(x)=\partial J(x)+\nabla C(x).
\end{aligned}
\label{eq:smoothsum}
\end{equation}
In particular, for $J=H+\delta_E$,
\begin{equation}
\partial J(x)=\nabla H(x)+N_E(x).
\label{eq:sumrule}
\end{equation}
The limiting subdifferential is closed under value-compatible limits: if $x^\ell\to x$, $J(x^\ell)\to J(x)$, $g^\ell\to g$ and $g^\ell\in\partial J(x^\ell)$, then $g\in\partial J(x)$.
\end{proposition}
These properties follow from the subdifferential definitions and smooth first-order expansion \cite{rockafellar1998variational}.

The KL property connects subgradient estimates with function decrease in descent convergence analyses \cite{attouch2013convergence,bolte2014proximal}. We use it to prove finite length of the original-variable sequence.

\begin{definition}\label{def:KL}
A proper lower semicontinuous function $\Phi$ has the Kurdyka--\L{}ojasiewicz (KL) property at $\bar x\in\dom\partial\Phi$ if there are a neighborhood $U$ of $\bar x$, $\eta>0$ and a continuous concave function $\varphi:[0,\eta)\to\R_+$, continuously differentiable on $(0,\eta)$, with $\varphi(0)=0$ and $\varphi'>0$, such that
\begin{equation}
\varphi'(\Phi(x)-\Phi(\bar x))\dist(0,\partial\Phi(x))\geq1
\label{eq:KL}
\end{equation}
whenever $x\in U$ and $\Phi(\bar x)<\Phi(x)<\Phi(\bar x)+\eta$. It is a KL function if this property holds at every point of $\dom\partial\Phi$.
\end{definition}

A set is semialgebraic if it is a finite union of sets described by finitely many polynomial equalities and inequalities. A function is semialgebraic when its graph over its effective domain is semialgebraic. Proper lower semicontinuous semialgebraic functions have the KL property, and locally admit a power desingularizing function \cite{bolte2007lojasiewicz,bolte2007clarke}. We write this function as $\varphi(t)=\frac{c_{\mathrm{KL}}}{\theta}t^\theta$, where $c_{\mathrm{KL}}>0$ and $0<\theta\leq1$.

\begin{proposition}[Uniform KL property]\label{prop:uniformKL}
Let $\Omega$ be a compact subset of $\dom\partial\Phi$. Suppose that $\Phi$ is proper and lower semicontinuous, is constant on $\Omega$ with value $\ell$, and has the KL property at every point of $\Omega$. Then there are $\varepsilon>0$, $\eta>0$ and one desingularizing function $\varphi$ as in Definition~\ref{def:KL} such that
\begin{equation}
\varphi'(\Phi(x)-\ell)\dist(0,\partial\Phi(x))\geq1
\label{eq:uniformKL}
\end{equation}
whenever $\dist(x,\Omega)<\varepsilon$ and $\ell<\Phi(x)<\ell+\eta$.
\end{proposition}
This uniformization is used in the convergence analysis of descent methods \cite{attouch2013convergence,bolte2014proximal}.

\subsection{Reconstruction geometry and effective criticality}
\label{subsec:reconstruction-preliminaries}
The next definitions distinguish feasible changes that preserve a reconstruction
from those that change it. They supply the geometric objects used to construct
effective directions and state original criticality.
\begin{definition}\label{def:tangent-fiber}
The contingent tangent cone of $E$ at $x\in E$ is
\begin{equation}
T_E(x)=\left\{d:\exists\,t_\ell\downarrow0,\ x^\ell\in E,
\ \frac{x^\ell-x}{t_\ell}\to d\right\}.
\label{eq:tangent}
\end{equation}
For a reconstruction map $Z$, the exact feasible reconstruction fiber and its tangent cone are
\begin{equation}
E_f(x)=\{y\in E:Z(y)=Z(x)\},\qquad
\V(x)=T_{E_f(x)}(x).
\label{eq:fiber}
\end{equation}
\end{definition}
For a differentiable reconstruction map, $\V(x)\subseteq T_E(x)\cap\ker DZ(x)$. The exact fiber distinguishes reconstruction-preserving feasible variations from genuine reconstruction changes, including changes at singular points.

We use smooth feasible transformations to identify the reconstruction-preserving directions. For $x\in E$, let $\mathcal D_\gamma y$ be jointly continuously differentiable in $(\gamma,y)$ on an ambient neighborhood of $(0,x)$, with $\mathcal D_0y=y$. The transformations preserve the original feasible set and reconstruction if
\begin{equation}
 \mathcal D_\gamma y\in E,\qquad
 Z(\mathcal D_\gamma y)=Z(y)
 \label{eq:feasible-transform}
\end{equation}
for every feasible $y$ near $x$ and every sufficiently small parameter $\gamma$ of either sign. Their velocity map and velocity space at $x$ are
\begin{equation}
 T_x\gamma=\left.\frac{d}{dt}\mathcal D_{t\gamma}x\right|_{t=0},
 \qquad \Sscale(x)=\range T_x.
 \label{eq:scale-space}
\end{equation}
The map $T_x$ is linear in $\gamma$. In the factorization models, these transformations are the feasible scaling transformations, and $\Sscale(x)$ is the scale space. Whenever a linear space $\Heff(x)$ satisfies
\begin{equation}
 \R^{\dim x}=\Sscale(x)\oplus\Heff(x),
 \qquad W(x)=\Heff(x)^\perp,
 \label{eq:effective-spaces}
\end{equation}
$\Pi_x$ denotes the linear projection onto $\Heff(x)$ along $\Sscale(x)$: if $d=s+h$ with $s\in\Sscale(x)$ and $h\in\Heff(x)$, then $\Pi_xd=h$. The effective feasible directions are
\begin{equation}
 \Qeff(x)=(T_E(x)\cap\Heff(x))\setminus\V(x).
 \label{eq:effective}
\end{equation}
They satisfy the original tangent constraints and remove directions tangent to the exact reconstruction fiber. The model constructions below specify $\Heff(x)$ and prove that $\Pi_x$ preserves feasible reconstruction changes. For full-space comparison, we use $\Qeff(x)=\Heff(x)=\R^{\dim x}$ and $W(x)=\{0\}$.

\begin{definition}[Effective criticality]\label{def:effective-criticality}
For a proper lower semicontinuous objective $J$ and comparison directions $\Qeff(x)$, define the original-objective residual by
\begin{equation}
r_{\mathrm{eff}}(x)=\inf_{p\in\partial J(x)}
 \sup_{d\in\Qeff(x),\ \|d\|=1}|\langle p,d\rangle|.
\label{eq:residual}
\end{equation}
The supremum is zero if there is no unit comparison direction. If
$\Qeff(x)$ is the full ambient space, this residual is
$\operatorname{dist}(0,\partial J(x))$.
We call $x$ an effective critical point of $J$ if a single
$p\in\partial J(x)$ satisfies $\langle p,d\rangle=0$,
$\forall d\in\Qeff(x)$.
\end{definition}

\section{Problem formulation and the CBPL method}
\label{sec:method}
In this section, we present the original problem, the controlled block update and the conditions used to construct its control potential. The explicit model constructions and block formulas are presented after the general convergence analysis.

\subsection{Original problem and controlled block update}
We consider Eq.~(\ref{e11}) under the following assumptions.
\begin{assumption}\label{ass:original}
Each $F_i:\R^{d_i}\to(-\infty,+\infty]$ is proper and lower semicontinuous. The function $H:\prod_{i=1}^N\R^{d_i}\to\R$ is continuously differentiable, and its gradient is Lipschitz continuous on bounded sets. The original objective $J$ is bounded below.
\end{assumption}
The functions $F_i$ may be nonconvex and nonsmooth and include the original feasibility constraints. These assumptions retain the original block-composite problem without assuming bounded iterates or coercivity.

Let $C$ be a control potential, finite and smooth on an open set $U$. Define the smooth coupling $f$ and the controlled objective $\Phi$ by
\begin{equation}
f=H+C,\qquad \Phi=J+C=f+\sum_{i=1}^NF_i.
\label{eq:merit}
\end{equation}
At iteration $k$, define the function for block $i$ by
\begin{equation}
h_i(x_i)=(H+C)\left(\{x_j^{k+1}\}_{j=1}^{i-1},x_i,
                         \{x_j^k\}_{j=i+1}^{N}\right).
\label{eq:block-function}
\end{equation}
The Controlled Block Proximal Linearized (CBPL) update is
\begin{equation}
\begin{aligned}
x_i^{k+1}
&\in\prox_{\sigma_i^k F_i}\left(x_i^k-
                 \sigma_i^k\nabla_{x_i}h_i(x_i^k)\right)\\
&=\argmin_{u_i}\left\{F_i(u_i)+
\ip{\nabla_{x_i}h_i(x_i^k)}{u_i-x_i^k}
+\frac{1}{2\sigma_i^k}\norm{u_i-x_i^k}^2\right\}.
\end{aligned}
\label{e12}
\end{equation}
Here $\sigma_i^k=\frac{1}{c_i^k}$. The control enters the smooth gradient, while each original function $F_i$ is retained in the proximal step. For a current vector $x=(x_1,\ldots,x_N)$, a block-$i$ update replaces only $x_i$ and gives
\begin{equation*}
 x^+=(x_1,\ldots,x_{i-1},x_i^+,x_{i+1},\ldots,x_N).
\end{equation*}
We select the control and the block parameter rule to satisfy the following conditions.

\begingroup
\renewcommand{\theenumi}{C\arabic{enumi}}
\renewcommand{\labelenumi}{(\theenumi)}
\begin{enumerate}
\item\label{cond:c1} The extended function $C$ is lower semicontinuous and bounded below, is finite and continuously differentiable on $U$, and equals $+\infty$ outside $U$. The gradient of $f$ is locally Lipschitz continuous on $U$, and an initial point $x_{(N)}^0\in\dom J\cap U$ is available.
\item\label{cond:c2} Every nonempty finite sublevel of $\Phi$ is a compact subset of $U$, and $\Phi$ is a KL function.
\item\label{cond:c3} At a current intermediate block iterate $x\in\dom J\cap U$, the parameter rule selects a finite coefficient $c_i^k\geq2\delta$, for a fixed $\delta>0$, and a global block minimizer $x_i^+$ in Eq.~(\ref{e12}). The selected updated state $x^+$ remains in $U$ and satisfies
\begin{equation}
f(x^+)-f(x)\leq\ip{\nabla_{x_i}f(x)}{x_i^+-x_i}
+\frac{c_i^k-\delta}{2}\norm{x_i^+-x_i}^2.
\label{eq:descenttest}
\end{equation}
The parameter rules are locally bounded uniformly in $k$: every point of $\dom J\cap U$ has a relative neighborhood on which all selected coefficients have a common finite upper bound.
\item\label{cond:c4} At every generated intermediate block iterate and every finite accumulation point, the original problem geometry supplies comparison directions $\Qeff(x)\subseteq\Heff(x)$ as introduced in Section~\ref{subsec:reconstruction-preliminaries}. The original model determines its pathological or redundant directions $\V(x)$; for reconstruction models, these are the exact feasible fiber directions in Eq.~(\ref{eq:fiber}). The control satisfies $\nabla C(x)\in W(x)$. Whenever $\V(x)=\{0\}$, full-space comparison is required: $\Qeff(x)=\Heff(x)=\R^{\sum_i d_i}$. Consequently, $W(x)=\{0\}$ and $\nabla C(x)=0$ at these points.
\end{enumerate}
\endgroup

Conditions (\ref{cond:c1})--(\ref{cond:c3}) specify the control and solver used to obtain descent, boundedness and convergence. Condition (\ref{cond:c4}) transfers their certificates to the original objective. The model constructions below identify the effective directions from the exact original feasible reconstruction fibers, prove that the scale projection retains reconstruction changes, and verify these conditions for the explicit graph controls. Thus the control geometry and the block parameter rule are determined by the original model.

\subsection{The algorithm and parameter selection}
Algorithm~\ref{alg:cbpl} applies one cyclic sweep of Eq.~(\ref{e12}) at each iteration. Any finite initialization required to reach $U$ is part of the control construction; the graph models have an explicit initialization from every original feasible input.

\begin{algorithm}[H]
\caption{Controlled Block Proximal Linearized method (CBPL)}\label{alg:cbpl}
\begin{algorithmic}[1]
\Require Original block functions $F_i$, coupling $H$, control $C$ and parameter rules satisfying (\ref{cond:c1})--(\ref{cond:c4}), and $x_{(N)}^0\in\dom J\cap U$.
\For{$k=0,1,2,\ldots$}
\For{$i=1,\ldots,N$}
\State Form $h_i(x_i)$ by Eq.~(\ref{eq:block-function}).
\State Choose $\sigma_i^k=\frac{1}{c_i^k}$ by the block parameter rule.
\State Compute $x_i^{k+1}\in\prox_{\sigma_i^k F_i}\big(x_i^k-\sigma_i^k\nabla_{x_i}h_i(x_i^k)\big)$.
\EndFor
\EndFor
\end{algorithmic}
\end{algorithm}

\begin{remark}\label{rem:parameters}
For the factorization models in Section~\ref{sec:experiments}, the finite
formulas there give
\begin{align}
c_i^k=C_{\mathrm{safe}}\left(
       (\{x_j^{k+1}\}_{j=1}^{i-1},\{x_j^k\}_{j=i}^{N}),i\right). \notag
\end{align}
Each block then requires one projection onto its original constraints.
These formulas verify Eq. (\ref{eq:descenttest}) and are uniformly bounded
on each controlled sublevel. Multiplying this coefficient by $t_i^k$,
with $1\leq t_i^k\leq\kappa<+\infty$ for a fixed $\kappa$, preserves the
guarantees. The implementation uses $t_i^k=1$.
\end{remark}

\section{Convergence analysis}
\label{sec:convergence}
In this section, we establish the convergence properties of Algorithm~\ref{alg:cbpl} under Assumption~\ref{ass:original} and conditions (\ref{cond:c1})--(\ref{cond:c4}). We prove finite length and whole-sequence convergence of the original variables, convergence of the original objective values and first-order conditions for the original objective. We then give separate convergence rates for the controlled objective, the original variables and the original objective values, followed by original stationarity rates. The final subsection establishes the reconstruction and normal-cone properties used to verify original criticality in the applications.

CBPL uses the gradient of $H+C$, so its descent quantity is $\Phi=J+C$.
We use its decrease and the control conditions to derive bounded iterates
and uniform step-size bounds. The subgradient estimate and KL property then
give finite length of the original-variable sequence. These results allow us
to prove convergence of $J$ and establish its first-order conditions.
Property~\ref{prop:framework} states these steps, which are proved in order below.
\begin{property}[Convergence analysis framework for CBPL]\label{prop:framework}
Under Assumption~\ref{ass:original} and conditions (\ref{cond:c1})--(\ref{cond:c4}), the sequence
$\{x_{(N)}^k\}_{k\geq0}$ generated by Algorithm~\ref{alg:cbpl} has the following
properties.
Let $M=\Phi(x_{(N)}^0)$ and $K_M=\{x:\Phi(x)\leq M\}$.

(\romannumeral1) Sufficient decrease of the controlled objective, i.e.,
$\exists\rho>0$ s.t.
$\rho\|x_{(N)}^{k+1}-x_{(N)}^k\|^2
\leq\Phi(x_{(N)}^k)-\Phi(x_{(N)}^{k+1})$,
and $\Phi(x_{(N)}^k)\downarrow\ell\in\mathbb R$.

(\romannumeral2) Bounded iterates and uniform step-size bounds, i.e.,
every intermediate block iterate lies in a compact $K_M\subset U$,
and $\exists\bar c<\infty$ s.t.
$0<\frac{1}{\bar c}\leq\sigma_i^k\leq\frac{1}{2\delta}$,
$\forall i\in[N]$, $k\geq0$.

(\romannumeral3) Subgradient bound, i.e., $\exists\rho_b>0$ and
$p_{x^{k+1}}\in\partial\Phi(x_{(N)}^{k+1})$ s.t.
$\|p_{x^{k+1}}\|\leq\rho_b\|x_{(N)}^{k+1}-x_{(N)}^k\|\to0$.
The cluster set $\Omega$ is nonempty and compact, and
$\Phi(\bar x)=\ell$, $0\in\partial\Phi(\bar x)$, $\forall\bar x\in\Omega$.

(\romannumeral4) Finite length and whole-sequence convergence, i.e.,
$\sum_{k=0}^{\infty}\|x_{(N)}^{k+1}-x_{(N)}^k\|<\infty$
and $x_{(N)}^k\to x^*\in K_M$.
The controlled objective errors are summable and satisfy
$\Phi(x_{(N)}^k)-\Phi(x^*)=o(k^{-1})$.

(\romannumeral5) Finite total variation of the original function values, i.e.,
$\sum_{k=0}^{\infty}|J(x_{(N)}^{k+1})-J(x_{(N)}^k)|<\infty$ and
$\sum_{k=0}^{\infty}|F_i(x_i^{k+1})-F_i(x_i^k)|<\infty$,
$\forall i\in[N]$, with $J(x_{(N)}^k)\to J(x^*)$ and
$F_i(x_i^k)\to F_i(x_i^*)$.

(\romannumeral6) Criticality for the original objective, i.e.,
$\sum_{k=1}^{\infty}r_{\mathrm{eff}}(x_{(N)}^k)<\infty$, and a single
$p^*=-\nabla C(x^*)\in\partial J(x^*)$ satisfies
$\langle p^*,d\rangle=0$, $\forall d\in\Qeff(x^*)$.
The scale orthogonality conditions in Theorem~\ref{thm:full} further give
$\sum_{k=1}^{\infty}\operatorname{dist}(0,\partial J(x_{(N)}^k))<\infty$
and $0\in\partial J(x^*)$.
\end{property}

\subsection{Sufficient decrease of the controlled objective}
We first prove Property~\ref{prop:framework}(i). The comparison in each original proximal subproblem gives decrease of $\Phi$ at every block update. Summing these inequalities gives the decrease over one complete sweep.
\begin{theorem}[Sufficient decrease and convergence of the controlled objective]\label{thm:descent}
Every intermediate block iterate is in $K_M=\{x:\Phi(x)\leq M\}$, where $M=\Phi(x_{(N)}^0)$. With $d_k=\|x_{(N)}^{k+1}-x_{(N)}^k\|$,
\begin{equation}
 \Phi(x_{(N)}^k)-\Phi(x_{(N)}^{k+1})\geq\frac{\delta}{2}d_k^2,
 \qquad \sum_{k=0}^{\infty}d_k^2<\infty,
 \qquad \Phi(x_{(N)}^k)\longrightarrow\ell\in\mathbb R.
 \label{eq:descent}
\end{equation}
\end{theorem}
\begin{proof}
(i) Each $F_i$ is proper and lower semicontinuous, and $H$ is finite and continuous. Their block-separable sum is proper and lower semicontinuous. The initial point has finite $J$ and $C$. Thus condition (\ref{cond:c1}) and Assumption~\ref{ass:original} make $\Phi$ proper and lower semicontinuous, with a finite lower bound $\underline\Phi$.

For the update of block $i$ in sweep $k$, write the current and updated vectors as
\begin{equation*}
 x=(\{x_j^{k+1}\}_{j=1}^{i-1},\{x_j^k\}_{j=i}^{N}),\qquad
 x^+=(\{x_j^{k+1}\}_{j=1}^{i},\{x_j^k\}_{j=i+1}^{N}).
\end{equation*}
Suppose that $x\in\dom J\cap U$ and $\Phi(x)\leq M$. At $k=0$ and $i=1$, these statements hold because $x=x_{(N)}^0\in\dom J\cap U$ and $\Phi(x_{(N)}^0)=M<\infty$. Put $\Delta_i^k=x_i^{k+1}-x_i^k$. Condition (\ref{cond:c3}) selects a finite $c_i^k$ and a global minimizer in Eq.~(\ref{e12}). Comparing it with $u_i=x_i^k\in\dom F_i$ gives
\begin{align}
 F_i(x_i^{k+1})+
 \langle\nabla_{x_i}h_i(x_i^k),\Delta_i^k\rangle
 +\frac{c_i^k}{2}\|\Delta_i^k\|^2
 &\leq F_i(x_i^k).
 \label{eq:block-prox-comparison}
\end{align}
The right-hand side is finite, so $F_i(x_i^{k+1})<\infty$. The other blocks remain in their respective domains, and condition (\ref{cond:c3}) gives $x^+\in U$. Hence $x^+\in\dom J\cap U$ and $\Phi(x^+)<\infty$. The same condition gives
\begin{align}
 h_i(x_i^{k+1})-h_i(x_i^k)
 &\leq\langle\nabla_{x_i}h_i(x_i^k),\Delta_i^k\rangle
       +\frac{c_i^k-\delta}{2}\|\Delta_i^k\|^2.
 \label{eq:block-smooth-comparison}
\end{align}
Rearranging Eq.~(\ref{eq:block-prox-comparison}) and using the upper bound in Eq.~(\ref{eq:block-smooth-comparison}) cancels the common inner product and gives
\begin{align}
 h_i(x_i^k)+F_i(x_i^k)-h_i(x_i^{k+1})-F_i(x_i^{k+1})
 &\geq\left(\frac{c_i^k}{2}-\frac{c_i^k-\delta}{2}\right)
          \|\Delta_i^k\|^2\notag\\
 &=\frac{\delta}{2}\|\Delta_i^k\|^2.
 \label{eq:block-controlled-decrease}
\end{align}
All other blocks have the same values on both sides. Adding their nonsmooth terms therefore gives
\begin{align}
&\Phi(\{x_j^{k+1}\}_{j=1}^{i-1},\{x_j^k\}_{j=i}^{N})
 -\Phi(\{x_j^{k+1}\}_{j=1}^{i},\{x_j^k\}_{j=i+1}^{N})
 \geq\frac{\delta}{2}\|\Delta_i^k\|^2.
 \label{eq:intermediate-controlled-decrease}
\end{align}
Equivalently,
\begin{equation*}
 \Phi(x^+)\leq\Phi(x)-\frac{\delta}{2}\|\Delta_i^k\|^2\leq M.
\end{equation*}
For $i<N$, $x^+$ is the current vector for block $i+1$; for $i=N$, it is $x_{(N)}^{k+1}$, the current vector for the next sweep. Thus the finite-value, domain and sublevel statements propagate through every block and every sweep.

(ii) By Eq.~(\ref{eq:block-function}), the coupling differences telescope:
\begin{align}
 \sum_{i=1}^{N}\bigl(h_i(x_i^k)-h_i(x_i^{k+1})\bigr)
 &=f(x_{(N)}^k)-f(x_{(N)}^{k+1}).\notag
\end{align}
The nonsmooth differences sum to $\sum_iF_i(x_i^k)-\sum_iF_i(x_i^{k+1})$. Summing Eq.~(\ref{eq:block-controlled-decrease}) and using the product norm, we obtain
\begin{align}
 \Phi(x_{(N)}^k)-\Phi(x_{(N)}^{k+1})
 &\geq\frac{\delta}{2}\sum_{i=1}^{N}\|\Delta_i^k\|^2
 =\frac{\delta}{2}d_k^2.\notag
\end{align}
Consequently $\Phi(x_{(N)}^k)$ is decreasing and bounded below by $\underline\Phi$. It has a finite limit $\ell$. For every integer $m\geq0$, finite summation gives
\begin{align}
 \frac{\delta}{2}\sum_{k=0}^{m}d_k^2
 &\leq\sum_{k=0}^{m}
       \bigl(\Phi(x_{(N)}^k)-\Phi(x_{(N)}^{k+1})\bigr)\notag\\
 &=\Phi(x_{(N)}^0)-\Phi(x_{(N)}^{m+1})\notag\\
 &\leq\Phi(x_{(N)}^0)-\underline\Phi<\infty.
 \label{eq:square-length-sum}
\end{align}
The nonnegative partial sums are uniformly bounded, so $\sum_kd_k^2<\infty$. Each summand therefore tends to zero, and $d_k\to0$. This proves Eq.~(\ref{eq:descent}).
\end{proof}

The following result gives a direct parameter rule for the original
projection in a constrained block update.

\begin{lemma}[Safe parameters for an original block projection]
\label{lem:safe-parameter}
At a current intermediate point $x\in\dom J\cap U$, suppose that
$F_i=\delta_{E_i}$, where $E_i$ is closed, and write
$g_i=\nabla_{x_i}h_i(x_i)$. For $c\geq2\delta$, let
$x_i^+\in P_{E_i}(x_i-g_i/c)$ and $d_i=x_i^+-x_i$.
Suppose that the actual block segment is contained in $U$ and satisfies
\begin{equation}
 \langle\nabla h_i(x_i+t d_i)-g_i,d_i\rangle
 \leq t(c-\delta)\|d_i\|^2,\qquad0\leq t\leq1.
 \label{eq:safe-segment-curvature}
\end{equation}
Then the update satisfies Eq.~(\ref{eq:descenttest}) and
\begin{equation}
 \langle g_i,d_i\rangle+\frac c2\|d_i\|^2\leq0,
 \qquad\|d_i\|\leq\frac{2\|g_i\|}{c}.
 \label{eq:safe-projection-displacement}
\end{equation}
In particular, suppose a radius $t_i>0$ gives a block ball contained in
$U$ and a bound $\ell_i(x,t_i)$ on the largest eigenvalue of
$\nabla^2h_i$ throughout that ball. It is sufficient to take
\begin{equation}
 c\geq\max\left\{2\delta,\frac{2\|g_i\|}{t_i},
                         \ell_i(x,t_i)+\delta\right\}.
 \label{eq:safe-radius-parameter}
\end{equation}
If the original block is a product of closed convex column sets
$E_{ir}$, its column projections instead satisfy
\begin{equation}
 \|(d_i)_r\|\leq\frac{\|(g_i)_r\|}{c}.
 \label{eq:safe-convex-column-displacement}
\end{equation}
Thus column balls of radii $t_{ir}$ may be certified by
$c\geq\max_r\|(g_i)_r\|/t_{ir}$ together with the curvature bound
on their product.
\end{lemma}
\begin{proof}
Comparison of the projected point with the feasible current block gives
\begin{align}
 \|d_i+g_i/c\|^2&\leq\|g_i/c\|^2,\notag\\
 \langle g_i,d_i\rangle+\frac c2\|d_i\|^2&\leq0.\notag
\end{align}
For $d_i\ne0$, Cauchy--Schwarz gives
$\frac c2\|d_i\|^2\leq\|g_i\|\|d_i\|$; for $d_i=0$ the
displacement bound holds directly. Integration along the actual segment
and Eq.~(\ref{eq:safe-segment-curvature}) give
\begin{align}
 h_i(x_i^+)-h_i(x_i)-\langle g_i,d_i\rangle
 &=\int_0^1\langle\nabla h_i(x_i+t d_i)-g_i,d_i\rangle\,dt\notag\\
 &\leq\frac{c-\delta}{2}\|d_i\|^2.
 \label{eq:safe-segment-majorization}
\end{align}
This is Eq.~(\ref{eq:descenttest}). Under
Eq.~(\ref{eq:safe-radius-parameter}), the displacement bound puts the
whole segment in the stated ball, and integration of the Hessian bound gives
\begin{align}
 \langle\nabla h_i(x_i+t d_i)-g_i,d_i\rangle
 &=\int_0^t\langle d_i,\nabla^2h_i(x_i+s d_i)d_i\rangle\,ds\notag\\
 &\leq t\ell_i(x,t_i)\|d_i\|^2
 \leq t(c-\delta)\|d_i\|^2.\notag
\end{align}
Finally, a convex column projection is nonexpansive and fixes the current
feasible column, so
\begin{align}
 \|(d_i)_r\|
 &=\|P_{E_{ir}}((x_i)_r-(g_i)_r/c)-P_{E_{ir}}((x_i)_r)\|
 \leq\|(g_i)_r\|/c.\notag
\end{align}
The same integration applies on the resulting product of column balls.
\end{proof}
Finite minima over admissible radius choices in
Eq.~(\ref{eq:safe-radius-parameter}) also give safe parameters. When the
candidate radii are positive and continuous and their curvature bounds are
locally bounded, these finite formulas are locally bounded, as required by
condition~(\ref{cond:c3}).

\subsection{Boundedness and step-size bounds}
We prove Property~\ref{prop:framework}(ii), including the common constants used in the subgradient estimate.
\begin{theorem}[Automatic boundedness and uniform parameters]
\label{thm:boundedness}
All intermediate block iterates lie in the compact set $K_M\subset U$. There are finite constants $\bar c$ and $L_K$ such that
\begin{equation}
\begin{aligned}
 2\delta\leq c_i^k\leq\bar c,\qquad
 &0<\frac{1}{\bar c}\leq\sigma_i^k\leq\frac{1}{2\delta},
 \quad\forall i\in[N],\ k\geq0,\\
 \|\nabla f(x)-\nabla f(y)\|\leq L_K\|x-y\|
 &\quad\forall x,y\in K_M.
\end{aligned}
 \label{eq:uniform}
\end{equation}
There also exist $M_x<\infty$ and $\eta_U>0$ such that
\begin{equation}
 \|x\|\leq M_x\quad(x\in K_M),\qquad
 K_M+\overline B(0,\eta_U)\subset U.
 \label{eq:uniform-control-neighborhood}
\end{equation}
\end{theorem}
\begin{proof}
(i) Theorem~\ref{thm:descent} puts every intermediate block iterate in $K_M$. Condition (\ref{cond:c2}) states that this nonempty finite sublevel is compact and contained in $U$. In particular, every generated block vector is bounded.

More explicitly, continuity of the norm on the compact sublevel gives
\begin{align}
 M_x&=\max_{x\in K_M}\|x\|<\infty,\notag\\
 \left\|(\{x_j^{k+1}\}_{j<i},\{x_j^k\}_{j\geq i})\right\|
 &\leq M_x,\qquad\forall i\in[N],\quad k\geq0.\notag
\end{align}
We next show that one positive neighborhood of the entire sublevel stays
inside the control domain. Otherwise, for every $m\geq1$, there would be
$z^m\in K_M$ and $e^m$ such that
\begin{equation*}
 \|e^m\|\leq\frac1m,\qquad z^m+e^m\notin U.
\end{equation*}
Compactness supplies a subsequence $z^{m_l}\to\bar z\in K_M$. Then
\begin{align}
 \|z^{m_l}+e^{m_l}-\bar z\|
 &\leq\|z^{m_l}-\bar z\|+\frac1{m_l}\longrightarrow0.\notag
\end{align}
Since $\bar z\in U$ and $U$ is open, these perturbed points must belong
to $U$ for all sufficiently large $l$, a contradiction. Thus some
$\eta_U>0$ satisfies Eq.~(\ref{eq:uniform-control-neighborhood}). The
same radius applies to every original block state, including all
intermediate states.

For each $z\in K_M$, the local boundedness in (\ref{cond:c3}) supplies a relative neighborhood $O_z$ and a finite bound $c_z$ for all coefficients selected there, uniformly in the iteration index. Since the number of blocks is finite, the bound can be taken common to their parameter rules. Compactness gives finitely many points $z_1,\ldots,z_q$ such that
\begin{align}
 K_M\subseteq\bigcup_{a=1}^{q}O_{z_a},\qquad
 \bar c=\max\{2\delta,c_{z_1},\ldots,c_{z_q}\}<\infty.\notag
\end{align}
Each actual coefficient is selected at a point in this cover. Hence
\begin{align}
 2\delta\leq c_i^k\leq\bar c,
 \qquad\forall i\in[N],\ k\geq0.\notag
\end{align}
Taking reciprocals of these positive numbers and using $\sigma_i^k=1/c_i^k$ proves the step-size bounds in Eq.~(\ref{eq:uniform}).

(ii) We next obtain a common gradient constant between all pairs of points of $K_M$. By local Lipschitz continuity of $\nabla f$, for each $z\in K_M$ choose $a_z>0$ such that $B(z,2a_z)\subset U$ and $\nabla f$ is $L_z$-Lipschitz on this ball. A finite subcover of $K_M$ by $B(z,a_z)$ gives centers $z_1,\ldots,z_s$. Put
\begin{align}
 \eta=\min_{1\leq a\leq s}a_{z_a}>0,\qquad
 L_0=\max_{1\leq a\leq s}L_{z_a}<\infty,\qquad
 G=\max_{z\in K_M}\|\nabla f(z)\|<\infty.
 \label{eq:uniform-gradient-constants}
\end{align}
The last maximum is finite by continuity on the compact set. Let $x,y\in K_M$ with $\|x-y\|<\eta$. Choose $a$ such that $x\in B(z_a,a_{z_a})$. The triangle inequality gives
\begin{align}
 \|y-z_a\|
 &\leq\|y-x\|+\|x-z_a\|
 <\eta+a_{z_a}\leq2a_{z_a}.\notag
\end{align}
Thus both points belong to the same local Lipschitz ball, and
\begin{align}
 \|\nabla f(x)-\nabla f(y)\|
 &\leq L_{z_a}\|x-y\|\leq L_0\|x-y\|.\notag
\end{align}
For $\|x-y\|\geq\eta$, boundedness of the gradient gives
\begin{align}
 \|\nabla f(x)-\nabla f(y)\|
 &\leq\|\nabla f(x)\|+\|\nabla f(y)\|\notag\\
 &\leq2G\leq\frac{2G}{\eta}\|x-y\|.\notag
\end{align}
Taking $L_K=\max\{L_0,2G/\eta\}$ proves the second part of Eq.~(\ref{eq:uniform}). The estimate compares points of $K_M$ directly, so it also applies to the different intermediate points occurring in a cyclic sweep.
\end{proof}

\subsection{Subgradient estimates and cluster points}
We prove Property~\ref{prop:framework}(iii). We first construct a regular subgradient from the original block minimizers and then establish the function-value compatibility needed at every cluster point.
\begin{theorem}[Subgradient bound]
\label{thm:subgradient}
For every $k\geq0$, define
\begin{equation}
 p_{x_i^{k+1}}
 =\nabla_{x_i}f(x_{(N)}^{k+1})-\nabla_{x_i}h_i(x_i^k)
       -c_i^k(x_i^{k+1}-x_i^k).
 \label{eq:certificate}
\end{equation}
Then $p_{x^{k+1}}=(p_{x_i^{k+1}})_{i=1}^{N}\in\widehat\partial\Phi(x_{(N)}^{k+1})\subseteq\partial\Phi(x_{(N)}^{k+1})$ and
\begin{equation}
 \|p_{x^{k+1}}\|\leq(\sqrt N L_K+\bar c)d_k.
 \label{eq:subgradientbound}
\end{equation}
\end{theorem}
\begin{proof}
(i) Global minimality in Eq.~(\ref{e12}) implies local minimality of the same subproblem at $x_i^{k+1}$. The regular Fermat rule and smooth addition in Proposition~\ref{prop:subdifferential-calculus} give
\begin{align}
 0&\in\widehat\partial F_i(x_i^{k+1})
       +\nabla_{x_i}h_i(x_i^k)+c_i^k\Delta_i^k,\notag\\
 v_i^{k+1}:=-\nabla_{x_i}h_i(x_i^k)-c_i^k\Delta_i^k
 &\in\widehat\partial F_i(x_i^{k+1}).
 \label{eq:block-regular-certificate}
\end{align}
Here $h_i$ is the block function of iteration $k$, as defined before Eq.~(\ref{e12}). Adding the smooth gradient at the complete new iterate yields
\begin{align}
 p_{x_i^{k+1}}
 &=\nabla_{x_i}f(x_{(N)}^{k+1})+v_i^{k+1}\notag\\
 &\in\nabla_{x_i}f(x_{(N)}^{k+1})
                  +\widehat\partial F_i(x_i^{k+1}).\notag
\end{align}
For the sum of functions on separate blocks, their regular subdifferentials form the Cartesian product. Indeed, testing a variation of one block gives one inclusion. Conversely, sum the regular subgradient inequalities for all blocks; their error terms are $o(\sum_i\|u_i-x_i^{k+1}\|)=o(\|u-x_{(N)}^{k+1}\|)$ because $N$ is finite. Therefore
\begin{align}
 p_{x^{k+1}}
 &\in\nabla f(x_{(N)}^{k+1})
          +\prod_{i=1}^{N}\widehat\partial F_i(x_i^{k+1})\notag\\
 &=\widehat\partial\Phi(x_{(N)}^{k+1})
 \subseteq\partial\Phi(x_{(N)}^{k+1}).
 \label{eq:regular-merit-certificate}
\end{align}

(ii) Both the complete new iterate and the point used to form $\nabla_{x_i}h_i(x_i^k)$ belong to $K_M$. Their product distance satisfies
\begin{align}
 \left\|x_{(N)}^{k+1}-
       (\{x_j^{k+1}\}_{j=1}^{i-1},\{x_j^k\}_{j=i}^{N})\right\|^2
 &=\sum_{j=i}^{N}\|\Delta_j^k\|^2\leq d_k^2.
 \label{eq:block-iterate-distance}
\end{align}
Applying Eq.~(\ref{eq:uniform}) to these two points, then retaining the $i$th gradient block, gives
\begin{align}
 \|\nabla_{x_i}f(x_{(N)}^{k+1})-\nabla_{x_i}h_i(x_i^k)\|
 &\leq L_K\left(\sum_{j=i}^{N}\|\Delta_j^k\|^2\right)^{1/2}
 \leq L_Kd_k.\notag
\end{align}
Square these bounds and sum over $i$. The resulting product norm is bounded by $\sqrt N L_Kd_k$. The coefficient terms satisfy
\begin{align}
 \left(\sum_i(c_i^k)^2\|\Delta_i^k\|^2\right)^{1/2}
 &\leq\bar c\left(\sum_i\|\Delta_i^k\|^2\right)^{1/2}
 =\bar c d_k.\notag
\end{align}
Finally, the triangle inequality in the product space and Eq.~(\ref{eq:certificate}) imply
\begin{align}
 \|p_{x^{k+1}}\|
 &\leq\left(\sum_i
  \|\nabla_{x_i}f(x_{(N)}^{k+1})-\nabla_{x_i}h_i(x_i^k)\|^2\right)^{1/2}
 +\left(\sum_i(c_i^k)^2\|\Delta_i^k\|^2\right)^{1/2}\notag\\
 &\leq(\sqrt N L_K+\bar c)d_k.\notag
\end{align}
This proves Eq.~(\ref{eq:subgradientbound}); Theorem~\ref{thm:descent} also gives $p_{x^{k+1}}\to0$.
\end{proof}

Let
\begin{align}
 \Omega=\{\bar x:\exists k_j\to\infty,\ x_{(N)}^{k_j}\to\bar x\}.\notag
\end{align}
\begin{lemma}[Compactness of the cluster set]
\label{lem:cluster-set}
The set $\Omega$ is nonempty and compact, and $\operatorname{dist}(x_{(N)}^k,\Omega)\to0$.
\end{lemma}
\begin{proof}
Theorem~\ref{thm:boundedness} places the sequence in the compact set $K_M$. It therefore has a convergent subsequence, and $\varnothing\ne\Omega\subseteq K_M$.

To prove closedness, let $u^m\in\Omega$ and $u^m\to u$. Since $u^m$ is a cluster point, choose increasing integers $k_m$ with
\begin{align}
 \|x_{(N)}^{k_m}-u^m\|<\frac1m.\notag
\end{align}
The triangle inequality gives
\begin{align}
 \|x_{(N)}^{k_m}-u\|
 &\leq\|x_{(N)}^{k_m}-u^m\|+\|u^m-u\|\notag\\
 &<\frac1m+\|u^m-u\|\longrightarrow0.\notag
\end{align}
Thus $u\in\Omega$. The set is closed, hence compact as a subset of $K_M$.

If the distance assertion were false, there would be $\varepsilon>0$ and indices $k_j\to\infty$ with $\operatorname{dist}(x_{(N)}^{k_j},\Omega)\geq\varepsilon$. A further subsequence converges in $K_M$, say $x_{(N)}^{k_{j_a}}\to\bar x$. By definition $\bar x\in\Omega$, so
\begin{align}
 0<\varepsilon
 &\leq\operatorname{dist}(x_{(N)}^{k_{j_a}},\Omega)
 \leq\|x_{(N)}^{k_{j_a}}-\bar x\|\longrightarrow0,\notag
\end{align}
a contradiction.
\end{proof}

\begin{theorem}[Controlled objective values and criticality at cluster points]
\label{thm:controlled-cluster}
Every cluster point satisfies
\begin{equation}
 \Phi(\bar x)=\ell,\qquad0\in\partial\Phi(\bar x),
 \qquad\forall\bar x\in\Omega.
 \label{eq:cluster}
\end{equation}
Along every subsequence converging to $\bar x$, the original block values converge to $F_i(\bar x_i)$, for every $i\in[N]$.
\end{theorem}
\begin{proof}
(i) Fix $\bar x\in\Omega$ and choose $x_{(N)}^{k_j+1}\to\bar x$. Such indices exist after discarding a finite prefix of a cluster subsequence. Theorem~\ref{thm:descent} gives $d_{k_j}\to0$, and hence
\begin{align}
 \|x_{(N)}^{k_j}-\bar x\|
 &\leq d_{k_j}+\|x_{(N)}^{k_j+1}-\bar x\|
 \longrightarrow0.\notag
\end{align}
Every corresponding intermediate point converges to $\bar x$, since its distance to $x_{(N)}^{k_j}$ is at most $d_{k_j}$. Lower semicontinuity gives
\begin{align}
 \Phi(\bar x)
 &\leq\liminf_{j\to\infty}\Phi(x_{(N)}^{k_j+1})
 =\ell<\infty.\notag
\end{align}
Thus $\bar x\in K_M\subset U$, and each $F_i(\bar x_i)$ is finite.

(ii) Compare Eq.~(\ref{e12}) at iteration $k_j$ with $u_i=\bar x_i$. Rearranging its linear terms gives
\begin{align}
 F_i(x_i^{k_j+1})
 &\leq F_i(\bar x_i)
 +\langle\nabla_{x_i}h_i(x_i^{k_j}),\bar x_i-x_i^{k_j+1}\rangle\notag\\
 &\quad+\frac{c_i^{k_j}}2
       \bigl(\|\bar x_i-x_i^{k_j}\|^2-\|\Delta_i^{k_j}\|^2\bigr).
 \label{eq:cluster-block-comparison}
\end{align}
Use $G$ from Eq.~(\ref{eq:uniform-gradient-constants}) and the uniform coefficient bound. Cauchy--Schwarz and the triangle inequality yield
\begin{align}
 \big|\langle\nabla_{x_i}h_i(x_i^{k_j}),\bar x_i-x_i^{k_j+1}\rangle\big|
 &\leq G\|\bar x_i-x_i^{k_j+1}\|\longrightarrow0,\notag\\
 \frac{c_i^{k_j}}2
 \big|\|\bar x_i-x_i^{k_j}\|^2-\|\Delta_i^{k_j}\|^2\big|
 &\leq\frac{\bar c}{2}
       \bigl(\|\bar x_i-x_i^{k_j}\|^2+\|\Delta_i^{k_j}\|^2\bigr)
 \longrightarrow0.\notag
\end{align}
Taking the upper limit in Eq.~(\ref{eq:cluster-block-comparison}) therefore gives
\begin{align}
 \limsup_{j\to\infty}F_i(x_i^{k_j+1})
 &\leq F_i(\bar x_i)
 +\limsup_{j\to\infty}\left[
 G\|\bar x_i-x_i^{k_j+1}\|
 +\frac{\bar c}{2}
 (\|\bar x_i-x_i^{k_j}\|^2+\|\Delta_i^{k_j}\|^2)\right]\notag\\
 &=F_i(\bar x_i).\notag
\end{align}
Lower semicontinuity supplies the reverse lower-limit inequality, so
\begin{align}
 F_i(\bar x_i)
 &\leq\liminf_{j\to\infty}F_i(x_i^{k_j+1})
 \leq\limsup_{j\to\infty}F_i(x_i^{k_j+1})
 \leq F_i(\bar x_i).\notag
\end{align}
Consequently $F_i(x_i^{k_j+1})\to F_i(\bar x_i)$ for every $i$.

(iii) The finite number of block limits and continuity of $f$ on $U$ imply
\begin{align}
 \ell
 &=\lim_{j\to\infty}\Phi(x_{(N)}^{k_j+1})\notag\\
 &=\lim_{j\to\infty}f(x_{(N)}^{k_j+1})
       +\sum_{i=1}^{N}\lim_{j\to\infty}F_i(x_i^{k_j+1})\notag\\
 &=f(\bar x)+\sum_{i=1}^{N}F_i(\bar x_i)=\Phi(\bar x).\notag
\end{align}
By Eqs.~(\ref{eq:regular-merit-certificate}) and (\ref{eq:subgradientbound}), $p_{x^{k_j+1}}\in\partial\Phi(x_{(N)}^{k_j+1})$ and $p_{x^{k_j+1}}\to0$. The value-compatible closedness in Proposition~\ref{prop:subdifferential-calculus} now gives $0\in\partial\Phi(\bar x)$. The same comparison applies to any cluster subsequence after shifting its index by one. This proves all assertions.
\end{proof}

\subsection{Finite length and global convergence}
We prove Property~\ref{prop:framework}(iv) and then derive summability of the controlled certificates and objective gaps.
\begin{theorem}[Finite length of the original variables]
\label{thm:controlled-global}
The whole sequence converges to a finite $x^*\in K_M$, with $0\in\partial\Phi(x^*)$, and
\begin{equation}
 \sum_{k=0}^{\infty}\|x_{(N)}^{k+1}-x_{(N)}^k\|<\infty.
 \label{eq:controlled-length}
\end{equation}
The path containing every intermediate block iterate also has finite total length. Put
\begin{equation}
 \rho=\frac{\delta}{2},\qquad \rho_b=\sqrt N L_K+\bar c,\qquad
 r_k=\Phi(x_{(N)}^k)-\ell,\qquad
 S_m=\sum_{k=m}^{\infty}d_k.
 \label{eq:length-tail-notation}
\end{equation}
For all sufficiently large $m\geq1$,
\begin{equation}
 \|x_{(N)}^m-x^*\|\leq S_m
 \leq d_{m-1}+\frac{\rho_b}{\rho}\varphi(r_m),
 \label{eq:kl-length-tail}
\end{equation}
where $\varphi$ is the uniform desingularizing function near $\Omega$.
\end{theorem}
\begin{proof}
(i) Lemma~\ref{lem:cluster-set} gives a nonempty compact $\Omega$ and $\operatorname{dist}(x_{(N)}^k,\Omega)\to0$. Theorem~\ref{thm:controlled-cluster} shows that $\Omega\subseteq\dom\partial\Phi$ and $\Phi=\ell$ on $\Omega$. Proposition~\ref{prop:uniformKL} therefore gives $\varepsilon,\eta>0$ and one concave $\varphi$ satisfying the uniform KL inequality.

Since the distances tend to zero and $r_k\downarrow0$, choose $k_0\geq1$ such that
\begin{align}
 \operatorname{dist}(x_{(N)}^k,\Omega)<\varepsilon,
 \qquad0\leq r_k<\eta,\qquad\forall k\geq k_0.\notag
\end{align}
If $r_k=0$, Eq.~(\ref{eq:descent}) forces $d_k=0$ and then forces every subsequent displacement to vanish. For $r_k>0$, the uniform KL inequality and Eq.~(\ref{eq:subgradientbound}) give
\begin{align}
 1
 &\leq\varphi'(r_k)\operatorname{dist}(0,\partial\Phi(x_{(N)}^k))\notag\\
 &\leq\varphi'(r_k)\|p_{x^k}\|
 \leq\rho_b\varphi'(r_k)d_{k-1}.
 \label{eq:kl-certificate-chain}
\end{align}
In particular $d_{k-1}>0$, and $\varphi'(r_k)\geq\frac{1}{\rho_b d_{k-1}}$. Concavity, followed by Eq.~(\ref{eq:descent}), yields
\begin{align}
 \varphi(r_k)-\varphi(r_{k+1})
 &\geq\varphi'(r_k)(r_k-r_{k+1})\notag\\
 &\geq\frac{r_k-r_{k+1}}{\rho_b d_{k-1}}
 \geq\frac{\rho d_k^2}{\rho_b d_{k-1}}.
 \label{eq:kl-concavity-chain}
\end{align}
Rearrange this inequality, take square roots, and apply $2\sqrt{uv}\leq u+v$ with $u=d_{k-1}$ and $v=(\rho_b/\rho)(\varphi(r_k)-\varphi(r_{k+1}))$. The result is
\begin{align}
 2d_k
 &\leq d_{k-1}+\frac{\rho_b}{\rho}
              (\varphi(r_k)-\varphi(r_{k+1})).
 \label{eq:kl-length-recurrence}
\end{align}
When $r_k=0$, its left side is zero and the same inequality holds directly.

(ii) Fix integers $n\geq m\geq k_0$. Summing Eq.~(\ref{eq:kl-length-recurrence}) gives
\begin{align}
 2\sum_{k=m}^{n}d_k
 &\leq\sum_{k=m}^{n}d_{k-1}
       +\frac{\rho_b}{\rho}(\varphi(r_m)-\varphi(r_{n+1}))\notag\\
 &=d_{m-1}+\sum_{k=m}^{n-1}d_k
       +\frac{\rho_b}{\rho}(\varphi(r_m)-\varphi(r_{n+1})).\notag
\end{align}
Subtract $\sum_{k=m}^{n-1}d_k$ and use $\varphi\geq0$. Then
\begin{align}
 \sum_{k=m}^{n}d_k+d_n
 &\leq d_{m-1}+\frac{\rho_b}{\rho}\varphi(r_m),
 \qquad n\geq m\geq k_0.
 \label{eq:kl-finite-tail-sum}
\end{align}
Thus the nonnegative partial sums of $\sum_{k=m}^{\infty}d_k$ are bounded. Adding the finite prefix proves Eq.~(\ref{eq:controlled-length}). Letting $n\to\infty$ in Eq.~(\ref{eq:kl-finite-tail-sum}) gives the second inequality in Eq.~(\ref{eq:kl-length-tail}).

(iii) For every integer $p\geq1$, the triangle inequality gives
\begin{align}
 \|x_{(N)}^{k+p}-x_{(N)}^k\|
 &\leq\sum_{j=k}^{k+p-1}d_j\leq S_k\longrightarrow0.\notag
\end{align}
The bound is independent of $p$, so the sequence is Cauchy. Completeness gives a finite limit $x^*$, and closedness of $K_M$ puts $x^*$ in $K_M$. Taking $p\to\infty$ in the same inequality gives $\|x_{(N)}^k-x^*\|\leq S_k$. Theorem~\ref{thm:controlled-cluster} gives $\Phi(x^*)=\ell$ and $0\in\partial\Phi(x^*)$.

Finally, Cauchy--Schwarz over the $N$ block displacements gives
\begin{align}
 \sum_{k=0}^{\infty}\sum_{i=1}^{N}\|\Delta_i^k\|
 &\leq\sqrt N\sum_{k=0}^{\infty}
            \left(\sum_i\|\Delta_i^k\|^2\right)^{1/2}
 =\sqrt N\sum_{k=0}^{\infty}d_k<\infty.\notag
\end{align}
This is the length of the path containing all intermediate block iterates.
\end{proof}

\begin{theorem}[Summable controlled certificates and objective gaps]
\label{thm:controlled-summability}
With the notation of Eq.~(\ref{eq:length-tail-notation}),
\begin{equation}
 \sum_{k=1}^{\infty}\|p_{x^k}\|<\infty,
 \qquad\sum_{k=0}^{\infty}r_k<\infty.
 \label{eq:controlled-summability}
\end{equation}
For all sufficiently large $m\geq1$,
\begin{equation}
 \sum_{k=m}^{\infty}\|p_{x^k}\|
 \leq\rho_b(d_{m-1}+S_m),\qquad
 \sum_{k=m}^{\infty}r_k
 \leq\rho_b\varphi(r_m)(d_{m-1}+S_m).
 \label{eq:controlled-summability-tails}
\end{equation}
\end{theorem}
\begin{proof}
(i) Sum Eq.~(\ref{eq:subgradientbound}) with its index shifted by one. For every $n\geq m\geq1$,
\begin{align}
 \sum_{k=m}^{n}\|p_{x^k}\|
 &\leq\rho_b\sum_{k=m}^{n}d_{k-1}\notag\\
 &\leq\rho_b\left(d_{m-1}+\sum_{k=m}^{\infty}d_k\right)
 =\rho_b(d_{m-1}+S_m).\notag
\end{align}
The finite-length result makes the right side finite. Letting $n\to\infty$ proves certificate summability and its tail bound.

(ii) Concavity of $\varphi$, evaluated at $r_k$ and zero, gives
\begin{align}
 0=\varphi(0)
 &\leq\varphi(r_k)-r_k\varphi'(r_k).\notag
\end{align}
For $r_k>0$ in the uniform KL region, Eq.~(\ref{eq:kl-certificate-chain}) therefore implies
\begin{align}
 r_k
 &\leq\frac{\varphi(r_k)}{\varphi'(r_k)}
 \leq\varphi(r_k)\|p_{x^k}\|
 \leq\rho_b\varphi(r_k)d_{k-1}.
 \label{eq:gap-certificate-bound}
\end{align}
This also holds when $r_k=0$, since both sides then vanish. For $k\geq m$ the controlled values decrease, so $r_k\leq r_m$; positivity of $\varphi'$ gives $\varphi(r_k)\leq\varphi(r_m)$. Sum Eq.~(\ref{eq:gap-certificate-bound}) to obtain
\begin{align}
 \sum_{k=m}^{\infty}r_k
 &\leq\rho_b\varphi(r_m)\sum_{k=m}^{\infty}d_{k-1}
 =\rho_b\varphi(r_m)(d_{m-1}+S_m).\notag
\end{align}
Adding the finite prefix proves $\sum_kr_k<\infty$.

\end{proof}

\begin{proposition}[Certificates at intermediate block iterates]
\label{prop:intermediate-certificates}
For $k\geq1$ and $i\in[N]$, the state
\begin{align}
 x=(\{x_j^{k+1}\}_{j=1}^{i},\{x_j^k\}_{j=i+1}^{N})\notag
\end{align}
admits a regular certificate $w^{k,i}\in\widehat\partial\Phi(x)$ with
\begin{equation}
 \|w^{k,i}\|\leq\rho_b\sqrt{d_k^2+d_{k-1}^2},\qquad
 \sum_{k=1}^{\infty}\sum_{i=1}^{N}\|w^{k,i}\|<\infty.
 \label{eq:intermediate-certificate-sum}
\end{equation}
\end{proposition}
\begin{proof}
Use the block regular subgradients $v_j^{k+1}$ from Eq.~(\ref{eq:block-regular-certificate}). At the stated point, block $j$ has already been updated in sweep $k$ if $j\leq i$; otherwise its most recent update was in sweep $k-1$. Thus
\begin{align}
 \widetilde v_j=
 \begin{cases}
 v_j^{k+1},&j\leq i,\\
 v_j^k,&j>i
 \end{cases}
 \quad\in\widehat\partial F_j(x_j).\notag
\end{align}
All these subgradients exist because $k\geq1$. Set $w^{k,i}=\nabla f(x)+(\widetilde v_j)_{j=1}^{N}$. The same regular product rule as in Eq.~(\ref{eq:regular-merit-certificate}) gives $w^{k,i}\in\widehat\partial\Phi(x)$.

For $j\leq i$, the gradient used to construct $v_j^{k+1}$ was evaluated at
\begin{equation*}
 (\{x_l^{k+1}\}_{l<j},\{x_l^k\}_{l\geq j}).
\end{equation*}
Its squared distance to $x$ equals
$\sum_{l=j}^{i}\|\Delta_l^k\|^2\leq d_k^2$.
For $j>i$, the corresponding point in the previous sweep was
\begin{equation*}
 (\{x_l^k\}_{l<j},\{x_l^{k-1}\}_{l\geq j}),
\end{equation*}
whose squared distance to $x$ equals
\begin{align}
 \sum_{l=1}^{i}\|\Delta_l^k\|^2
 +\sum_{l=j}^{N}\|\Delta_l^{k-1}\|^2
 &\leq d_k^2+d_{k-1}^2.\notag
\end{align}
All comparison points are in $K_M$, so Eq.~(\ref{eq:uniform}) bounds the product norm of the gradient differences by $\sqrt N L_K\sqrt{d_k^2+d_{k-1}^2}$. The selected coefficient terms have product norm at most
\begin{align}
 \bar c\left(\sum_{j\leq i}\|\Delta_j^k\|^2
             +\sum_{j>i}\|\Delta_j^{k-1}\|^2\right)^{1/2}
 &\leq\bar c\sqrt{d_k^2+d_{k-1}^2}.\notag
\end{align}
The product triangle inequality proves the first bound in Eq.~(\ref{eq:intermediate-certificate-sum}). For the second, sum it and use $\sqrt{d_k^2+d_{k-1}^2}\leq d_k+d_{k-1}$:
\begin{align}
 \sum_{k=1}^{\infty}\sum_{i=1}^{N}\|w^{k,i}\|
 &\leq N\rho_b\sum_{k=1}^{\infty}(d_k+d_{k-1})
 \leq2N\rho_b\sum_{k=0}^{\infty}d_k<\infty.\notag
\end{align}
\end{proof}

\subsection{Convergence and total variation of the original function values}
We prove Property~\ref{prop:framework}(v) together with finite total variation of the original objective and its block terms. These properties describe the original values along the complete update path.
\begin{theorem}[Finite total variation and convergence of the original objective]
\label{thm:original-values}
\label{prop:original-positive-variation}
The original objective values satisfy
\begin{equation}
 \lim_{k\to\infty}J(x_{(N)}^k)=J(x^*)=\ell-C(x^*),
 \label{eq:original-value-limit}
\end{equation}
and
\begin{equation}
 \sum_{k=0}^{\infty}|J(x_{(N)}^{k+1})-J(x_{(N)}^k)|<\infty,
 \qquad
 \sum_{k=0}^{\infty}|C(x_{(N)}^{k+1})-C(x_{(N)}^k)|<\infty.
 \label{eq:original-control-variation}
\end{equation}
Both total variations remain finite when all intermediate block updates are included. There are $L_C<\infty$ and $k_C\geq1$ such that, for every $m\geq k_C$,
\begin{equation}
\begin{aligned}
 \sum_{k=m}^{\infty}|J(x_{(N)}^{k+1})-J(x_{(N)}^k)|
 &\leq r_m+L_CS_m,\\
 \sup_{n,p\geq m}|J(x_{(N)}^n)-J(x_{(N)}^p)|
 &\leq r_m+L_CS_m,\\
 |J(x_{(N)}^m)-J(x^*)|
 &\leq r_m+L_C\|x_{(N)}^m-x^*\|\leq r_m+L_CS_m.
\end{aligned}
 \label{eq:original-variation-tail}
\end{equation}
Write $[t]_+=\max\{t,0\}$. At every actual block update or complete sweep from $x$ to $x^+$,
\begin{equation}
\begin{aligned}
 [J(x^+)-J(x)]_+&\leq[C(x)-C(x^+)]_+,\\
 [J(x)-J(x^+)]_+&\leq\Phi(x)-\Phi(x^+)+[C(x^+)-C(x)]_+.
\end{aligned}
 \label{eq:original-positive-negative-change}
\end{equation}
The positive and negative variations are finite. In particular, for $m\geq k_C$,
\begin{align}
 \sum_{k=m}^{\infty}[J(x_{(N)}^{k+1})-J(x_{(N)}^k)]_+
 &\leq\sum_{k=m}^{\infty}[C(x_{(N)}^k)-C(x_{(N)}^{k+1})]_+
 \leq L_CS_m.
 \label{eq:original-upward-tail}
\end{align}
\end{theorem}
\begin{proof}
(i) Theorem~\ref{thm:controlled-global} gives $x_{(N)}^k\to x^*\in K_M\subset U$. Choose $a>0$ such that $\overline B(x^*,2a)\subset U$. The intermediate points satisfy
\begin{align}
 \| (\{x_j^{k+1}\}_{j=1}^{i},\{x_j^k\}_{j=i+1}^{N})-x^*\|
 &\leq\|x_{(N)}^k-x^*\|+d_k\longrightarrow0,\notag
\end{align}
uniformly over the finite set $i\in\{0,\ldots,N\}$. Hence choose $k_C$ so all these points for $k\geq k_C$ lie in $B(x^*,a)$. Continuity of $\nabla C$ on the compact ball gives
\begin{align}
 L_C=\max_{z\in\overline B(x^*,a)}\|\nabla C(z)\|<\infty.\notag
\end{align}
The ball is convex and contained in $U$. For any two of its points, the fundamental theorem of calculus and Cauchy--Schwarz give
\begin{align}
 |C(y)-C(x)|
 &=\left|\int_0^1\langle\nabla C(x+t(y-x)),y-x\rangle\,dt\right|\notag\\
 &\leq\int_0^1\|\nabla C(x+t(y-x))\|\|y-x\|\,dt
 \leq L_C\|y-x\|.
 \label{eq:control-tail-lipschitz}
\end{align}

(ii) Use $J=\Phi-C$ and controlled descent. For $k\geq k_C$,
\begin{align}
 |J(x_{(N)}^{k+1})-J(x_{(N)}^k)|
 &\leq|\Phi(x_{(N)}^{k+1})-\Phi(x_{(N)}^k)|
       +|C(x_{(N)}^{k+1})-C(x_{(N)}^k)|\notag\\
 &=\Phi(x_{(N)}^k)-\Phi(x_{(N)}^{k+1})
       +|C(x_{(N)}^{k+1})-C(x_{(N)}^k)|\notag\\
 &\leq\Phi(x_{(N)}^k)-\Phi(x_{(N)}^{k+1})+L_Cd_k.
 \label{eq:original-increment-bound}
\end{align}
The equality uses monotonicity of $\Phi$. For integers $n>m\geq k_C$, summing the last inequality yields
\begin{align}
 \sum_{k=m}^{n-1}|J(x_{(N)}^{k+1})-J(x_{(N)}^k)|
 &\leq\sum_{k=m}^{n-1}
       (\Phi(x_{(N)}^k)-\Phi(x_{(N)}^{k+1}))
       +L_C\sum_{k=m}^{n-1}d_k\notag\\
 &=\Phi(x_{(N)}^m)-\Phi(x_{(N)}^n)
       +L_C\sum_{k=m}^{n-1}d_k.\notag
\end{align}
Let $n\to\infty$ and use $\Phi(x_{(N)}^n)\to\ell$. We obtain
\begin{align}
 \sum_{k=m}^{\infty}|J(x_{(N)}^{k+1})-J(x_{(N)}^k)|
 &\leq r_m+L_CS_m,\notag\\
 \sum_{k=m}^{\infty}|C(x_{(N)}^{k+1})-C(x_{(N)}^k)|
 &\leq L_CS_m.
 \label{eq:original-control-tail-sums}
\end{align}
The finitely many earlier original values and control values are finite. Adding their finite variations proves Eq.~(\ref{eq:original-control-variation}).

(iii) For any $n>p\geq m$, finite telescoping and the triangle inequality give
\begin{align}
 |J(x_{(N)}^n)-J(x_{(N)}^p)|
 &\leq\sum_{k=p}^{n-1}|J(x_{(N)}^{k+1})-J(x_{(N)}^k)|
 \leq r_m+L_CS_m.\notag
\end{align}
The right side tends to zero as $m\to\infty$, so the original values are Cauchy and have a finite limit. Theorems~\ref{thm:controlled-cluster} and \ref{thm:controlled-global}, followed by continuity of $C$, identify it:
\begin{align}
 \lim_{k\to\infty}J(x_{(N)}^k)
 &=\lim_{k\to\infty}(\Phi(x_{(N)}^k)-C(x_{(N)}^k))\notag\\
 &=\ell-C(x^*)=\Phi(x^*)-C(x^*)=J(x^*).\notag
\end{align}
Moreover,
\begin{align}
 |J(x_{(N)}^m)-J(x^*)|
 &=|r_m-C(x_{(N)}^m)+C(x^*)|\notag\\
 &\leq r_m+L_C\|x_{(N)}^m-x^*\|
 \leq r_m+L_CS_m,\notag
\end{align}
by Eq.~(\ref{eq:control-tail-lipschitz}) and Theorem~\ref{thm:controlled-global}. This proves Eq.~(\ref{eq:original-variation-tail}). Substituting Eq.~(\ref{eq:kl-length-tail}) also gives, for all sufficiently large $m$,
\begin{align}
 \sum_{k=m}^{\infty}|J(x_{(N)}^{k+1})-J(x_{(N)}^k)|
 &\leq r_m+L_C\left(d_{m-1}+\frac{\rho_b}{\rho}\varphi(r_m)\right).
 \label{eq:original-variation-kl-tail}
\end{align}

(iv) The same argument applies to each block update because Eq.~(\ref{eq:intermediate-controlled-decrease}) gives controlled descent there. Its control variation is at most $L_C\|\Delta_i^k\|$. Summing first over blocks and then over sweeps gives
\begin{align}
&\sum_{k=m}^{\infty}\sum_{i=1}^{N}
 \Big|J(\{x_j^{k+1}\}_{j=1}^{i},\{x_j^k\}_{j=i+1}^{N})
 -J(\{x_j^{k+1}\}_{j=1}^{i-1},\{x_j^k\}_{j=i}^{N})\Big|\notag\\
&\qquad\leq r_m+L_C\sum_{k=m}^{\infty}\sum_i\|\Delta_i^k\|
 \leq r_m+L_C\sqrt N S_m.
 \label{eq:intermediate-original-variation}
\end{align}
The corresponding control variation is at most $L_C\sum_{k=m}^{\infty}\sum_i\|\Delta_i^k\|$. Both series are finite by Eq.~(\ref{eq:controlled-length}). This also proves the full-path assertions.

(v) Put $a=\Phi(x)-\Phi(x^+)\geq0$ and $b=C(x^+)-C(x)$. The identity $J=\Phi-C$ gives
\begin{align}
 J(x^+)-J(x)=-a-b.\notag
\end{align}
Since $a\geq0$, $[-a-b]_+\leq[-b]_+$ and $[a+b]_+\leq a+[b]_+$. These are exactly Eq.~(\ref{eq:original-positive-negative-change}). Summing its first line and using Eq.~(\ref{eq:original-control-tail-sums}) proves Eq.~(\ref{eq:original-upward-tail}). Summing its second line gives a bound by $r_m$ plus the positive variation of $C$. The bounds in Eq.~(\ref{eq:original-control-tail-sums}) make both variations finite. The same summation over the original block steps proves the assertions on the complete update path.
\end{proof}

\begin{theorem}[Finite total variation and value compatibility of the original block terms]
\label{thm:block-value-variation}
For every $i\in[N]$,
\begin{equation}
 \sum_{k=0}^{\infty}|F_i(x_i^{k+1})-F_i(x_i^k)|<\infty,
 \qquad F_i(x_i^k)\longrightarrow F_i(x_i^*).
 \label{eq:block-value-variation}
\end{equation}
These statements also hold along the path containing every intermediate block update.
\end{theorem}
\begin{proof}
(i) From Eq.~(\ref{eq:block-prox-comparison}) and the gradient bound $G$ in Eq.~(\ref{eq:uniform-gradient-constants}),
\begin{align}
 F_i(x_i^{k+1})-F_i(x_i^k)
 &\leq-\langle\nabla_{x_i}h_i(x_i^k),\Delta_i^k\rangle
       -\frac{c_i^k}{2}\|\Delta_i^k\|^2\notag\\
 &\leq G\|\Delta_i^k\|.
 \label{eq:block-value-positive-increment}
\end{align}
Thus, for all integers $n>m\geq0$,
\begin{align}
 \sum_{k=m}^{n-1}[F_i(x_i^{k+1})-F_i(x_i^k)]_+
 &\leq G\sum_{k=m}^{n-1}\|\Delta_i^k\|.\notag
\end{align}
The right side remains bounded as $n\to\infty$ by Eq.~(\ref{eq:controlled-length}). Hence the positive variation is finite.

The projection $K_i=\{x_i:x\in K_M\}$ is nonempty and compact. The proper lower semicontinuous $F_i$ has a finite lower bound $m_i$ on $K_i$: otherwise a sequence in $K_i$ with values tending to $-\infty$ would have a convergent subsequence, contradicting lower semicontinuity and the fact that $F_i$ never equals $-\infty$. In particular $F_i(x_i^n)\geq m_i$. For the negative variation use $[-t]_+=[t]_+-t$ and telescope:
\begin{align}
 \sum_{k=m}^{n-1}[F_i(x_i^k)-F_i(x_i^{k+1})]_+
 &=\sum_{k=m}^{n-1}[F_i(x_i^{k+1})-F_i(x_i^k)]_+
       +F_i(x_i^m)-F_i(x_i^n)\notag\\
 &\leq G\sum_{k=m}^{n-1}\|\Delta_i^k\|+F_i(x_i^m)-m_i.\notag
\end{align}
Letting $n\to\infty$ proves that the negative variation is finite too. The sum of the positive and negative variations is the absolute variation in Eq.~(\ref{eq:block-value-variation}). It also follows that $F_i(x_i^k)$ has a finite limit $\ell_i$.

(ii) Theorem~\ref{thm:controlled-global} gives $x_{(N)}^k\to x^*$. The block-value compatibility in Theorem~\ref{thm:controlled-cluster} therefore gives $F_i(x_i^k)\to F_i(x_i^*)$. During a cyclic sweep, the $i$th original term changes only at its own block update. Its values at other steps are unchanged. Including all intermediate points therefore leaves its total variation unchanged, and their values are either $F_i(x_i^k)$ or $F_i(x_i^{k+1})$, both converging to $F_i(x_i^*)$.
\end{proof}

\subsection{Effective and full criticality of the original objective}
We prove Property~\ref{prop:framework}(vi). The controlled certificates have summable norms. The comparison geometry transfers this summability to the original effective residual, and scale orthogonality transfers it to the full original residual.
\begin{theorem}[Effective criticality of the original objective]
\label{thm:original}
For $k\geq0$,
\begin{equation}
 p^J_{x^{k+1}}=p_{x^{k+1}}-\nabla C(x_{(N)}^{k+1})
 \in\partial J(x_{(N)}^{k+1}).
 \label{eq:originalcertificate}
\end{equation}
These original limiting subgradients satisfy
\begin{equation}
 r_{\mathrm{eff}}(x_{(N)}^{k+1})
 \leq\operatorname{dist}(\partial J(x_{(N)}^{k+1}),W(x_{(N)}^{k+1}))
 \leq\|p_{x^{k+1}}\|\longrightarrow0.
 \label{eq:controlled-transfer}
\end{equation}
Both residuals on the left are summable over $k\geq0$. At the finite whole-sequence limit,
\begin{equation}
 -\nabla C(x^*)\in\partial J(x^*)\cap W(x^*).
 \label{eq:controlled-original-limit}
\end{equation}
This single original subgradient annihilates every effective comparison direction at $x^*$. The best-iterate residual rates are given in Corollary~\ref{cor:original-residual-rates}.
\end{theorem}
\begin{proof}
(i) The point $x_{(N)}^{k+1}$ belongs to $U$. Smooth addition in Eq.~(\ref{eq:smoothsum}) therefore gives
\begin{align}
 \partial\Phi(x_{(N)}^{k+1})
 &=\partial J(x_{(N)}^{k+1})+\nabla C(x_{(N)}^{k+1}).\notag
\end{align}
Subtracting the smooth gradient from the certificate in Eq.~(\ref{eq:regular-merit-certificate}) proves Eq.~(\ref{eq:originalcertificate}).

Fix this point and write it as $x$. For arbitrary $p\in\partial J(x)$, $z\in W(x)$ and a unit $d\in\Qeff(x)$, condition (\ref{cond:c4}) gives $\langle z,d\rangle=0$. Hence Cauchy--Schwarz yields
\begin{align}
 |\langle p,d\rangle|
 &=|\langle p-z,d\rangle|
 \leq\|p-z\|\|d\|=\|p-z\|.\notag
\end{align}
The bound is independent of $d$, so take its supremum and then the infimum over $z$ and $p$:
\begin{align}
 r_{\mathrm{eff}}(x)
 &=\inf_{p\in\partial J(x)}
       \sup_{\substack{d\in\Qeff(x)\\\|d\|=1}}|\langle p,d\rangle|\notag\\
 &\leq\inf_{p\in\partial J(x)}\inf_{z\in W(x)}\|p-z\|
 =\operatorname{dist}(\partial J(x),W(x)).
 \label{eq:effective-infimum-chain}
\end{align}
The same inequality holds if there is no unit comparison direction, by the convention in Eq.~(\ref{eq:residual}). For any $w\in\partial\Phi(x)$, smooth addition gives
$p=w-\nabla C(x)\in\partial J(x)$, and condition~(\ref{cond:c4})
gives $z=-\nabla C(x)\in W(x)$. Hence
\begin{equation}
 r_{\mathrm{eff}}(x)\leq\operatorname{dist}(\partial J(x),W(x))
 \leq\|w\|,\qquad w\in\partial\Phi(x).
 \label{eq:transfer}
\end{equation}
Choose $w=p_{x^{k+1}}$ to obtain the constructed-certificate estimate.
This proves Eq.~(\ref{eq:controlled-transfer}). Summing it and applying Theorem~\ref{thm:controlled-summability} proves summability of both original residuals.

(ii) Theorem~\ref{thm:controlled-global} gives $0\in\partial\Phi(x^*)$. Smooth addition at this same point gives
\begin{align}
 0\in\partial J(x^*)+\nabla C(x^*)
 \quad\Longrightarrow\quad-\nabla C(x^*)\in\partial J(x^*).\notag
\end{align}
Condition (\ref{cond:c4}) holds at the finite limit, so $-\nabla C(x^*)\in W(x^*)$. Consequently
\begin{align}
 \langle-\nabla C(x^*),d\rangle=0,
 \qquad\forall d\in\Qeff(x^*)\subseteq\Heff(x^*).\notag
\end{align}
This proves Eq.~(\ref{eq:controlled-original-limit}) with one original subgradient for all effective comparison directions.

(iii) For intermediate block points, smooth addition applied to $w^{k,i}$ in Proposition~\ref{prop:intermediate-certificates} gives the identical comparison with $\|w^{k,i}\|$. Their residuals are therefore summable along all updates after the first complete sweep as well.
\end{proof}

\begin{theorem}[Full criticality of the original objective]
\label{thm:full}
Suppose the original objective has a scale space $\mathcal S(x)$ satisfying
\begin{equation}
 \partial J(x)\subseteq\mathcal S(x)^\perp,\qquad
 \nabla C(x)\in\mathcal S(x),\qquad\forall x\in K_M.
 \label{eq:full-scale-orthogonality}
\end{equation}
Then
\begin{equation}
 \operatorname{dist}(0,\partial J(x_{(N)}^{k+1}))
 \leq\|p^J_{x^{k+1}}\|\leq\|p_{x^{k+1}}\|\longrightarrow0,
 \qquad0\in\partial J(x^*),\qquad\nabla C(x^*)=0.
 \label{eq:full-original-criticality}
\end{equation}
Moreover,
\begin{equation}
 \sum_{k=1}^{\infty}\|p^J_{x^k}\|<\infty,\qquad
 \sum_{k=1}^{\infty}\operatorname{dist}(0,\partial J(x_{(N)}^k))<\infty,
 \qquad\sum_{k=1}^{\infty}\|\nabla C(x_{(N)}^k)\|<\infty.
 \label{eq:full-residual-summability}
\end{equation}
Their best-iterate rates are given in Corollary~\ref{cor:original-residual-rates}. There are $L<\infty$ and $k_L$ such that, for every $m\geq k_L$,
\begin{equation}
\begin{aligned}
 \sum_{k=m}^{\infty}|C(x_{(N)}^{k+1})-C(x_{(N)}^k)|
 &\leq\frac L2 S_m^2,\\
 \sum_{k=m}^{\infty}|J(x_{(N)}^{k+1})-J(x_{(N)}^k)|
 &\leq r_m+\frac L2 S_m^2,\\
 |J(x_{(N)}^m)-J(x^*)|
 &\leq r_m+\frac L2\|x_{(N)}^m-x^*\|^2
 \leq r_m+\frac L2S_m^2.
\end{aligned}
 \label{eq:full-quadratic-variation-tail}
\end{equation}
\end{theorem}
\begin{proof}
(i) At $x=x_{(N)}^{k+1}$, Eq.~(\ref{eq:originalcertificate}) and the scale inclusions give
\begin{align}
 p^J_{x^{k+1}}\in\mathcal S(x)^\perp,
 \qquad\nabla C(x)\in\mathcal S(x),
 \qquad\langle p^J_{x^{k+1}},\nabla C(x)\rangle=0.\notag
\end{align}
Expanding the squared norm of $p_{x^{k+1}}=p^J_{x^{k+1}}+\nabla C(x)$ therefore gives
\begin{align}
 \|p_{x^{k+1}}\|^2
 &=\|p^J_{x^{k+1}}\|^2
   +2\langle p^J_{x^{k+1}},\nabla C(x)\rangle+\|\nabla C(x)\|^2\notag\\
 &=\|p^J_{x^{k+1}}\|^2+\|\nabla C(x)\|^2.
 \label{eq:full-pythagoras}
\end{align}
Both nonnegative terms are bounded by $\|p_{x^{k+1}}\|^2$. Taking square roots yields the two pointwise bounds
\begin{align}
 \operatorname{dist}(0,\partial J(x))
 &\leq\|p^J_{x^{k+1}}\|\leq\|p_{x^{k+1}}\|,
 \qquad\|\nabla C(x)\|\leq\|p_{x^{k+1}}\|.\notag
\end{align}
Theorem~\ref{thm:controlled-summability} proves Eq.~(\ref{eq:full-residual-summability}) by summation.

(ii) At $x^*$, Eq.~(\ref{eq:controlled-original-limit}) gives $-\nabla C(x^*)\in\partial J(x^*)$. Apply Eq.~(\ref{eq:full-scale-orthogonality}) at this point:
\begin{align}
 -\nabla C(x^*)\in\mathcal S(x^*)^\perp,
 \qquad\nabla C(x^*)\in\mathcal S(x^*).\notag
\end{align}
Their inner product must vanish, so
\begin{align}
 0&=\langle-\nabla C(x^*),\nabla C(x^*)\rangle
   =-\|\nabla C(x^*)\|^2.\notag
\end{align}
Hence $\nabla C(x^*)=0$, and its negative original certificate is zero. This proves full original criticality in Eq.~(\ref{eq:full-original-criticality}).

(iii) Since $\nabla f$ is locally Lipschitz on $U$ and $\nabla H$ is Lipschitz on bounded sets, $\nabla C=\nabla f-\nabla H$ is locally Lipschitz at $x^*$. Choose a convex closed ball about $x^*$ contained in $U$ on which it is $L$-Lipschitz, and choose $k_L$ so all subsequent iterates lie in this ball. For $t\in[0,1]$, the gradient at a segment point satisfies
\begin{align}
 &\|\nabla C((1-t)x_{(N)}^k+t x_{(N)}^{k+1})\|\notag\\
 &\quad=\|\nabla C((1-t)x_{(N)}^k+t x_{(N)}^{k+1})-\nabla C(x^*)\|\notag\\
 &\quad\leq L\|(1-t)(x_{(N)}^k-x^*)+t(x_{(N)}^{k+1}-x^*)\|\notag\\
 &\quad\leq L[(1-t)S_k+tS_{k+1}],\notag
\end{align}
using $\|x_{(N)}^k-x^*\|\leq S_k$. Integrate this bound along the segment and use $S_k-S_{k+1}=d_k$:
\begin{align}
 |C(x_{(N)}^{k+1})-C(x_{(N)}^k)|
 &\leq Ld_k\int_0^1[(1-t)S_k+tS_{k+1}]\,dt\notag\\
 &=\frac L2(S_k+S_{k+1})d_k
 =\frac L2(S_k^2-S_{k+1}^2).\notag
\end{align}
Summing from $k=m$ to $n$ and letting $n\to\infty$ proves the first line of Eq.~(\ref{eq:full-quadratic-variation-tail}). Use $J=\Phi-C$ and telescope the decreasing controlled values to obtain its second line. Finally, integration from $x^*$ to $x_{(N)}^m$ gives
\begin{align}
 |C(x_{(N)}^m)-C(x^*)|
 &\leq\int_0^1Lt\|x_{(N)}^m-x^*\|^2\,dt
 =\frac L2\|x_{(N)}^m-x^*\|^2.\notag
\end{align}
Substitute this in $J(x_{(N)}^m)-J(x^*)=r_m-C(x_{(N)}^m)+C(x^*)$ to obtain the last line. The residual summability statements also hold at intermediate points by applying the same orthogonal identity to Proposition~\ref{prop:intermediate-certificates}. For the quadratic variation bound along all block steps, the identical segment argument uses the remaining total block-path length in place of $S_m$.
\end{proof}
The structural results below provide a way to verify Eq.~(\ref{eq:full-scale-orthogonality}) through original feasible transformations.

Positive homogeneity converts summable control gradients into summable
control values and, when the whole control has this property, summable
absolute errors of the original objective.

\begin{corollary}[Nonnegative homogeneous controls]
\label{cor:homogeneous-control}
Let $C_0$ be nonnegative, continuously differentiable near a convergent
bounded sequence $x_{(N)}^k\to x^*$ and its limit, and positively homogeneous of degree
$\kappa>0$. If $\sum_{k=1}^{\infty}\|\nabla C_0(x_{(N)}^k)\|<\infty$, then
\begin{equation}
 \sum_{k=1}^{\infty}C_0(x_{(N)}^k)<\infty,\qquad C_0(x^*)=0.
 \label{eq:homogeneous-control-summability}
\end{equation}
When $C=C_0$ and the controlled gaps
$r_k=\Phi(x_{(N)}^k)-\Phi(x^*)$ are nonnegative and summable, one also has
\begin{equation}
 \sum_{k=1}^{\infty}|J(x_{(N)}^k)-J(x^*)|<\infty,
 \qquad\Phi(x^*)=J(x^*).
 \label{eq:homogeneous-original-value-summability}
\end{equation}
\end{corollary}
\begin{proof}
Differentiate $C_0(tx)=t^\kappa C_0(x)$ at $t=1$ and set
$M_x=\sup_k\|x_{(N)}^k\|<\infty$. Euler's identity gives
\begin{align}
 0\leq\kappa C_0(x_{(N)}^k)
 &=\langle\nabla C_0(x_{(N)}^k),x_{(N)}^k\rangle
 \leq M_x\|\nabla C_0(x_{(N)}^k)\|.\notag
\end{align}
Summation proves the first assertion, and continuity gives $C_0(x^*)=0$.
For the second assertion, $J=\Phi-C_0$ gives
\begin{align}
 |J(x_{(N)}^k)-J(x^*)|
 &=|r_k-C_0(x_{(N)}^k)|\leq r_k+C_0(x_{(N)}^k).\notag
\end{align}
Both terms are summable.
\end{proof}

The same orthogonal decomposition recovers full original criticality on
the blocks where the remaining control has zero gradient.

\begin{proposition}[Original block criticality and exact balance]
\label{prop:partial-criticality}
Under the convergence and summability conclusions of
Theorems~\ref{thm:controlled-global} and \ref{thm:controlled-summability},
write $C=C_{\mathrm{bal}}+C_{\mathrm{rest}}$, with both parts continuously
differentiable on $U$, and $C_{\mathrm{bal}}$ nonnegative and positively
homogeneous of degree $\kappa>0$.
Suppose that, at the generated points and their limit, a space
$\mathcal S_0(x)$ of velocities of transformations preserving the original
objective satisfies
\begin{equation}
 \partial J(x)\subseteq\mathcal S_0(x)^\perp,\qquad
 \nabla C_{\mathrm{bal}}(x)\in\mathcal S_0(x),\qquad
 \nabla C_{\mathrm{rest}}(x)\in\mathcal S_0(x)^\perp.
 \label{eq:partial-criticality-orthogonality}
\end{equation}
Let $\mathcal B_-$ be a collection of original block indices such that
$\nabla_iC_{\mathrm{rest}}(x)=0$ for $i\in\mathcal B_-$. Define
\begin{equation}
 \begin{aligned}
 r_-(x)&=\left(\sum_{i\in\mathcal B_-}
  \dist(0,\nabla_iH(x)+\partial F_i(x_i))^2\right)^{1/2},\\
 R(x)&=\dist(0,\widehat\partial\Phi(x)).
 \end{aligned}
 \label{eq:partial-original-block-residual}
\end{equation}
Then, for $k\geq1$,
\begin{equation}
 \begin{aligned}
 \|\nabla C_{\mathrm{bal}}(x_{(N)}^k)\|&\leq R(x_{(N)}^k)
       \leq\|p_{x^k}\|,\\
 r_-(x_{(N)}^k)&\leq R(x_{(N)}^k),
 \end{aligned}
 \label{eq:partial-criticality-bounds}
\end{equation}
and each of these quantities is summable. Moreover,
\begin{equation}
 \begin{aligned}
 \sum_{k=1}^{\infty}C_{\mathrm{bal}}(x_{(N)}^k)&<\infty,&
 C_{\mathrm{bal}}(x^*)&=0,\\
 \nabla C_{\mathrm{bal}}(x^*)&=0,&
 0&\in\nabla_iH(x^*)+\partial F_i(x_i^*)
       \quad(i\in\mathcal B_-).
 \end{aligned}
 \label{eq:partial-criticality-limit}
\end{equation}
For $F_i=\delta_{E_i}$ these are full original block conditions with
$\partial F_i(x_i)=N_{E_i}(x_i)$.
\end{proposition}
\begin{proof}
(i) For $p\in\partial\Phi(x)$, smooth addition gives
\begin{align}
 q=p-\nabla C_{\mathrm{bal}}(x)
 &\in\partial J(x)+\nabla C_{\mathrm{rest}}(x)
 \subseteq\mathcal S_0(x)^\perp.\notag
\end{align}
Therefore
\begin{equation}
 \|p\|^2=\|q\|^2+\|\nabla C_{\mathrm{bal}}(x)\|^2.
 \label{eq:partial-criticality-pythagoras}
\end{equation}
The block-separable subdifferential formula for Eq.~(\ref{e11}) and
$\nabla_iC_{\mathrm{rest}}=0$ give
$q_i\in\nabla_iH(x)+\partial F_i(x_i)$ for $i\in\mathcal B_-$.
Consequently,
\begin{align}
 r_-(x)^2&\leq\sum_{i\in\mathcal B_-}\|q_i\|^2
        \leq\|q\|^2\leq\|p\|^2,\qquad
 \|\nabla C_{\mathrm{bal}}(x)\|\leq\|p\|.\notag
\end{align}
These inequalities hold for every regular certificate as well. Taking its
norm infimum gives the bounds by $R(x)$, and the constructed regular
certificate gives $R(x_{(N)}^k)\leq\|p_{x^k}\|$.
Theorem~\ref{thm:controlled-summability} proves all residual and gradient
summability assertions. Corollary~\ref{cor:homogeneous-control} proves
the balance-value assertions.

(ii) At $x^*$, Theorem~\ref{thm:controlled-global} gives
$0\in\partial\Phi(x^*)$. Set $p=0$ in
Eq.~(\ref{eq:partial-criticality-pythagoras}) to obtain
$\nabla C_{\mathrm{bal}}(x^*)=0$ and
$-\nabla C_{\mathrm{rest}}(x^*)\in\partial J(x^*)$.
Its blocks in $\mathcal B_-$ are zero, proving the original block
criticality in Eq.~(\ref{eq:partial-criticality-limit}).
If the same structural relations hold at intermediate block iterates,
Proposition~\ref{prop:intermediate-certificates} and the identical
pointwise bounds give summability along the complete block path.
\end{proof}

The following corollary applies the proximal comparison and smooth sum rule
directly to the original problem in Eq.~(\ref{e11}). It uses a full-variable
proximal update and establishes original-value and full-residual convergence
under Assumption~\ref{ass:original}.
\begin{corollary}[Original proximal updates for the general problem]
\label{prop:general}
Under Assumption~\ref{ass:original}, let
$0<\lambda_-\leq\lambda_k\leq\lambda_+<\infty$ and
\begin{equation}
 x_{(N)}^{k+1}\in\operatorname*{arg\,min}_u
       \left\{J(u)+\frac{\lambda_k}{2}\|u-x_{(N)}^k\|^2\right\}.
 \label{eq:generalprox}
\end{equation}
From every $x_{(N)}^0\in\operatorname{dom}J$, all these minimizers exist and are in
the original domain. For a finite lower bound $m_0$ of $J$,
\begin{gather*}
 J(x_{(N)}^k)\to\ell_J\in\mathbb R,\qquad
 \sum_k\|x_{(N)}^{k+1}-x_{(N)}^k\|^2
       \leq\frac{2(J(x_{(N)}^0)-m_0)}{\lambda_-},\\
 \operatorname{dist}(0,\partial J(x_{(N)}^{k+1}))
       \leq\lambda_+\|x_{(N)}^{k+1}-x_{(N)}^k\|\to0.
\end{gather*}
Every finite cluster point $\bar x$ has $J(\bar x)=\ell_J$ and
$0\in\partial J(\bar x)$.
\end{corollary}
\begin{proof}
Proper lower semicontinuous block functions and the finite continuous $H$
make $J$ proper and lower semicontinuous: near any point, each $F_i$ has a
finite local lower bound, and the liminf inequality gives lower
semicontinuity of their finite sum. Each anchored objective in
Eq.~(\ref{eq:generalprox}) is proper and lower semicontinuous, with
sublevels contained in finite balls by its quadratic term and $J\geq m_0$.
A minimizing sequence has a convergent subsequence in such a sublevel;
lower semicontinuity gives a global minimizer.
Comparison with the anchor gives
\begin{equation*}
 J(x_{(N)}^{k+1})+\frac{\lambda_k}{2}\|x_{(N)}^{k+1}-x_{(N)}^k\|^2\leq J(x_{(N)}^k).
\end{equation*}
Its lower bound proves the value limit and square-summable increments.
Expanding the global minimizing comparison at $x_{(N)}^{k+1}$ gives
\begin{equation*}
 -\lambda_k(x_{(N)}^{k+1}-x_{(N)}^k)\in\widehat\partial J(x_{(N)}^{k+1}),
\end{equation*}
which proves the full residual estimate. For $x_{(N)}^{k_j}\to\bar x$, lower
semicontinuity first gives $J(\bar x)\leq\ell_J$. The increments vanish,
so $x_{(N)}^{k_j-1}\to\bar x$. Global comparison with $u=\bar x$ gives
\begin{equation*}
 J(x_{(N)}^{k_j})+\frac{\lambda_{k_j-1}}2\|x_{(N)}^{k_j}-x_{(N)}^{k_j-1}\|^2
 \leq J(\bar x)+\frac{\lambda_{k_j-1}}2\|\bar x-x_{(N)}^{k_j-1}\|^2.
\end{equation*}
Taking limits gives the reverse value inequality and hence value compatibility.
The displayed true Fr\'echet subgradients tend to zero along this subsequence,
so $0\in\partial J(\bar x)$ follows from the limiting definition.
\end{proof}

\subsection{Convergence rates}
\label{subsec:convergence-rates}
We first quantify the convergence of the controlled objective $\Phi=J+C$.
We then give separate rates for the original variables and the original
objective $J$, using the path and value bounds proved above. The final
corollary gives rates for the original stationarity residuals.

\begin{theorem}[Convergence rates of the controlled objective]
\label{thm:convergence-rates}
Let $\{x_{(N)}^k\}$ be generated by Algorithm~\ref{alg:cbpl}, and put
$r_k=\Phi(x_{(N)}^k)-\Phi(x^*)$. Then $r_k=o(k^{-1})$.
Suppose that the desingularizing function at $x^*$ is
$\varphi(t)=\frac{c_{\mathrm{KL}}}{\theta}t^\theta$, where
$c_{\mathrm{KL}}>0$ and $0<\theta\leq1$. The following assertions hold.

(\romannumeral1) If $\theta=1$, the algorithm terminates after finitely many
iterations.

(\romannumeral2) If $\frac12\leq\theta<1$, there are $c_1>0$,
$q\in(0,1)$ and $k_0$ such that
$r_k\leq c_1q^{k-k_0}$ for every $k\geq k_0$.

(\romannumeral3) If $0<\theta<\frac12$, there are $c_2>0$ and $k_0$ such that
\begin{equation}
 r_k\leq c_2(k-k_0)^{-\frac{1}{1-2\theta}},\qquad k>k_0.
 \label{eq:rate-polynomial-gap}
\end{equation}
\end{theorem}
\begin{proof}
Theorem~\ref{thm:controlled-summability} gives $\sum_kr_k<\infty$.
Set $q_K=\lfloor K/2\rfloor$. Since $r_k$ is nonnegative and decreasing,
\begin{align}
 Kr_K
 &\leq\frac{K}{K-q_K}\sum_{j=q_K+1}^{K}r_j
 \leq2\sum_{j=q_K+1}^{\infty}r_j\longrightarrow0.
 \label{eq:rate-summable-gap-certificate}
\end{align}
This proves the last-iterate bound.

If $r_k=0$ at some iteration, sufficient decrease makes all
subsequent displacements zero. Otherwise, whole-sequence and value
convergence place the iterates in the KL neighborhood for all $k\geq k_0$,
where $k_0\geq1$. Since $\varphi'(t)=c_{\mathrm{KL}}t^{\theta-1}$,
Eq.~(\ref{eq:kl-certificate-chain}) gives
\begin{align}
 1&\leq c_{\mathrm{KL}}(r_k)^{\theta-1}\|p_{x^k}\|
      \leq c_{\mathrm{KL}}\rho_b(r_k)^{\theta-1}d_{k-1},\notag\\
 r_{k-1}-r_k
 &\geq\rho d_{k-1}^2
 \geq\kappa(r_k)^{2-2\theta},\qquad
 \kappa=\frac{\rho}{c_{\mathrm{KL}}^2\rho_b^2}>0.
 \label{eq:rate-gap-recurrence}
\end{align}
For $\theta=1$, each positive gap decreases by at least $\kappa$.
Summation would give $r_k\leq r_{k_0}-\kappa(k-k_0)<0$ for large $k$,
a contradiction. Thus a zero gap is attained in finitely many steps.

For $\frac12\leq\theta<1$, increase $k_0$ so that $0<r_k\leq1$.
Then $2-2\theta\in(0,1]$, and Eq.~(\ref{eq:rate-gap-recurrence}) yields
\begin{align}
 r_{k-1}-r_k&\geq\kappa(r_k)^{2-2\theta}\geq\kappa r_k,
 &r_k&\leq\frac{r_{k-1}}{1+\kappa}.\notag
\end{align}
Iteration proves the linear bound with $q=(1+\kappa)^{-1}$ and
$c_1=r_{k_0}$.

For $0<\theta<\frac12$, put $\gamma=1-2\theta>0$.
If $r_k\leq\frac12r_{k-1}$, monotonicity gives
\begin{align}
 (r_k)^{-\gamma}-(r_{k-1})^{-\gamma}
 &\geq(2^\gamma-1)(r_{k-1})^{-\gamma}
 \geq(2^\gamma-1)(r_{k_0})^{-\gamma}.\notag
\end{align}
If $r_k>\frac12r_{k-1}$, integration and
Eq.~(\ref{eq:rate-gap-recurrence}) give
\begin{align}
 (r_k)^{-\gamma}-(r_{k-1})^{-\gamma}
 &=\gamma\int_{r_k}^{r_{k-1}}t^{-1-\gamma}\,dt\notag\\
 &\geq\gamma(r_{k-1})^{-1-\gamma}(r_{k-1}-r_k)\notag\\
 &\geq\gamma\kappa
       \left(\frac{r_k}{r_{k-1}}\right)^{1+\gamma}
 \geq\frac{\gamma\kappa}{2^{1+\gamma}}.\notag
\end{align}
Thus every inverse-gap increment is at least
$c=\min\{(2^\gamma-1)(r_{k_0})^{-\gamma},
\frac{\gamma\kappa}{2^{1+\gamma}}\}>0$. Finite summation gives Eq.~(\ref{eq:rate-polynomial-gap}):
\begin{align}
 (r_k)^{-\gamma}
 &\geq(r_{k_0})^{-\gamma}+c(k-k_0),
 &r_k&\leq c^{-\frac1\gamma}(k-k_0)^{-\frac1\gamma}.
 \notag
\end{align}

\end{proof}

The rates of $\Phi$ transfer to the original variables through the remaining
path length. The original objective has its own rate because
$J=\Phi-C$; the orthogonality conditions give a quadratic bound on the
remaining variation of $C$.

\begin{theorem}[Convergence rates of the original variables and objective]
\label{thm:original-rates}
Let $\{x_{(N)}^k\}$ be the sequence in Theorem~\ref{thm:controlled-global},
with $r_k$ and $S_k$ from Eq.~(\ref{eq:length-tail-notation}).

(\romannumeral1) Suppose
$\varphi(t)=\frac{c_{\mathrm{KL}}}{\theta}t^\theta$,
$c_{\mathrm{KL}}>0$ and $0<\theta\leq1$.
For $\theta=1$, the original variables and their objective values become
constant after finitely many iterations. For $\frac12\leq\theta<1$, let
$q\in(0,1)$ be the factor in Theorem~\ref{thm:convergence-rates}. Then
\begin{equation}
 \|x_{(N)}^k-x^*\|=O(q^{k/2}),\qquad
 |J(x_{(N)}^k)-J(x^*)|=O(q^{k/2}).
 \label{eq:original-linear-rates}
\end{equation}
For $0<\theta<\frac12$,
\begin{equation}
 \|x_{(N)}^k-x^*\|=O\!\left(k^{-\frac{\theta}{1-2\theta}}\right),\qquad
 |J(x_{(N)}^k)-J(x^*)|
 =O\!\left(k^{-\frac{\theta}{1-2\theta}}\right).
 \label{eq:original-polynomial-rates}
\end{equation}

(\romannumeral2) Under the orthogonality conditions in
Eq.~(\ref{eq:full-scale-orthogonality}), the original objective rates improve to
\begin{equation}
 |J(x_{(N)}^k)-J(x^*)|=
 \begin{cases}
 O(q^k),&\frac12\leq\theta<1,\\
 O\!\left(k^{-\frac{2\theta}{1-2\theta}}\right),&0<\theta<\frac12.
 \end{cases}
 \label{eq:full-original-value-rates}
\end{equation}
Consequently the last original objective error is $O(k^{-1})$ for
$\theta\geq\frac13$ in (i), and for $\theta\geq\frac14$ under
Eq.~(\ref{eq:full-scale-orthogonality}).

(\romannumeral3) If the whole control $C$ satisfies
Corollary~\ref{cor:homogeneous-control}, then
\begin{equation}
 \sum_{k=1}^{\infty}|J(x_{(N)}^k)-J(x^*)|<\infty,\qquad
 \min_{1\leq j\leq K}|J(x_{(N)}^j)-J(x^*)|=o(K^{-1}).
 \label{eq:original-best-value-rate}
\end{equation}
\end{theorem}
\begin{proof}
(i) Sufficient decrease gives $\rho d_{k-1}^2\leq r_{k-1}-r_k\leq r_{k-1}$.
Substitute this bound and the power desingularizing function into
Eq.~(\ref{eq:kl-length-tail}) to obtain
\begin{align}
 \|x_{(N)}^k-x^*\|\leq S_k
 &\leq d_{k-1}+\frac{\rho_b}{\rho}\varphi(r_k)\notag\\
 &\leq\sqrt{\frac{r_{k-1}}{\rho}}
       +\frac{\rho_bc_{\mathrm{KL}}}{\rho\theta}r_k^\theta.
 \label{eq:rate-original-path}
\end{align}
For $\frac12\leq\theta<1$, Theorem~\ref{thm:convergence-rates} gives
$r_k=O(q^k)$. The two terms in Eq.~(\ref{eq:rate-original-path}) are
$O(q^{k/2})$ and $O(q^{k\theta})$, respectively. Since $\theta\geq\frac12$,
\begin{equation*}
 \|x_{(N)}^k-x^*\|\leq S_k=O(q^{k/2}).
\end{equation*}
For $0<\theta<\frac12$, the same substitution gives
\begin{align}
 S_k
 &\leq O\!\left(k^{-\frac{1}{2(1-2\theta)}}\right)
       +O\!\left(k^{-\frac{\theta}{1-2\theta}}\right)
 =O\!\left(k^{-\frac{\theta}{1-2\theta}}\right).
 \label{eq:rate-polynomial-path}
\end{align}
The original-value bound in Eq.~(\ref{eq:original-variation-tail}) now yields
\begin{align}
 |J(x_{(N)}^k)-J(x^*)|
 &\leq r_k+L_CS_k\notag\\
 &=\begin{cases}
 O(q^{k/2}),&\frac12\leq\theta<1,\\
 O\!\left(k^{-\frac{\theta}{1-2\theta}}\right),&0<\theta<\frac12.
 \end{cases}\notag
\end{align}
For $\theta=1$, Theorem~\ref{thm:convergence-rates} makes every subsequent
displacement zero after a finite iteration, proving the finite case.

(ii) Theorem~\ref{thm:full} replaces the linear path bound by
\begin{equation*}
 |J(x_{(N)}^k)-J(x^*)|\leq r_k+\frac L2 S_k^2.
\end{equation*}
In the linear case this is $O(q^k)$. In the polynomial case,
\begin{align}
 r_k+\frac L2S_k^2
 &=O\!\left(k^{-\frac{1}{1-2\theta}}\right)
   +O\!\left(k^{-\frac{2\theta}{1-2\theta}}\right)
 =O\!\left(k^{-\frac{2\theta}{1-2\theta}}\right).
 \notag
\end{align}
The thresholds follow from
$\theta/(1-2\theta)\geq1\iff\theta\geq1/3$ and
$2\theta/(1-2\theta)\geq1\iff\theta\geq1/4$ in the polynomial range;
finite termination and geometric convergence also satisfy $O(k^{-1})$.

(iii) Corollary~\ref{cor:homogeneous-control} gives the absolute-error sum.
For $q_K=\lfloor K/2\rfloor$,
\begin{align}
 K\min_{1\leq j\leq K}|J(x_{(N)}^j)-J(x^*)|
 &\leq\frac{K}{K-q_K}\sum_{j=q_K+1}^{K}|J(x_{(N)}^j)-J(x^*)|\notag\\
 &\leq2\sum_{j=q_K+1}^{\infty}|J(x_{(N)}^j)-J(x^*)|
 \longrightarrow0.
 \notag
\end{align}
This proves Eq.~(\ref{eq:original-best-value-rate}).
\end{proof}

The following result bounds the first-order residuals of the original
objective by the controlled certificates.
\begin{corollary}[Rates of the original stationarity residuals]
\label{cor:original-residual-rates}
For every integer $K\geq1$,
\begin{equation}
 \min_{1\leq j\leq K}\|p_{x^j}\|^2
 \leq\frac{\rho_b^2}{\rho K}
       \bigl(\Phi(x_{(N)}^0)-\underline\Phi\bigr),
 \label{eq:rate-squared-certificate}
\end{equation}
where $\underline\Phi$ is a finite lower bound. This estimate follows
from sufficient decrease and the subgradient bound. The same upper bound holds for the best squared original effective residual and, under Eq.~(\ref{eq:full-scale-orthogonality}), the best squared full original residual. Under the KL framework,
\begin{equation}
 \min_{1\leq j\leq K}\|p_{x^j}\|=o(K^{-1}),\qquad
 \min_{1\leq j\leq K}r_{\mathrm{eff}}(x_{(N)}^j)^2=o(K^{-2}).
 \label{eq:rate-original-residuals}
\end{equation}
At $j_K\in\argmin_{1\leq j\leq K}\|p_{x^j}\|$, the original effective
residual is $o(K^{-1})$. Under Eq.~(\ref{eq:full-scale-orthogonality}), the
full original residual, the original certificate norm $\|p^J_{x^{j_K}}\|$ and the control gradient norm have the same bound at $j_K$;
in particular,
\begin{equation}
 \min_{1\leq j\leq K}\operatorname{dist}(0,\partial J(x_{(N)}^j))^2
 =o(K^{-2}).
 \label{eq:rate-full-original-residual}
\end{equation}
\end{corollary}
\begin{proof}
Equations~(\ref{eq:descent}) and (\ref{eq:subgradientbound}) imply
\begin{align}
 K\min_{1\leq j\leq K}\|p_{x^j}\|^2
 &\leq\sum_{j=1}^{K}\|p_{x^j}\|^2
 \leq\rho_b^2\sum_{j=1}^{K}d_{j-1}^2\notag\\
 &\leq\frac{\rho_b^2}{\rho}
       \bigl(\Phi(x_{(N)}^0)-\Phi(x_{(N)}^K)\bigr)
 \leq\frac{\rho_b^2}{\rho}
       \bigl(\Phi(x_{(N)}^0)-\underline\Phi\bigr).\notag
\end{align}
Dividing by $K$ gives Eq.~(\ref{eq:rate-squared-certificate}). Under the
KL framework, Theorem~\ref{thm:controlled-summability} gives
$\sum_j\|p_{x^j}\|<\infty$. Hence, with $q_K=\lfloor K/2\rfloor$,
\begin{align}
 K\min_{1\leq j\leq K}\|p_{x^j}\|
 &\leq\frac{K}{K-q_K}\sum_{j=q_K+1}^{K}\|p_{x^j}\|
 \leq2\sum_{j=q_K+1}^{\infty}\|p_{x^j}\|\longrightarrow0.\notag
\end{align}
The pointwise comparisons in Eqs.~(\ref{eq:controlled-transfer}) and
(\ref{eq:full-original-criticality}) give the original residual rates at
the same $j_K$. Squaring their bounds proves
Eqs.~(\ref{eq:rate-original-residuals})--(\ref{eq:rate-full-original-residual}).
\end{proof}

\subsection{Feasible transformations and original residuals}
\label{subsec:reconstruction-structure}
The original-criticality results above use comparison geometry and scale
orthogonality supplied by the original problem. We establish the structural
results used in Section~\ref{sec:experiments} to verify condition~(\ref{cond:c4})
and the scale orthogonality in Eq.~(\ref{eq:full-scale-orthogonality}), and to
compute the original residuals for nonnegative sparsity constraints.
Feasible transformations give orthogonality to the original limiting normal
cone and preserve reconstruction changes after projection.
Let $\mathcal E$ be a closed feasible set in the original product space.
\begin{lemma}[Feasible transformations and scale orthogonality]
\label{lem:scale-orthogonality}
Let $\mathcal D_\gamma$ be a jointly continuously differentiable transformation
preserving the original feasible set in both parameter directions, as in
Section~\ref{subsec:reconstruction-preliminaries}. Then
\begin{equation}
 N_{\E}(x)\subseteq\Sscale(x)^\perp.
 \label{eq:action-normal-orthogonality}
\end{equation}
If $J=H+\delta_{\E}$ and $H$ is invariant under this transformation, then
\begin{equation}
 \partial J(x)\subseteq\Sscale(x)^\perp.
 \label{eq:action-original-orthogonality}
\end{equation}
More generally, let $Z$ be continuously differentiable, suppose the transformation
preserves $Z$, and let
\begin{equation}
 H(x)=H_{\mathrm{rec}}(Z(x))+\frac\alpha2\operatorname{tr}(A(x)^TLA(x)),
 \qquad\alpha>0,\quad L=L^T\succeq0,
 \label{eq:action-quadratic-objective}
\end{equation}
where $A(x)$ is a matrix block of $x$ and $H_{\mathrm{rec}}$ is a
continuously differentiable loss depending on the reconstruction.
If, for every column $a_r$ of $A$, a transformation velocity
$v_r$ satisfies $\dot a_r=\kappa_r a_r$, $\kappa_r\ne0$, and
$\dot a_j=0$ for $j\ne r$, then
\begin{equation}
 \langle p,v_r\rangle=\alpha\kappa_r a_r^TLa_r,
 \qquad\forall p\in\partial J(x).
 \label{eq:action-quadratic-derivative}
\end{equation}
Consequently, the pairings with every direction in $\Sscale(x)$ vanish
if and only if $LA(x)=0$, and full original criticality implies $LA(x)=0$.
\end{lemma}
\begin{proof}
For $n\in\widehat N_{\E}(x)$ and $v=T_x\gamma$, the original feasible
curve satisfies
$\mathcal D_{t\gamma}x=x+tv+o(|t|)$ for both signs of $t$.
The regular-normal inequality gives
\begin{align}
 \langle n,v\rangle\leq0\quad(t\downarrow0),\qquad
 \langle n,v\rangle\geq0\quad(t\uparrow0).
 \notag
\end{align}
If $v=0$, the same equality holds directly. For $n\in N_{\E}(x)$,
choose $x^\ell\to x$ and $n^\ell\to n$ with
$n^\ell\in\widehat N_{\E}(x^\ell)$. The transformation is feasible at each
$x^\ell$, and joint differentiability gives $T_{x^\ell}\gamma\to T_x\gamma$.
Hence
\begin{align}
 \langle n,T_x\gamma\rangle
 &=\lim_{\ell\to\infty}\langle n^\ell,T_{x^\ell}\gamma\rangle=0.\notag
\end{align}
This proves Eq.~(\ref{eq:action-normal-orthogonality}). Smooth addition
gives $p=\nabla H(x)+n$ for every $p\in\partial J(x)$. Differentiating
$H(\mathcal D_{t\gamma}x)=H(x)$ proves
Eq.~(\ref{eq:action-original-orthogonality}).

For Eq.~(\ref{eq:action-quadratic-objective}), the loss derivative is zero.
The quadratic derivative and normal orthogonality give
\begin{align}
 \langle p,v_r\rangle
 &=\alpha\langle La_r,\kappa_r a_r\rangle
 =\alpha\kappa_r\|L^{1/2}a_r\|^2.\notag
\end{align}
Thus vanishing of all these pairings forces $La_r=0$ for every $r$.
If $LA=0$, the entire quadratic gradient is zero, and reconstruction
invariance makes the loss derivative zero along every transformation velocity.
This proves the equivalence. Taking $p=0$ proves the full-criticality assertion.
\end{proof}

Projection along these actual transformations retains feasible reconstruction
changes, including those passing through singular or zero structures.

\begin{theorem}[Preservation of feasible reconstruction changes]
\label{app:fiber-preservation}
Let $Z$ be continuously differentiable. Suppose that the feasible
reconstruction-preserving transformation and the spaces defined in
Section~\ref{subsec:reconstruction-preliminaries} satisfy
$\R^{\sum_i d_i}=\Sscale(x)\oplus\Heff(x)$. Then
\begin{equation}
 \Pi_x(T_{\E}(x)\setminus\V(x))=\Qeff(x),\qquad
 DZ(x)(\Qeff(x))\setminus\{0\}
 =DZ(x)(T_{\E}(x))\setminus\{0\}.
 \label{app:image-preservation}
\end{equation}
For every right-differentiable feasible curve $y(t)$ with $y(0)=x$,
there is a feasible curve $z(t)$ satisfying
\begin{equation}
 z(0)=x,\qquad z'(0+)=\Pi_xy'(0+),\qquad
 Z(z(t))=Z(y(t))
 \label{app:curve-preservation}
\end{equation}
for all sufficiently small $t\geq0$. Moreover,
\begin{equation}
 \begin{aligned}
 T_{\E}(x)+\Sscale(x)&=T_{\E}(x),&
 \V(x)+\Sscale(x)&=\V(x),\\
 \Pi_xT_{\E}(x)&=T_{\E}(x)\cap\Heff(x),&
 \Pi_x\V(x)&=\V(x)\cap\Heff(x),\\
 d\in\V(x)&\Longleftrightarrow\Pi_xd\in\V(x).
 \end{aligned}
 \label{app:cone-identities}
\end{equation}
\end{theorem}
\begin{proof}
(i) Let $d\in T_{\E}(x)$ and $s=T_x\gamma\in\Sscale(x)$.
By the definition of the actual tangent cone, choose $t_\ell\downarrow0$
and $y_\ell\in\E$ with $y_\ell=x+t_\ell d+o(t_\ell)$.
Joint differentiability and feasibility give
\begin{align}
 \mathcal D_{t_\ell\gamma}y_\ell
 &=x+t_\ell(d+T_x\gamma)+o(t_\ell)\in\E.\notag
\end{align}
Therefore $d+s\in T_{\E}(x)$. If $d\in\V(x)$, choose this sequence
in the exact fiber $E_f(x)$. Reconstruction invariance then gives
$Z(\mathcal D_{t_\ell\gamma}y_\ell)=Z(y_\ell)=Z(x)$, so
$d+s\in\V(x)$. Taking $s=0$ proves the two cone equalities.

Since $d-\Pi_xd\in\Sscale(x)$ and both signs of this velocity are
allowed, those equalities imply
\begin{align}
 d\in T_{\E}(x)&\Longleftrightarrow\Pi_xd\in T_{\E}(x),&
 d\in\V(x)&\Longleftrightarrow\Pi_xd\in\V(x).\notag
\end{align}
Together with $\Pi_xd\in\Heff(x)$ and $\Pi_xd=d$ on $\Heff(x)$,
these prove Eq.~(\ref{app:cone-identities}) and the first identity in
Eq.~(\ref{app:image-preservation}).

(ii) Differentiating $Z(\mathcal D_{t\gamma}x)=Z(x)$ gives
$DZ(x)T_x\gamma=0$, hence $DZ(x)\Pi_xd=DZ(x)d$.
Also, differentiability of $Z$ along sequences in the exact fiber gives
$\V(x)\subseteq\ker DZ(x)$. Thus every $d\in T_{\E}(x)$ with
$DZ(x)d\ne0$ lies outside $\V(x)$, and its projection belongs to
$\Qeff(x)$ with the same reconstruction derivative. The reverse image
inclusion follows from $\Qeff(x)\subseteq T_{\E}(x)$.

(iii) Let $d=y'(0+)$ and choose $\gamma$ with
$T_x\gamma=d-\Pi_xd$. Set $z(t)=\mathcal D_{-t\gamma}y(t)$.
The actual transformation gives feasibility and $Z(z(t))=Z(y(t))$.
Its first-order expansion is
\begin{align}
 z(t)&=x+t(d-T_x\gamma)+o(t)
      =x+t\Pi_xd+o(t),\notag
\end{align}
which proves Eq.~(\ref{app:curve-preservation}).
\end{proof}

\label{subsec:sparse-stationarity}
For nonnegative sparsity constraints, the following formulas also give
attainment and finite computation of the original residual.

Let $\mathcal F$ be a nonempty downward-closed family of subsets of
$\{1,\ldots,d\}$, and consider
\begin{equation}
 \E=\{x\in\R_+^d:\supp(x)\in\mathcal F\},\qquad J=H+\delta_{\E},
 \label{eq:sparse-feasible-family}
\end{equation}
where $H$ is continuously differentiable. Products of nonnegative
$\ell_0$-constrained blocks have this form after their entries are collected
into one vector. Write $I(x)=\{j:x_j>0\}$. For $T\in\mathcal F$, define
\begin{equation}
 \begin{aligned}
 A(T)&=\{j\notin T:T\cup\{j\}\in\mathcal F\},\\
 \mathcal N(T)&=\{n:n_j=0\ (j\in T),\ n_j\leq0\ (j\in A(T)),
                   \ n_j\in\mathbb R\ (j\notin T\cup A(T))\}.
 \end{aligned}
 \label{app:support-normal-cone}
\end{equation}

\begin{lemma}
\label{app:true-normal-formula}
(i) $\forall x\in \E$, we have
\begin{equation}
 \widehat N_{\E}(x)=\mathcal N(I(x)),\qquad
 N_{\E}(x)=\bigcup_{\substack{T\in\mathcal F\\I(x)\subseteq T}}\mathcal N(T),
 \qquad \partial J(x)=\nabla H(x)+N_{\E}(x).
 \label{app:normal-formula}
\end{equation}
(ii) For a positive-homogeneous comparison set $Q$, the infimum
\begin{equation}
 \inf_{g\in\partial J(x)}
    \sup_{d\in Q,\ \|d\|=1}|\langle g,d\rangle|
 \label{app:attained-seminorm}
\end{equation}
is attained. Its value is zero precisely when an original limiting subgradient
annihilates $Q$.

(iii) For a linear comparison space $\Heff(x)$ and $W(x)=\Heff(x)^\perp$,
an attaining certificate for $\operatorname{dist}(\partial J(x),W(x))$
can be computed by finitely many linear systems.
\end{lemma}

\begin{proof}
(i) Fix $x\in\E$ and put $T=I(x)$. If $n\in\widehat N_{\E}(x)$ and $j\in T$,
then $x+te_j,x-te_j\in\E$ for all sufficiently small $t>0$. The defining
normal inequality gives
\begin{align}
n_j&\leq0,\qquad -n_j\leq0
\quad\Longrightarrow\quad n_j=0. \notag
\end{align}
If $j\in A(T)$, then $x+te_j\in\E$, and the same inequality gives $n_j\leq0$.
Thus $\widehat N_{\E}(x)\subseteq\mathcal N(T)$.

Conversely, let $n\in\mathcal N(T)$. All positive entries of $x$ remain positive
at nearby points $y\in\E$. Hence $T\subseteq I(y)$, and downward closure gives
$I(y)\setminus T\subseteq A(T)$. Consequently,
\begin{align}
\langle n,y-x\rangle
&=\sum_{j\in T}n_j(y_j-x_j)
  +\sum_{j\in I(y)\setminus T}n_jy_j \notag\\
&=\sum_{j\in I(y)\setminus T}n_jy_j\leq0. \notag
\end{align}
This proves $\mathcal N(T)\subseteq\widehat N_{\E}(x)$.

Next, let $n\in N_{\E}(x)$. There exist sequences satisfying
\begin{align}
x^\ell\in\E,\quad x^\ell\to x,\quad
n^\ell\in\widehat N_{\E}(x^\ell),\quad n^\ell\to n. \notag
\end{align}
Since $\mathcal F$ is finite, a subsequence has constant support $T$.
Then $I(x)\subseteq T$ and $n^\ell\in\mathcal N(T)$. Closedness of
$\mathcal N(T)$ yields
\begin{align}
n\in\mathcal N(T)
\quad\Longrightarrow\quad
N_{\E}(x)\subseteq
\bigcup_{\substack{T\in\mathcal F\\I(x)\subseteq T}}\mathcal N(T). \notag
\end{align}
For the reverse inclusion, fix $T\in\mathcal F$ with $I(x)\subseteq T$ and
$n\in\mathcal N(T)$, and set
\begin{align}
x^\ell=x+\frac{1}{\ell}\sum_{j\in T\setminus I(x)}e_j. \notag
\end{align}
Then $x^\ell\in\E$, $I(x^\ell)=T$, $x^\ell\to x$ and
$n\in\widehat N_{\E}(x^\ell)$. Therefore $n\in N_{\E}(x)$, proving the second
formula in Eq. (\ref{app:normal-formula}). Smooth addition now gives
\begin{align}
\partial J(x)
&=\nabla H(x)+\partial\delta_{\E}(x)
=\nabla H(x)+N_{\E}(x). \notag
\end{align}
The defining sequences also satisfy
$J(x^\ell)=H(x^\ell)\to H(x)=J(x)$.

(ii) The cone $\mathcal N(T)$ is polyhedral. Its linear images are
polyhedral and therefore closed \cite{rockafellar1998variational}.
Let $M=\operatorname{span}Q$. Part (i) and linearity of the orthogonal projection give
\begin{equation*}
 P_M\partial J(x)=P_M\nabla H(x)+
 \bigcup_{\substack{T\in\mathcal F\\I(x)\subseteq T}}P_M\mathcal N(T).
\end{equation*}
Each projected cone is closed. The union is finite,
so this set is nonempty and closed. If $M=\{0\}$,
the supremum in Eq. (\ref{app:attained-seminorm}) is zero and any original
subgradient attains the infimum. Suppose $M\ne\{0\}$ and define on $M$
\begin{align}
\psi(z)=\sup_{d\in Q,\ \|d\|=1}|\langle z,d\rangle|. \notag
\end{align}
Positive homogeneity of $Q$ gives
$\operatorname{span}\{d\in Q:\|d\|=1\}=M$. Hence
\begin{align}
\psi(z)=0\quad\Longrightarrow\quad z\in M\cap M^\perp=\{0\}. \notag
\end{align}
Also, $|\psi(z)-\psi(w)|\leq\|z-w\|$, $\forall z,w\in M$.
Compactness of the unit sphere in $M$ therefore yields
\begin{align}
\kappa:=\min_{\substack{z\in M\\\|z\|=1}}\psi(z)>0,\qquad
\kappa\|z\|\leq\psi(z)\leq\|z\|,\quad\forall z\in M. \notag
\end{align}
Let $z^\ell\in P_M\partial J(x)$ be a minimizing sequence for $\psi$.
The lower bound makes $\{z^\ell\}$ bounded. Closedness gives a convergent
subsequence with limit $z^*\in P_M\partial J(x)$, and continuity gives
\begin{align}
\psi(z^*)=\inf_{z\in P_M\partial J(x)}\psi(z). \notag
\end{align}
Choose $g^*\in\partial J(x)$ with $P_Mg^*=z^*$. Since $Q\subseteq M$,
\begin{align}
\sup_{d\in Q,\ \|d\|=1}|\langle g^*,d\rangle|
&=\psi(P_Mg^*)
=\inf_{g\in\partial J(x)}\sup_{d\in Q,\ \|d\|=1}|\langle g,d\rangle|. \notag
\end{align}
Thus the infimum is attained. Its value is zero if and only if this
subgradient annihilates $Q$.

(iii) For each $T\in\mathcal F$ with $I(x)\subseteq T$, take the generators $-e_j$ for $j\in A(T)$
and $\pm e_j$ for $j\notin T\cup A(T)$, project them orthogonally onto $\Heff$,
and discard zero projections. Denote the resulting columns by $u_\ell$, and put
$a=P_{\Heff}\nabla H(x)$. Enumerate every independent subset $K$ of these
columns, including the empty subset, and solve
\begin{equation}
 (U_K^TU_K)c=-U_K^Ta.
 \label{app:normal-enumeration}
\end{equation}
Retain a candidate when $c\geq0$ and
$\langle u_\ell,a+U_Kc\rangle\geq0$ for every generator. For the empty subset,
use $a$ itself and impose these inequalities. At a minimum-norm point of
$a+\operatorname{cone}\{u_\ell\}$, these inequalities follow by increasing
individual coefficients. Each positive coefficient has zero pairing, by also
decreasing that coefficient. Reduction to independent generators therefore
produces a retained candidate solving \eqref{app:normal-enumeration}.
Conversely, let $y=a+U_Kc$ be a retained candidate and
$z=a+\sum_\ell d_\ell u_\ell$, $d_\ell\geq0$, be any point of the translated cone.
The linear system and the retained inequalities give
\begin{align}
\langle y,z-y\rangle
&=\sum_\ell d_\ell\langle y,u_\ell\rangle
 -\sum_{\ell\in K}c_\ell\langle y,u_\ell\rangle\geq0, \notag\\
\|z\|^2
&=\|y\|^2+2\langle y,z-y\rangle+\|z-y\|^2
\geq\|y\|^2. \notag
\end{align}
Taking the smallest norm over all incident $T$, and lifting the coefficients to
their original signed coordinate generators, returns an attaining vector
in $\partial J(x)$.
\end{proof}

These results connect feasible reconstruction transformations with the
comparison spaces in condition~(\ref{cond:c4}) and the original certificates
in Theorem~\ref{thm:original}. In the next section, we give the actual
transformations, controls and block formulas and verify these conditions
for the original matrix and tensor models.

\section{Theoretical and numerical results for matrix and tensor factorization}
\label{sec:experiments}
In this section, we establish the scale degeneracy of the original matrix and tensor models, construct their controls and explicit block updates, and verify the conditions of the general convergence theory. We then present numerical results for three graph-regularized sparse nonnegative models and examine the retained graph action and original-variable scales.

\subsection{Theoretical results}
\label{subsec:factorization-theory}
We consider the original nonnegative CP and independent-core Tucker models, including their matrix special cases and $\ell_0$-constraints.

\subsubsection{Factorization models and scale degeneracy}
\label{subsec:application-theory}
The matrix and tensor models have reconstruction-preserving transformations that permit factor scales to diverge and graph action to vanish. We first identify these transformations and the complete original feasible fibers, including zero and singular structures.

We use the original factorization objective
\begin{equation}
 J(x)=H(x)+\delta_{\E}(x),\qquad
 H(x)=\frac12\|\mathcal X-Z(x)\|_F^2+
       \frac\alpha2\operatorname{tr}(A_N^\top LA_N),
 \label{eq:factorization-model}
\end{equation}
where $\mathcal X$ is any finite real array, $N\geq2$, $\alpha>0$, and
$L$ is nonzero, symmetric and positive semidefinite. For CP decomposition,
\begin{equation}
 Z(A)=\sum_{r=1}^{R}a_{1r}\otimes\cdots\otimes a_{Nr},\qquad
 A_i\in\R_+^{n_i\times R},\qquad \E=\prod_{i=1}^{N}E_i.
 \label{eq:cpmodel}
\end{equation}
For Tucker decomposition with an independent core,
\begin{equation}
 Z(\mathcal G,A)=\mathcal G\times_1A_1\cdots\times_NA_N,\qquad
 \E=E_G\times\prod_{i=1}^{N}E_i,
 \label{eq:tuckermodel}
\end{equation}
where $\mathcal G\in\R_+^{r_1\times\cdots\times r_N}$,
$A_i\in\R_+^{n_i\times r_i}$ and $g_{\boldsymbol k}$ denotes the core entry
with $\boldsymbol k=(k_1,\ldots,k_N)$. In both models,
\begin{equation*}
 E_i=\{A_i\geq0:\|A_i\|_0\leq s_i\},\qquad
 E_G=\{\mathcal G\geq0:\|\mathcal G\|_0\leq s_G\}.
\end{equation*}
All factor and core budgets are arbitrary legal integers, including zero;
a full entry budget recovers the corresponding nonsparse nonnegative block.
The $N=2$ reconstructions $Z=A_1A_2^\top$ and $Z=A_1GA_2^\top$ give
$\ell_0$-GSNMF and independent-core $\ell_0$-GSNMTF, respectively.
CP has $B=N$ update blocks; Tucker has $B=N+1$, ordered as
$\mathcal G,A_1,\ldots,A_N$, so factor $i$ is update block $i+1$.
We write $x_{(B)}$ for the complete block vector.
These reconstructions use the standard tensor conventions of \cite{kolda2009tensor}.

The graph term measures disagreement along the graph. For a combinatorial Laplacian $L=D-W$,
\begin{equation}
\operatorname{tr}(A_N^\top LA_N)
=\frac12\sum_{u,v}W_{uv}\norm{(A_N)_{u:}-(A_N)_{v:}}^2.
\label{eq:graphmeaning}
\end{equation}
It is zero exactly when the representation is constant on every connected component. Graph regularization is useful because it organizes sample representations through these differences \cite{cai2011graph}; eliminating all such differences removes its discriminating structure.

The reconstruction-preserving transformations identify both the graph degeneracy and the comparison geometry. In CP, for arbitrary real $\gamma_{ir}$, $i<N$, let $\mathcal D_\gamma$ multiply $a_{ir}$ by $e^{\gamma_{ir}}$ and $a_{Nr}$ by $e^{-\sum_{i<N}\gamma_{ir}}$. Its velocity map $T_x$ has generators

\begin{equation}
 d_{ir}(x)=(0,\ldots,a_{ir},\ldots,0,-a_{Nr})
 \quad(i<N),
 \label{app:cp-scale-generators}
\end{equation}

Its kernel consists exactly of $\gamma_{ir}a_{ir}=0$ and $(\sum_{i<N}\gamma_{ir})a_{Nr}=0$ in every component. Thus every nonzero rank-one component has its $N-1$ independent scale directions, and the same map includes zero components.
For Tucker, multiply $a_{ia}$ by $e^{\gamma_{ia}}$ and every original core entry $g_{\boldsymbol k}$ by $e^{-\sum_i\gamma_{ik_i}}$. The velocity and its exact logarithmic stabilizer are

\begin{equation}
 T_x\gamma=
 \left(\left\{-g_{\boldsymbol{k}}\sum_i\gamma_{ik_i}\right\}_{\boldsymbol{k}},
       \{\gamma_{ia}a_{ia}\}_{i,a}\right).
 \label{app:tucker-scale-generators}
\end{equation}

\begin{equation}
 \gamma_{ia}a_{ia}=0\ (i,a),\qquad
 g_{\boldsymbol{k}}\sum_i\gamma_{ik_i}=0\ (\boldsymbol{k}).
 \label{app:tucker-stabilizer}
\end{equation}

Hence $\dim\range T_x=\sum_i r_i-\dim\ker T_x$, equal to $\sum_i r_i$ when all factor columns are nonzero. Both transformations preserve all original signs, supports, budgets and reconstructions for positive and negative parameters. We use $\Sscale(x)=\range T_x$ throughout.

\begin{proposition}[Complete feasible reconstruction structure]\label{prop:fullstructure}
The original feasible reconstruction relation is the finite union of the
polynomial support relations below. For Tucker, the invertible core--factor
transformations and their exact infinitesimal kernel are given by
Eqs.~(\ref{app:tucker-gl})--(\ref{app:tucker-gl-velocity}); feasibility is
decided by the original nonnegativity and sparsity constraints.
\end{proposition}

\begin{proof}
Enumerate every allowed support in each original block, impose nonnegativity and zero entries outside it, and add the polynomial reconstruction equations

\begin{equation}
 \begin{aligned}
 &\sum_r\prod_i (a_{ir}(y))_{j_i}
   =\sum_r\prod_i (a_{ir}(x))_{j_i}&&\text{(CP)},\\
 &\sum_{\boldsymbol{k}}g_{\boldsymbol{k}}(y)
          \prod_i(A_i(y))_{j_ik_i}
   =\sum_{\boldsymbol{k}}g_{\boldsymbol{k}}(x)
          \prod_i(A_i(x))_{j_ik_i}&&\text{(Tucker)}
 \end{aligned}
 \label{app:fiber-polynomial}
\end{equation}

for every output multi-index. This finite union describes exactly

\begin{equation}
 \mathcal R=\{(x,y)\in \E\times \E:Z(y)=Z(x)\}.
 \label{app:true-fiber}
\end{equation}
For the fiber and tangent cone defined in Eq.~(\ref{eq:fiber}),
Definition~\ref{def:tangent-fiber} gives the explicit representation

\begin{equation}
 \begin{split}
 d\in \V(x)\quad\Longleftrightarrow\quad
 &\text{for all }\eta>0,\tau>0\text{ there are }0<t<\tau, y\in \E,\\
 &Z(y)=Z(x),\qquad\|y-x-td\|^2<\eta^2t^2.
 \end{split}
 \label{app:fiber-tangent}
\end{equation}

This representation includes zero-component and rank-changing branches.
Every $Q_i\in\operatorname{GL}(r_i)$ gives the Tucker transformation

\begin{equation}
 \Gamma_Q(x)=
 (\mathcal{G}\times_1Q_1^{-1}\cdots\times_NQ_N^{-1},
          A_1Q_1,\ldots,A_NQ_N),
 \label{app:tucker-gl}
\end{equation}

where $(\mathcal G\times_iM)_{k_1\cdots a\cdots k_N}=\sum_bM_{ab}g_{k_1\cdots b\cdots k_N}$. Each $Q_iQ_i^{-1}$ cancels in the reconstruction equations. Its infinitesimal map is

\begin{equation}
 (B_1,\ldots,B_N)\longmapsto
 \left(-\sum_i\mathcal{G}\times_iB_i,\ A_1B_1,\ldots,A_NB_N\right).
 \label{app:tucker-gl-velocity}
\end{equation}

The kernel is exactly $A_iB_i=0$ for every $i$ and $\sum_i\mathcal G\times_iB_i=0$; its image dimension is $\sum_i r_i^2$ minus this kernel dimension. A transformation is feasible exactly when all transformed entries are nonnegative and satisfy their original support budgets. The polynomial relation also includes fiber points outside any particular invertible orbit.
\end{proof}

\begin{theorem}[Original graph degeneracy]\label{app:scale-collapse}
Consider Eq.~(\ref{eq:factorization-model}) with either reconstruction in Eqs.~(\ref{eq:cpmodel})--(\ref{eq:tuckermodel}), with $\alpha>0$ and nonzero symmetric positive semidefinite $L$. At every original feasible point, criticality along all genuine reconstruction-preserving diagonal scales is equivalent to $LA_N=0$. Every full critical point of the original $J$ therefore satisfies $LA_N=0$. For a connected combinatorial Laplacian, the last factor has the form $A_N=\mathbf1c^\top$.
\end{theorem}

\begin{proof}
For CP, scale $a_{ir}(t)=e^ta_{ir}$ and $a_{Nr}(t)=e^{-t}a_{Nr}$ with $i<N$. For Tucker, use $a_{Nr}(t)=e^ta_{Nr}$ and multiply its complete core slice $k_N=r$ by $e^{-t}$. The affected reconstruction terms cancel exactly. The corresponding objective derivatives are
\begin{equation}
 \langle g,d_{ir}\rangle=-\alpha a_{Nr}^\top La_{Nr}\quad\text{(CP)},\qquad
 \langle g,d_{Nr}\rangle=\alpha a_{Nr}^\top La_{Nr}\quad\text{(Tucker)},
 \quad g\in\partial J(x).
 \label{eq:scalederivative}
\end{equation}
The column velocities have $\kappa_r=-1$ for CP and $\kappa_r=1$ for Tucker in Lemma~\ref{lem:scale-orthogonality}. Its graph conclusion gives the scale-criticality equivalence and the full-criticality implication. Finally, a connected combinatorial Laplacian has $\ker L=\operatorname{span}\{\mathbf1\}$.
\end{proof}

\begin{example}\label{ex:converse}
The condition $LA_N=0$ does not imply full criticality of the original objective.
For NMF with $A_1=0$, $A_2=(1,1)^\top$, $\mathcal{X}=(1,1)$, a legal first-factor
budget of one, and the two-node Laplacian, one has $LA_2=0$.
The feasible change $A_1=t>0$ reduces the reconstruction loss from $1$ to
$(1-t)^2$, so this point is not full critical.
\end{example}

The control must therefore prevent both scale escape and loss of graph action.

\subsubsection{Control construction, explicit updates and computational cost}
\label{subsec:control-updates}
We construct the graph control for every positive last-factor budget. Put $r_{\mathrm{last}}=R$ for CP and $r_{\mathrm{last}}=r_N$ for Tucker, $K=\min\{r_{\mathrm{last}},s_N\}$, and choose $I\subseteq[r_{\mathrm{last}}]$ with $|I|=K$. If $s_N\geq r_{\mathrm{last}}$, $I$ contains all last indices. Let $m=N-1$, $s_{ir}=\|a_{ir}\|^2$ for CP and $s_{ia}=\|a_{ia}\|^2$ for Tucker. Let $P$ project orthogonally onto $\range L$ and let $\nu=\lambda_{\min}^{+}(L)>0$. For $r\in I$, set
\begin{equation}
 q_r=\|Pa_{Nr}\|^2,\qquad s_r=s_{Nr}=\|a_{Nr}\|^2,\qquad
 h_r=\frac{s_r}{q_r^2},\qquad U_I=\{x:q_r>0\ (r\in I)\}.
 \label{eq:graphcoordinates}
\end{equation}
The graph-scale control is
\begin{equation}
 C_{\mathrm{graph}}=\epsilon\sum_{r\in I}(h_r+h_r^{-1}-2),\qquad\epsilon>0.
 \label{eq:graphcontrol}
\end{equation}
Under $a_{Nr}\mapsto ta_{Nr}$, $h_r\mapsto t^{-2}h_r$, so the two terms control both scale directions. Also $q_r\leq s_r$ implies $h_r\geq1/q_r$, preventing graph-nullspace approach. Balance the remaining scales by

\begin{equation}
b_{ir}=\frac{\norm{a_{ir}}^2-\norm{a_{N-1,r}}^2}{2},\quad i<N-1,
\qquad C_{\mathrm{bal}}=\mu\sum_{r}\sum_{i<N-1}b_{ir}^2,\quad \mu>0.
\label{eq:cpbalance}
\end{equation}

for CP, with an empty sum for NMF, and by

\begin{equation}
u_{ia}=\sum_{\boldsymbol{k}:k_i=a}g_{\boldsymbol{k}}^2,\qquad
b_{ia}=\frac{\norm{a_{ia}}^2-u_{ia}}{2},\quad i<N,
\qquad C_{\mathrm{bal}}=\mu\sum_{i<N,a}b_{ia}^2.
\label{eq:tuckerbalance}
\end{equation}

for Tucker. In CP, $x_{I^c}$ collects every factor column indexed by $r\notin I$. In Tucker it collects only the last columns indexed by $r\notin I$ and the core entries with $k_N\notin I$. Define
\begin{equation}
 \begin{aligned}
 \mathcal M_I&=\{x:x_{I^c}=0\},\\
 C=C_I&=C_{\mathrm{bal}}+C_{\mathrm{graph}}+\frac\epsilon2\|x_{I^c}\|^2
       &&\text{on }U_I,\qquad C_I=+\infty\text{ outside }U_I.
 \end{aligned}
 \label{eq:budget-control}
\end{equation}
The quadratic tail controls the inactive coordinates on the full original feasible domain. When $I$ contains all indices, this tail is empty.

\begin{proposition}[Feasible initialization and preservation under block updates]\label{prop:allbudgets}
For $s_N>0$, every original feasible input admits a finite feasible initialization in $\mathcal M_I\cap U_I$. The original global block projections preserve $\mathcal M_I$ automatically.
For $s_N=0$, the original objective is constant on $\E$, and the constant original-feasible sequence has zero length and full original criticality.
\end{proposition}
\begin{proof}
Choose $e_j\geq0$ with $Pe_j\ne0$. Retain an input already in $\mathcal M_I\cap U_I$; otherwise replace active last columns by $e_j$, set other last columns to zero, and delete the corresponding inactive first-factor columns in CP or mode-$N$ core slices in Tucker. The last factor uses $K\leq s_N$ entries, and the other operations only delete entries.
At a point of $\mathcal M_I$, inactive CP columns satisfy, for every output index $\boldsymbol j$,
\begin{equation*}
 \frac{\partial Z_{\boldsymbol j}}{\partial(a_{ir})_u}
 =\mathbf1_{j_i=u}\prod_{\ell\ne i}(a_{\ell r})_{j_\ell}=0\quad(r\notin I),
\end{equation*}
because $N\geq2$ and the product contains a zero factor. For Tucker,
\begin{equation*}
 \begin{aligned}
 \frac{\partial Z_{\boldsymbol j}}{\partial g_{\boldsymbol k}}
 &=\prod_\ell(A_\ell)_{j_\ell k_\ell}=0 &&(k_N\notin I),\\
 \frac{\partial Z_{\boldsymbol j}}{\partial(a_{Nr})_u}
 &=\mathbf1_{j_N=u}\sum_{\boldsymbol k:k_N=r}g_{\boldsymbol k}
                  \prod_{\ell<N}(A_\ell)_{j_\ell k_\ell}=0 &&(r\notin I).
 \end{aligned}
\end{equation*}
The balance derivatives on these coordinates are proportional to the zero columns or entries. The graph and quadratic-tail derivatives are zero there as well. Hence
\begin{equation*}
 x_{I^c}=0\ \Longrightarrow\ (\nabla(H+C_I)(x))_{I^c}=0
 \ \Longrightarrow\ (x_j-\nabla_j(H+C_I)(x)/c_j)_{I^c}=0.
\end{equation*}
Nonnegative hard thresholding on the complete original block only removes entries, so its output remains in $\mathcal M_I$. Induction proves invariance at all block states.
If $s_N=0$, $A_N=0$ forces $Z=0$ and the graph term to vanish. Thus $J=\|\mathcal X\|_F^2/2$ on $\E$, with $0\in\widehat\partial J(x)$ at every feasible point.
\end{proof}

\begin{theorem}[Automatic compactness of all factors and the independent core]
\label{thm:compactness}
For $s_N>0$, $\Phi=J+C_I$ is proper, lower semicontinuous and semialgebraic. Every nonempty finite sublevel $K_M=\{x:\Phi(x)\leq M\}$ is compact and contained in $U_I$. Put

\begin{equation}
 \begin{aligned}
 A_M&=2+\frac{M}{\epsilon},&
 H_M&=\frac{A_M+\sqrt{A_M^2-4}}2,&
 Q_M&=\frac{2M}{\alpha\nu},\\
 T_M&=\|\mathcal{X}\|_F+\sqrt{2M},&
 D_M&=2\sqrt{\frac{M}{\mu}},&
 \rho_M&=H_M^{-\frac{1}{2}}.
 \end{aligned}
 \label{app:level-constants}
\end{equation}

Every active last column satisfies

\begin{equation}
 H_M^{-1}\leq h_r\leq H_M,\qquad
 H_M^{-1}\leq q_r\leq Q_M,\qquad
 H_M^{-1}\leq s_{Nr}\leq H_MQ_M^2.
 \label{app:last-bounds}
\end{equation}

For CP, every active first-factor column satisfies

\begin{equation}
 \|a_{ir}\|^2\leq
 2D_M+(T_M\sqrt{H_M})^{\frac{2}{N-1}}\quad(i<N).
 \label{app:cp-bounds}
\end{equation}

For Tucker, put

\begin{equation}
 B_G=\max\left\{\sqrt{2D_M},
       \left(2^{\frac{N-1}{2}}\frac{T_M}{\rho_M}\right)^{\frac{1}{N}}\right\}.
 \label{app:core-bound}
\end{equation}

Then $g_{\boldsymbol k}\leq B_G$ for $k_N\in I$. In both models,

\begin{equation}
 \|x_{I^c}\|^2\leq \frac{2M}{\epsilon}.
 \label{eq:inactive-bound}
\end{equation}

With $\widehat B_G=\max\{B_G,\sqrt{2M/\epsilon}\}$ when $I^c\ne\varnothing$ and $\widehat B_G=B_G$ otherwise, the Tucker bounds are
\begin{equation}
 \|\mathcal G\|_F^2\leq\Big(\prod_i r_i\Big)\widehat B_G^2,\qquad
 s_{ia}\leq\Big(\prod_{j\ne i}r_j\Big)\widehat B_G^2+D_M\quad(i<N).
 \label{app:tucker-bounds}
\end{equation}
\end{theorem}
\begin{proof}
Every term in $\Phi$ is nonnegative. Hence $h_r+h_r^{-1}\leq A_M$, whose two positive endpoint roots are $H_M^{-1}$ and $H_M$. Since $q_r\leq s_r=h_rq_r^2$ and $a_{Nr}^\top La_{Nr}\geq\nu q_r$,
\begin{equation*}
 H_M^{-1}\leq h_r^{-1}\leq q_r\leq Q_M,\qquad
 H_M^{-1}\leq q_r\leq s_r=h_rq_r^2\leq H_MQ_M^2.
\end{equation*}
This proves the active last-factor bounds, including $\|a_{Nr}\|\geq\rho_M$.
The loss bound gives $\|Z\|_F\leq T_M$. Every nonnegative CP component is bounded entrywise by $Z$, so
\begin{equation}
 \prod_i\|a_{ir}\|=\|a_{1r}\otimes\cdots\otimes a_{Nr}\|_F\leq T_M.
 \label{app:component-bound}
\end{equation}
For $r\in I$, divide by $\rho_M$ to obtain $\prod_{i<N}\|a_{ir}\|\leq T_M\sqrt{H_M}$. The balance gives $|s_{ir}-s_{mr}|\leq D_M$. If $s_{mr}>D_M+(T_M\sqrt{H_M})^{2/m}$, all $m$ first squared norms exceed $(T_M\sqrt{H_M})^{2/m}$, contradicting that product bound. Thus $s_{mr}\leq D_M+(T_M\sqrt{H_M})^{2/m}$ and $s_{ir}\leq2D_M+(T_M\sqrt{H_M})^{2/m}$. For $m=1$, the product bound directly suffices.
For Tucker, $|s_{ia}-u_{ia}|\leq D_M$ and each nonnegative core component is bounded by $Z$. If $k_N\in I$ and $g_{\boldsymbol k}>\sqrt{2D_M}$, then
\begin{equation*}
 s_{ik_i}\geq u_{ik_i}-D_M\geq g_{\boldsymbol k}^2/2\quad(i<N),\qquad
 T_M\geq g_{\boldsymbol k}\prod_i\|a_{ik_i}\|
       \geq\rho_Mg_{\boldsymbol k}^{N}/2^{(N-1)/2}.
\end{equation*}
Solving gives $g_{\boldsymbol k}\leq B_G$. The tail bounds the remaining entries by Eq.~(\ref{eq:inactive-bound}); counting entries in each slice gives Eq.~(\ref{app:tucker-bounds}). The inactive CP columns and last Tucker columns are bounded by the same tail estimate.
All variables are therefore bounded. The original feasible set is a finite union of closed nonnegative coordinate cones. The finite part is continuous on $\E\cap U_I$, and $\Phi\geq\epsilon(1/q_r-2)\to+\infty$ as any active $q_r\to0$. This proves lower semicontinuity and closedness of each bounded sublevel in $U_I$. Initialization proves properness. The polynomial support descriptions and rational expressions with positive denominators prove semialgebraicity.
\end{proof}

The semialgebraic KL property in Section~\ref{sec:preliminaries} gives the local concave KL property at every $\bar x\in\E\cap U_I$. Together with Theorem~\ref{thm:compactness}, this verifies conditions~\ref{cond:c1}--\ref{cond:c2}.

Here $x$ denotes the current point before a block update,
$f=H+C_I$, and $g_j=\nabla_jf(x)$. We write $x+td_j$ for the complete vector obtained by replacing only block $j$ by $x_j+td_j$. At block $j$ of iteration $k$, these formulas
are evaluated at $(\{x_i^{k+1}\}_{i=1}^{j-1},\{x_i^k\}_{i=j}^{B})$, where $B$
is the number of update blocks, and
$c_j^k=\frac{1}{\sigma_j^k}=C_{\mathrm{safe}}(x,j)$.
Fix $\delta>0$ and $\beta=1.01$. Every formula below uses the current point. For Tucker, the block order is the independent core followed by
$A_1,\ldots,A_N$. In the computational formulas, a factor label $i$ identifies
$A_i$ and a core label $\mathcal{G}$ identifies the core. Their numerical cyclic block
indices are $i$ for CP and $i+1$ for Tucker, with the Tucker core at cyclic
index $1$.

For the original sparse nonnegative set $E_i$, nonnegative hard thresholding \cite{beck2013sparsity,bolte2014proximal} gives the projection
\begin{equation}
\mathcal H_{s_i}([V]_+)\in P_{E_i}(V)\quad(i=1,\ldots,N),
\label{eq:hardprojection}
\end{equation}
where $\mathcal H_{s_i}$ retains the $s_i$ largest entries. On a fixed support,
the nearest point is its positive part; retaining the largest positive entries
maximizes the reduction in squared distance. This proves the formula, including
ties resolved by a fixed ordering and zero or full budgets.
The same formula applies to the flattened core. At a full entry budget it reduces to $[V]_+$.

On $\mathcal M_I$, the CP control gradients are, columnwise,
\begin{equation}
 \begin{aligned}
 (\nabla_iC)_r&=\mu(s_{ir}-s_{mr})a_{ir}&& (i<m),\\
 (\nabla_mC)_r&=\mu\sum_{i<m}(s_{mr}-s_{ir})a_{mr},\\
 (\nabla_NC)_r&=\epsilon(1-h_r^{-2})
     \left(\frac{2a_{Nr}}{q_r^2}-\frac{4s_{Nr}Pa_{Nr}}{q_r^3}\right).
 \end{aligned}
 \label{app:cp-control-grad}
\end{equation}
The last-column formula is used for $r\in I$; the inactive gradient is zero. When $m=1$, the first two lines together mean that the sole first-factor control
gradient is zero. For Tucker,
\begin{equation}
 \begin{aligned}
 (\nabla_iC)_a&=\mu(s_{ia}-u_{ia})a_{ia}&& (i<N),\\
 (\nabla_{\mathcal{G}}C)_{\boldsymbol{k}}
   &=-\mu g_{\boldsymbol{k}}\sum_{i<N}(s_{ik_i}-u_{ik_i}),
 \end{aligned}
 \label{app:tucker-control-grad}
\end{equation}
and the last-column formula in \eqref{app:cp-control-grad} is unchanged.

The least-squares term is quadratic in each individual block. Its exact largest
block-Hessian eigenvalue is
\begin{equation}
 \begin{aligned}
 \ell_{\mathrm{loss},i}
   &=\lambda_{\max}\left(\mathop{\ast}_{j\ne i}(A_j^TA_j)\right)
                         &&\text{(CP factor)},\\
 \ell_{\mathrm{loss},\mathcal{G}}
   &=\prod_i\lambda_{\max}(A_i^TA_i)
                         &&\text{(Tucker core)},\\
 \ell_{\mathrm{loss},i}
   &=\lambda_{\max}(B_iB_i^T)
                         &&\text{(Tucker factor)},
 \end{aligned}
 \label{app:loss-curvature}
\end{equation}
where $\ast$ is the Hadamard product and the unfolding convention defines
$B_i$ by $Z^{(i)}=A_iB_i$. Explicitly,
$B_i=G^{(i)}(\bigotimes_{j\ne i}A_j)^T$, with the same mode order as the
unfolding. These formulas follow by differentiating the corresponding linear
least-squares map; the core Hessian is a Kronecker product of factor Gram matrices.
Finite upper bounds may also replace these eigenvalues. For a symmetric positive
semidefinite Gram matrix $M$, for example,
\begin{equation*}
 \lambda_{\max}(M)\leq
 \min\left\{\|M\|_F,\ \max_a\sum_b|M_{ab}|\right\}.
\end{equation*}

For one last column write $a=p+z$, where $p=Pa$, $z=(I-P)a$, and put
$q=\|p\|^2>0$, $v=\|z\|^2$, $s=q+v$, $\tau=\frac{q}{s}$. Direct differentiation gives
\begin{equation}
 \begin{aligned}
 \nabla^2(\frac{s}{q^2})
 &=\frac{2I}{q^2}-\frac{8(ap^T+pa^T)}{q^3}
        +\frac{24spp^T}{q^4}-\frac{4sP}{q^3},\\
 \nabla^2(\frac{q^2}{s})
 &=4\tau P-2\tau^2I+
       \frac8s(p-\tau a)(p-\tau a)^T.
 \end{aligned}
 \label{app:rational-hessians}
\end{equation}
When $v>0$, the first matrix restricted to the orthonormal basis
$\frac{p}{\sqrt q},\frac{z}{\sqrt v}$ is
\begin{equation}
 \begin{pmatrix}
 \frac{6}{q^2}+\frac{20v}{q^3}&-\frac{8\sqrt v}{q^{\frac{5}{2}}}\\
 -\frac{8\sqrt v}{q^{\frac{5}{2}}}&\frac{2}{q^2}
 \end{pmatrix}.
 \label{app:radial-hessian}
\end{equation}
Its remaining eigenvalues, where the corresponding subspaces exist, are
$-\frac{2}{q^2}-\frac{4v}{q^3}$ on $\operatorname{range}(P)\cap p^\perp$ and $\frac{2}{q^2}$ on
$\ker(P)\cap z^\perp$. At $v=0$ the actual radial eigenvalue is $\frac{6}{q^2}$;
the same upper bounds below apply by continuity. Also,
$\|p-\tau a\|^2=s\tau(1-\tau)$, so
\begin{equation}
 \lambda_{\max}(\nabla^2(\frac{q^2}{s}))
 \leq 4\tau-2\tau^2+8\tau(1-\tau)
 =12\tau-10\tau^2\leq\frac{18}{5}.
 \label{app:inverse-h-bound}
\end{equation}

Define
\begin{equation}
 \Theta=\{\frac{1}{64},\frac{1}{32},\frac{1}{16},\frac{1}{8},\frac{1}{4},\frac{1}{2}\},\qquad
 \lambda_+(a,d,b)=\frac{a+d+\sqrt{(a-d)^2+4b^2}}2.
 \label{app:theta-set}
\end{equation}
For each column $r$ and $\vartheta\in\Theta$, use the current values in $x$ to set
\begin{equation}
 \begin{aligned}
 a_{r\vartheta}&=(1-\vartheta)^2q_r,&
 b_{r\vartheta}&=(\|(I-P)a_{Nr}\|+\vartheta\sqrt{q_r})^2,\\
 U_{r\vartheta}
 &=\lambda_+\left(
   \frac6{a_{r\vartheta}^2}+\frac{20b_{r\vartheta}}{a_{r\vartheta}^3},
   \frac2{a_{r\vartheta}^2},
   \frac{8\sqrt{b_{r\vartheta}}}{a_{r\vartheta}^{\frac{5}{2}}}
   \right).
 \end{aligned}
 \label{app:column-ball-curvature}
\end{equation}
Here $a_{r\vartheta}$ and $b_{r\vartheta}$ are scalars, distinguished by their two
indices from the factor columns. Put
\begin{equation}
 \begin{aligned}
 C_{r\vartheta}=
 \max\left\{2\delta,
       \frac{\|(g_N)_r\|}{\vartheta\sqrt{q_r}},
       \ell_{\mathrm{loss},N}+\alpha\lambda_{\max}(L)+\delta
          +\epsilon(U_{r\vartheta}+\frac{18}{5})\right\},\\
 C_{\mathrm{safe}}(x,N)
       =\beta\max_{r\in I}\min_{\vartheta\in\Theta}C_{r\vartheta}
       \quad(s_N=n_Nr_{\mathrm{last}}).
 \end{aligned}
 \label{app:safe-last}
\end{equation}
For a smaller entry budget, use instead
\begin{equation}
\begin{aligned}
 t_\vartheta&=\vartheta\min_{r\in I}\sqrt{q_r},\\
 V_\vartheta&=\ell_{\mathrm{loss},N}+\alpha\lambda_{\max}(L)
       +\epsilon\max_{r\in I}(U_{r\vartheta}+\frac{18}{5}),\\
 C_{\mathrm{safe}}(x,N)&=\beta\min_{\vartheta\in\Theta}
 \max\left\{2\delta,\frac{2\|g_N\|_F}{t_\vartheta},
 V_\vartheta+\delta\right\}.
\end{aligned}
\label{eq:sparse-last-safe}
\end{equation}
The block label $N$ in \eqref{app:safe-last} means the last factor; it is the
last of $N+1$ update blocks in the Tucker model.

\begin{lemma}\label{app:safe-last-lemma}
At every generated last-factor block state, each $c\geq C_{\mathrm{safe}}(x,N)$ gives one original projection $A_N^+=\mathcal H_{s_N}([A_N-g_N/c]_+)$ whose actual update segment stays in $U_I$ and satisfies the curvature condition in Eq.~(\ref{eq:safe-segment-curvature}). Hence Lemma~\ref{lem:safe-parameter} gives the required block majorization and decrease.
\end{lemma}
\begin{proof}
At the full entry budget choose $\vartheta_r\in\argmin_{\vartheta\in\Theta}C_{r\vartheta}$ for each column. Equation~(\ref{eq:safe-convex-column-displacement}) gives
$\|(d_N)_r\|\leq\|(g_N)_r\|/c\leq\vartheta_r\sqrt{q_r}$.
At a smaller budget choose the minimizing $\vartheta$ in Eq.~(\ref{eq:sparse-last-safe}). Equation~(\ref{eq:safe-projection-displacement}) gives
$\|d_N\|_F\leq2\|g_N\|_F/c\leq t_\vartheta$, hence
$\|(d_N)_r\|\leq\vartheta\sqrt{q_r}$ for all $r\in I$; put $\vartheta_r=\vartheta$ in this case.
On the segment $a_r(t)=a_{Nr}+t(d_N)_r$,
\begin{equation*}
 \|Pa_r(t)\|\geq(1-\vartheta_r)\sqrt{q_r}>0,\qquad
 \|(I-P)a_r(t)\|\leq\|(I-P)a_{Nr}\|+\vartheta_r\sqrt{q_r}.
\end{equation*}
Write $M(q,v)$ for the radial matrix in Eq.~(\ref{app:radial-hessian}). On this segment $q\geq a_{r\vartheta_r}$ and $v\leq b_{r\vartheta_r}$, so every $\xi\in\R^2$ satisfies
\begin{align*}
 \xi^\top M(q,v)\xi
 &\leq\left(\frac6{a_{r\vartheta_r}^2}
       +\frac{20b_{r\vartheta_r}}{a_{r\vartheta_r}^3}\right)\xi_1^2
       +\frac2{a_{r\vartheta_r}^2}\xi_2^2
       +\frac{16\sqrt{b_{r\vartheta_r}}}{a_{r\vartheta_r}^{5/2}}
                     |\xi_1\xi_2|\\
 &\leq U_{r\vartheta_r}(\xi_1^2+\xi_2^2).
\end{align*}
The last inequality uses the largest eigenvalue of the symmetric matrix with these diagonal entries and positive off-diagonal entry $8\sqrt{b_{r\vartheta_r}}/a_{r\vartheta_r}^{5/2}$, as defined in Eq.~(\ref{app:column-ball-curvature}). The other range eigenvalues are nonpositive and the other kernel eigenvalues are at most $2/a_{r\vartheta_r}^2$. Hence $\nabla^2(s/q^2)\preceq U_{r\vartheta_r}I$. Equation~(\ref{app:inverse-h-bound}) adds $18/5$ for the inverse term. The control Hessians are columnwise, so adding loss and graph curvature bounds the actual segment quadratic form by $(c-\delta)\|d_N\|_F^2$. Proposition~\ref{prop:allbudgets} keeps inactive coordinates zero, making the tail and its directional derivatives zero along this segment. This verifies Eq.~(\ref{eq:safe-segment-curvature}), so Lemma~\ref{lem:safe-parameter} applies.
\end{proof}

The scalar quartic identity needed for these blocks is
\begin{equation}
 \nabla^2\left[\frac{\mu}{4}(\|a\|^2-b)^2\right]
 =\mu\big[(\|a\|^2-b)I+2aa^T\big]
 \preceq\mu(3\|a\|^2-b)I.
 \label{app:quartic-hessian}
\end{equation}
The upper bound remains valid when $a=0$ or $b=0$, and may be negative.
At radius $t>0$, the following one-sided control curvature bounds follow:
\begin{equation}
 \begin{aligned}
 \ell_{\mathrm{ctrl},i}(t)
   &=\mu\max_r\{3(\|a_{ir}\|+t)^2-s_{mr}\}
                          &&\text{(CP, }i<m\text{)},\\
 \ell_{\mathrm{ctrl},m}(t)
   &=\mu\max_r\{3(m-1)(\|a_{mr}\|+t)^2-\sum_{i<m}s_{ir}\}
                          &&\text{(CP, }m>1\text{)},\\
 \ell_{\mathrm{ctrl},1}(t)&=0
                          &&\text{(CP, }m=1\text{)},\\
 \ell_{\mathrm{ctrl},i}(t)
   &=\mu\max_a\{3(\sqrt{s_{ia}}+t)^2-u_{ia}\}
                          &&\text{(Tucker, }i<N\text{)},\\
 \ell_{\mathrm{ctrl},\mathcal{G}}(t)
   &=\mu\sum_{i<N}\max_a\{3(\sqrt{u_{ia}}+t)^2-s_{ia}\}
                          &&\text{(Tucker core)}.
 \end{aligned}
 \label{app:first-core-curvature}
\end{equation}
For the core, slices for a fixed mode are disjoint. Different modes contribute a
sum of Hessians; their upper bounds are added as in the last line.

For each such update block $j$, define the finite radius set
\begin{equation}
 \rho_j=\max\{1,\|x_j\|_F\},\qquad
 \mathcal T_j=\{1,\frac{\rho_j}{64},\frac{\rho_j}{16},\frac{\rho_j}{4},\rho_j,4\rho_j\},
 \label{app:radius-set}
\end{equation}
and set
\begin{equation}
 \begin{aligned}
 C_j(t)&=\max\left\{2\delta,\frac{2\|g_j\|_F}{t},
           \ell_{\mathrm{loss},j}+\ell_{\mathrm{ctrl},j}(t)+\delta\right\},\\
 C_{\mathrm{safe}}(x,j)&=\beta\min_{t\in\mathcal T_j}C_j(t).
 \end{aligned}
 \label{app:safe-first-core}
\end{equation}

\begin{theorem}[Explicit block updates and uniform parameter bounds]
\label{app:safe-uniform-normal}
Starting from $x_{(B)}^0\in\E\cap\mathcal M_I\cap U_I$, the displayed safe formulas use one original global projection per block, preserve $\mathcal M_I\cap U_I$, and satisfy condition~(\ref{cond:c3}). Every block state lies in $K_{\Phi(x_{(B)}^0)}$, and $2\delta\leq c_j^k\leq c_{\max}<\infty$. In particular, the original regular-normal certificates of Theorem~\ref{thm:subgradient} apply in the full original feasible set.
\end{theorem}
\begin{proof}
For a first-factor or core block, Eq.~(\ref{eq:safe-projection-displacement}) puts the actual segment within the minimizing radius in Eq.~(\ref{app:safe-first-core}). The column and slice norms there are at most their current norms plus that radius, so Eq.~(\ref{app:first-core-curvature}) bounds the control curvature. The last factor stays fixed in $U_I$. Proposition~\ref{prop:allbudgets} keeps all inactive coordinates zero along this segment, so the quadratic tail has zero directional curvature. Lemma~\ref{lem:safe-parameter} therefore verifies condition~(\ref{cond:c3}) for these blocks; Lemma~\ref{app:safe-last-lemma} verifies it for the last block. For every first or core block, $1\in\mathcal T_j$ bounds its coefficient by $\beta C_j(1)$ on the compact sublevel. For the last block, fix any $\vartheta_0\in\Theta$: Eq.~(\ref{app:safe-last}) is bounded by $\beta\max_rC_{r\vartheta_0}$ at the full entry budget, and Eq.~(\ref{eq:sparse-last-safe}) by its finite expression at $\vartheta_0$ otherwise. Bounded variables and gradients and the positive lower bound $q_r\geq H_M^{-1}$ make both bounds uniform. The finite part $H+C_I$ is smooth on $U_I$, so its gradient is locally Lipschitz. The projection is global on each original block, giving the full original regular normal through Eq.~(\ref{eq:block-regular-certificate}), including inactive coordinates.
\end{proof}

For ordinary nonnegative matrix and tensor factorization, set
$H^{(0)}(x)=\frac12\|\mathcal X-Z(x)\|_F^2$,
$J^{(0)}=H^{(0)}+\delta_{\E}$ and
$\Phi^{(0)}=J^{(0)}+C^{(0)}$, with the same original CP or Tucker reconstruction and feasible set. Every factor and the independent core retain nonnegativity and any legal $\ell_0$ budget, including zero; full entry budgets give the corresponding nonsparse models.

For CP, including NMF at $N=2$, balance each factor column against the last one:
\begin{equation}
 C^{(0)}=\frac\mu4\sum_{i<N,r}(s_{ir}-s_{Nr})^2,\qquad\mu>0.
 \label{app:ordinary-merit}
\end{equation}
For Tucker, balance every factor column against its complete core slice:
\begin{equation}
 C^{(0)}=\frac\mu4\sum_{i=1}^{N}\sum_{a=1}^{r_i}(s_{ia}-u_{ia})^2,\qquad
 s_{ia}=\|a_{ia}\|^2,\qquad
 u_{ia}=\sum_{\boldsymbol k:k_i=a}g_{\boldsymbol k}^2.
 \label{app:ordinary-tucker-control}
\end{equation}
The $N=2$ independent-core reconstruction $Z=A_1GA_2^\top$ gives ordinary nonnegative matrix tri-factorization and its $\ell_0$-constrained forms.

These controls also identify the scale directions needed to recover original criticality.

The controls above preserve the simple form of the original block
projections. The following property quantifies their additional arithmetic.
Here $\operatorname{nnz}(L)$ denotes the number of nonzero entries of $L$. The matrix-data products below use $\operatorname{nnz}(X)$ for sparse storage and $mn$ for a dense $m\times n$ array.

\begin{property}[Computational cost of the block updates]
\label{prop:computational-cost}
For a fixed number of modes, consider one cyclic sweep using the explicit
safe parameters. In addition to evaluating the original loss gradients,
their curvature bounds and the original projections, the CP controls require
$O(R\sum_i n_i)$ arithmetic operations and applications of $P$ to the graph
factor. For the independent-core Tucker controls, the corresponding cost is
\begin{equation}
 O\!\left(\sum_i n_ir_i+\prod_i r_i\right)
 \quad\text{and the application of }P\text{ to }A_N.
 \label{app:additional-control-cost}
\end{equation}
The ordinary balancing controls have the same scan costs without $P$.
For a combinatorial graph Laplacian, its connected components can be
computed once in $O(n_N+\operatorname{nnz}(L))$ operations, and each
$PA_N$ is obtained in $O(n_Nr_N)$ operations by subtracting the mean
within each component. Thus these controls preserve the per-sweep
arithmetic order of PALM~\cite{bolte2014proximal} when the same original
gradient, curvature and projection computations are used. Each block still
uses one original projection.
\end{property}
\begin{proof}
For CP, the column quantities $s_{ir}$ and $q_r$ are sums over the factor
entries. Computing all balances, their gradients and their curvature
bounds therefore costs $O(R\sum_i n_i)$ operations for fixed order.
For Tucker, the factor-column sums cost $O(\sum_i n_ir_i)$, and the
core-slice sums cost $O(N\prod_i r_i)$. The core gradient and its
curvature use these same slice quantities. This proves
Eq.~(\ref{app:additional-control-cost}) for fixed $N$.
The radius and graph-margin sets in
Eqs.~(\ref{app:radius-set}) and (\ref{app:theta-set}) have fixed
cardinalities. Their scans add only a fixed number of scalar operations
per column or slice and require no further projected candidates.

For a graph Laplacian, the kernel consists of vectors constant on each
connected component. Component sums and mean subtraction give the stated
cost of $P$. The factor scans are bounded by the corresponding Gram and
projection costs already present in the original block updates, proving
the per-sweep comparison. For a general symmetric positive semidefinite
$L$, the same accounting includes the cost of applying its range
projector; an explicitly stored dense projector costs $O(n_N^2r_N)$
per application after its construction.
\end{proof}

\subsubsection{Convergence guarantees and additional structural properties}
\label{subsec:graph-criticality}
\label{subsec:ordinary-factorization}
We verify the original comparison geometry and then apply the general convergence and rate theorems. The additional model conclusions are uniform graph noncollapse, exact balance and full original criticality on the factors without graph regularization and the independent core. For the graph control $C_I$, Theorems~\ref{thm:compactness} and \ref{app:safe-uniform-normal} establish compactness and safe updates. Effective criticality further requires comparison directions that retain the original reconstruction changes. Let $D_I$ be the differential of the coordinate map consisting of all balances and the active $h_r$, and set
\begin{equation}
 \Heff(x)=\ker D_I(x),\qquad W(x)=\range D_I(x)^\top.
 \label{eq:comparisonspaces}
\end{equation}
\begin{lemma}\label{thm:geometry}
At every $x\in\E\cap\mathcal M_I\cap U_I$,
\begin{equation}
 \R^{\dim x}=\Sscale(x)\oplus\Heff(x),\qquad\nabla C_I(x)\in W(x).
 \label{eq:directsum}
\end{equation}
\end{lemma}
\begin{proof}
For an active CP component, put $m=N-1$. If all first columns are nonzero, the kernel equations for a scale velocity are
\begin{equation*}
 s_{ir}\gamma_{ir}=s_{mr}\gamma_{mr}\quad(i<m),\qquad
 \sum_{i=1}^m\gamma_{ir}=0.
\end{equation*}
Writing the common value as $a$, the second equation is $a\sum_{i\leq m}s_{ir}^{-1}=0$, hence all $\gamma_{ir}=0$. The scale dimension and coordinate rank are both $m$: the balance gradients are independent, and $dh_r(a_{Nr})=-2h_r\ne0$ is in the separate last block.
If exactly $p<m$ first columns are nonzero, a zero first-column generator provides the pure last radial direction. Subtracting it from the surviving generators gives their $p$ pure radial directions. The zero first column makes the common balance pairing zero, so the $p$ surviving radial pairings vanish; the $h_r$ equation then makes the last radial velocity zero. The balances have rank $p$, and $h_r$ adds one. Thus $D_I$ is injective on the scale image and its rank equals $p+1=\dim\Sscale$ for this component. Inactive CP components have zero velocities and balance differentials.
For Tucker, index first-factor scales by $\mathcal I=\{(i,a):i<N\}$, put $B_{\boldsymbol k,(i,a)}=\mathbf1_{k_i=a}$, and let $B_{N,I}$ be the incidence matrix of the active last indices. With $g$ collecting all original core entries, the first balance Gram matrix is
\begin{equation}
 Q=\operatorname{diag}(\{s_{ia}\}_{\mathcal I})+B^\top\operatorname{diag}(g^2)B.
 \label{app:tucker-gram}
\end{equation}
The active $h$ gradients are independent and orthogonal to the first balance gradients, and
\begin{equation*}
 \begin{aligned}
 D_bT_x\gamma&=Q\gamma_F+B^\top\operatorname{diag}(g^2)B_{N,I}\gamma_{N,I},\\
 D_hT_x\gamma&=-2\operatorname{diag}(\{h_r\}_{r\in I})\gamma_{N,I}.
 \end{aligned}
\end{equation*}
If these vanish, $\gamma_{N,I}=0$ and $Q\gamma_F=0$. Since
\begin{equation*}
 \gamma_F^\top Q\gamma_F
 =\sum_{i<N,a}s_{ia}\gamma_{ia}^2+
   \sum_{\boldsymbol k}g_{\boldsymbol k}^2
           \left(\sum_{i<N}\gamma_{ik_i}\right)^2,
\end{equation*}
all first-factor and core velocities vanish. Inactive last scales have zero velocity because both their columns and core slices are zero. Thus $\ker(D_IT_x)=\ker T_x$, and explicitly
\begin{equation*}
 \dim\ker T_x=(|\mathcal I|-\rank Q)+(r_N-K),\qquad
 \rank T_x=\rank Q+K=\rank D_I.
\end{equation*}
This proves $\rank D_I=\dim\Sscale$ for every singular or zero first-column/core configuration.
In both models $D_I|_{\Sscale}$ is injective and $D_I(\Sscale)=\range D_I$. For every ambient $d$ there is therefore a unique $s\in\Sscale$ with $D_Is=D_Id$, giving $d=s+(d-s)$ with $d-s\in\ker D_I$. This proves the direct sum. Finally, the tail gradient is zero on $\mathcal M_I$, so
\begin{equation}
 \nabla C_I=2\mu\sum b\nabla b+
           \epsilon\sum_{r\in I}(1-h_r^{-2})\nabla h_r\in W.
 \label{eq:annihilation}
\end{equation}
\end{proof}

Let $\Pi_x$ project onto $\Heff=\ker D_I$ along $\Sscale$, and use the effective directions in Eq.~(\ref{eq:effective}) with this $\Heff$ and the original fiber cone $\V$.
The original transformations satisfy Theorem~\ref{app:fiber-preservation}: they preserve the full original feasible set and reconstruction, and Lemma~\ref{thm:geometry} gives its direct sum. Thus its tangent, fiber, image and curve identities hold for all original comparison directions. Inactive coordinates have zero scale velocities and coordinate differentials, so $\Pi_x$ is the identity on their coordinate subspace. For Tucker this also covers a feasible curve activating an inactive core entry. Indeed, the transformed curve multiplies every original core entry, including a newly activated one, by
\begin{equation*}
 g_{\boldsymbol k}(z(t))=e^{t\sum_i\gamma_{ik_i}}g_{\boldsymbol k}(y(t)),
\end{equation*}
and its factor columns by $e^{-t\gamma_{ia}}$. These factors cancel term by term, even when $k_N\notin I$; one may take $\gamma_{Nr}=0$ on inactive last indices. All original signs and supports are preserved. Hence the comparison geometry retains every original feasible activation and reconstruction change. Condition~\ref{cond:c4} now follows from Eq.~(\ref{eq:annihilation}). Closedness of $\mathcal M_I$ and compactness inside $U_I$ verify the same conditions at finite limits. Lemma~\ref{app:true-normal-formula} gives attainment of the original effective residual, and Theorem~\ref{thm:original} supplies its certificate transfer.

\begin{theorem}[Uniform graph noncollapse]
\label{thm:graph-noncollapse}
Let $s_N>0$, with $K=|I|$ and $C_I$ from Eq.~(\ref{eq:budget-control}).
For every nonempty finite controlled sublevel
$K_M=\{x:J(x)+C_I(x)\leq M\}$,
\begin{equation}
\|LA_N\|_F\geq\nu\sqrt{\frac{K}{H_M}}>0,
\qquad\forall x\in K_M.
\label{app:graph-lower}
\end{equation}
In particular, the bound holds for all intermediate block iterates and
every cluster point of CBPL with $M=\Phi(x_{(B)}^0)$.
\end{theorem}

\begin{proof}
For every $r\in I$, Eq.~(\ref{app:last-bounds}) gives $q_r\geq H_M^{-1}$, and $\|La_{Nr}\|^2\geq\nu^2q_r$. Summing these $K$ terms gives $\|LA_N\|_F^2\geq\nu^2K/H_M$. All block states and their limits remain in $K_M$.
\end{proof}

\begin{corollary}[Convergence of the original graph variables]\label{thm:global}
For either reconstruction in Eqs.~(\ref{eq:cpmodel})--(\ref{eq:tuckermodel}), every original feasible input admits the initialization in Proposition~\ref{prop:allbudgets}. For $s_N>0$, the control in Eq.~(\ref{eq:budget-control}) and the direct safe parameters satisfy conditions~\ref{cond:c1}--\ref{cond:c4}. The complete original $B$-block sequence, including the independent core, therefore satisfies Property~\ref{prop:framework} and the summability, total-variation and rate conclusions of Theorems~\ref{thm:controlled-summability}, \ref{thm:original-values}, \ref{thm:block-value-variation}, \ref{thm:original} and \ref{thm:convergence-rates}. Its finite whole-sequence limit $x^*\in\E\cap\mathcal M_I\cap U_I$ satisfies
\begin{equation}
 -\nabla C_I(x^*)\in\partial J(x^*)\cap W(x^*),\qquad
 \|LA_N^*\|_F\geq\nu\sqrt{\frac{K}{H_M}}>0,
 \label{app:noncollapsed-original-limit}
\end{equation}
where $M=\Phi(x_{(B)}^0)$. Thus the limit is effectively critical for the original objective and retains nonzero graph action. The same graph bound holds along the entire block path by Theorem~\ref{thm:graph-noncollapse}. The original-variable and original-objective rates follow from Theorem~\ref{thm:original-rates}, and the original residual rates follow from Corollary~\ref{cor:original-residual-rates}. At $s_N=0$, Proposition~\ref{prop:allbudgets} gives the constant original-feasible sequence and full original criticality.
\end{corollary}
\begin{proof}
Theorem~\ref{thm:compactness} and semialgebraicity verify conditions~\ref{cond:c1}--\ref{cond:c2}; Theorem~\ref{app:safe-uniform-normal} verifies condition~(\ref{cond:c3}); and Eqs.~(\ref{eq:directsum})--(\ref{eq:annihilation}) with preservation of the original comparison directions verify condition~(\ref{cond:c4}), also at the finite limit. The cited general results therefore apply.
\end{proof}

\begin{theorem}[Exact balance and original block criticality]
\label{thm:exact-balance-block-criticality}
Consider the sequence in Corollary~\ref{thm:global}, with any original
factor and core budgets and $s_N>0$. Let $\mathcal B_-$ contain the
factor blocks $A_i$, $i<N$, and the independent core when present. For a block $V\in\mathcal B_-$, let $E_V$ be its original constraint set: $E_{A_i}=E_i$ and $E_{\mathcal{G}}=E_G$.
Define
\begin{equation}
\begin{aligned}
 r_-(x)&=\left(\sum_{V\in\mathcal B_-}
 \operatorname{dist}\bigl(0,\nabla_VH(x)+N_{E_V}(V)\bigr)^2
 \right)^{\frac12},\\
 R_k&=\operatorname{dist}\bigl(0,\nabla(H+C_I)(x_{(B)}^k)
                  +\widehat N_{\mathcal E}(x_{(B)}^k)\bigr).
\end{aligned}
\label{app:original-block-residual}
\end{equation}
Then the following properties hold.

(i) The balance gradient and balance values are summable:
\begin{equation}
 \sum_{k=1}^{\infty}\|\nabla C_{\mathrm{bal}}(x_{(B)}^k)\|<\infty,
 \qquad
 \sum_{k=1}^{\infty}C_{\mathrm{bal}}(x_{(B)}^k)<\infty,
 \qquad C_{\mathrm{bal}}(x^*)=0.
\label{app:balance-summability}
\end{equation}
The squared differences $(s_{ir}^k-s_{N-1,r}^k)^2$ for CP and
$(s_{ia}^k-u_{ia}^k)^2$, $i<N$, for Tucker are therefore summable
over all their indices and iterations. At the limit,
\begin{equation}
\begin{aligned}
 \|a_{ir}^*\|&=\|a_{N-1,r}^*\| &&(i<N-1,\ \text{CP}),\\
 \|a_{ia}^*\|^2&=\sum_{\boldsymbol{k}:k_i=a}(g_{\boldsymbol{k}}^*)^2
                                      &&(i<N,\ \text{Tucker}).
\end{aligned}
\label{app:exact-graph-balance}
\end{equation}

(ii) The joint original full residual satisfies
\begin{equation}
\begin{aligned}
 r_-(x_{(B)}^k)&\leq R_k\leq\|p_{x^k}\|,\qquad k\geq1,\\
 \sum_{k=1}^{\infty}r_-(x_{(B)}^k)&<\infty,\qquad
 \sum_{k=1}^{\infty}R_k<\infty.
\end{aligned}
\label{app:original-block-residual-bound}
\end{equation}
At the limit, every block in $\mathcal B_-$ satisfies its original
full criticality condition:
\begin{equation}
 0\in\nabla_VH(x^*)+N_{E_V}(V^*),\qquad V\in\mathcal B_-.
\label{app:original-block-criticality}
\end{equation}

(iii) For each active last index $r\in I$, the scale error
\begin{equation}
 e_r^k=\alpha(a_{Nr}^k)^\top La_{Nr}^k
          -2\epsilon\bigl(h_r^k-(h_r^k)^{-1}\bigr)
\label{app:graph-scale-error}
\end{equation}
is absolutely summable, and
\begin{equation}
 \alpha(a_{Nr}^*)^\top La_{Nr}^*
       =2\epsilon\bigl(h_r^*-(h_r^*)^{-1}\bigr),\qquad h_r^*>1.
\label{app:graph-scale-balance}
\end{equation}
\end{theorem}

\begin{proof}
For CP, let $d_{ir}$ be Eq.~(\ref{app:cp-scale-generators}) and put $m=N-1$. The velocities $d_{ir}-d_{mr}$, $i<m$, scale the paired first columns by $e^t,e^{-t}$ and keep $A_N$ fixed. For Tucker, let $d_{ia}$ scale $a_{ia}$ by $e^t$ and its complete core slice by $e^{-t}$, with $i<N$. Direct differentiation gives
\begin{equation*}
 \nabla C_{\mathrm{bal}}=
 \begin{cases}
 \mu\sum_{i<m,r}(s_{ir}-s_{mr})(d_{ir}-d_{mr}),&\text{CP},\\
 \mu\sum_{i<N,a}(s_{ia}-u_{ia})d_{ia},&\text{Tucker}.
 \end{cases}
\end{equation*}
The original objective is invariant under these true transformations. Their velocities span $\mathcal S_{\mathrm{first}}$, so Lemma~\ref{lem:scale-orthogonality} gives $\partial J\subseteq\mathcal S_{\mathrm{first}}^\perp$. The graph-control gradient has only a last-factor block and is orthogonal to this space; the tail gradient is zero at the generated states and limit. Thus Proposition~\ref{prop:partial-criticality} applies with $C_{\mathrm{bal}}$ of degree four and $\mathcal B_-$ as stated. It proves parts (i)--(ii), including the original full block normals and regular minimum-norm certificate $R_k$. The exact balances follow from the defining sums of squares.
For (iii), use the true scale $v_r$ increasing $a_{Nr}$ by $e^t$ and decreasing $a_{mr}$ in CP, or its complete core slice in Tucker, by $e^{-t}$. Lemma~\ref{lem:scale-orthogonality} eliminates every original normal. Along this scale the graph derivative is $\alpha a_{Nr}^\top La_{Nr}$, while $h_r(t)=e^{-2t}h_r$ gives control derivative $-2\epsilon(h_r-h_r^{-1})$. Hence, for $p\in\partial\Phi(x)$,
\begin{equation*}
 \alpha a_{Nr}^\top La_{Nr}-2\epsilon(h_r-h_r^{-1})
 =\langle p-\nabla C_{\mathrm{bal}},v_r\rangle,
 \qquad |e_r|\leq\|x\|\|p\|.
\end{equation*}
The last bound uses $\|v_r\|\leq\|x\|$ and the orthogonal certificate estimate of Proposition~\ref{prop:partial-criticality}; it also holds with $R_k$ for the minimum regular certificate. Boundedness and Theorem~\ref{thm:controlled-summability} give absolute summability. At $x^*$ use $p=0$ and $\nabla C_{\mathrm{bal}}(x^*)=0$ to obtain the exact equality. Its left side is positive since $a_{Nr}^{*\top}La_{Nr}^*\geq\nu q_r^*>0$, forcing $h_r^*>1$.
\end{proof}

For the ordinary controls in Eqs.~(\ref{app:ordinary-merit}) and (\ref{app:ordinary-tucker-control}), all reconstruction-preserving scales are loss invariant.

\begin{theorem}[Finite length, full original criticality and exact balance]
\label{thm:ordinary}
\label{thm:ordinary-tucker}
For ordinary nonnegative matrix factorization, CP decomposition, Tucker decomposition and independent-core matrix tri-factorization, with or without their original $\ell_0$ constraints, the controls in Eqs.~(\ref{app:ordinary-merit}) and (\ref{app:ordinary-tucker-control}) have compact finite sublevels and the local concave KL property for every finite real data array, every choice of factor dimensions and every legal collection of budgets. From every original feasible initial point, CBPL with one explicit safe original projection per block has finite total length along the complete block path and converges to a finite $x^*\in\E$, with
\begin{equation}
 J^{(0)}(x_{(B)}^k)\to J^{(0)}(x^*),\qquad
 \operatorname{dist}(0,\partial J^{(0)}(x_{(B)}^k))\to0,\qquad
 0\in\partial J^{(0)}(x^*).
 \label{app:ordinary-tucker-limit}
\end{equation}
For both controls,
\begin{equation}
\begin{aligned}
 \sum_{k=1}^{\infty}\operatorname{dist}(0,\partial J^{(0)}(x_{(B)}^k))&<\infty,\\
 \sum_{k=1}^{\infty}\|\nabla C^{(0)}(x_{(B)}^k)\|&<\infty,\\
 \sum_{k=1}^{\infty}C^{(0)}(x_{(B)}^k)&<\infty,\\
 \sum_{k=1}^{\infty}|J^{(0)}(x_{(B)}^k)-J^{(0)}(x^*)|&<\infty.
\end{aligned}
 \label{app:ordinary-value-summability}
\end{equation}
Theorems~\ref{thm:convergence-rates}--\ref{thm:original-rates} and Corollary~\ref{cor:original-residual-rates} give the controlled, original-variable, original-objective and full residual rates. In particular, Eq.~(\ref{eq:original-best-value-rate}) applies to these homogeneous controls. Moreover, $C^{(0)}(x^*)=0$ and $\Phi^{(0)}(x^*)=J^{(0)}(x^*)$. The exact CP column and factor balances are
\begin{equation}
 \|a_{ir}^*\|=\|a_{Nr}^*\|\quad(i<N,r),\qquad
 \|A_i^*\|_F=\|A_N^*\|_F\quad(\forall i).
 \label{app:ordinary-conclusion}
\end{equation}
The exact Tucker column--slice and factor--core balances are
\begin{equation}
 \|a_{ia}^*\|^2=\sum_{\boldsymbol k:k_i=a}(g_{\boldsymbol k}^*)^2\quad(\forall i,a),\qquad
 \|A_i^*\|_F=\|\mathcal G^*\|_F\quad(\forall i).
 \label{app:ordinary-tucker-balance-summability}
\end{equation}
\end{theorem}
\begin{proof}
On $\Phi^{(0)}\leq M$, put $T_M=\|\mathcal X\|_F+\sqrt{2M}$ and $D_M=2\sqrt{M/\mu}$. Nonnegative CP components give $\prod_i\|a_{ir}\|\leq T_M$, and the control gives $|s_{ir}-s_{Nr}|\leq D_M$. If $s_{Nr}>D_M+T_M^{2/N}$, all $N$ squared norms exceed $T_M^{2/N}$, contradicting the product bound. Thus
\begin{equation}
 s_{Nr}\leq D_M+T_M^{\frac2N},\qquad
 s_{ir}\leq2D_M+T_M^{\frac2N}\quad(i<N).
 \label{app:ordinary-bounds}
\end{equation}
For Tucker, $\|Z\|_F\leq T_M$ and $|s_{ia}-u_{ia}|\leq D_M$. Choose a largest core entry $g=g_{\boldsymbol k}$. Nonnegativity gives
\begin{equation*}
 g\prod_i\|a_{ik_i}\|\leq T_M,\qquad
 s_{ik_i}\geq u_{ik_i}-D_M\geq g^2-D_M.
\end{equation*}
If $g^2>D_M+T_M^{2/(N+1)}$, both $g$ and the selected factor norms exceed $T_M^{1/(N+1)}$, contradicting the product bound. Consequently,
\begin{equation}
 \max_{\boldsymbol k}g_{\boldsymbol k}^2\leq D_M+T_M^{\frac2{N+1}},\qquad
 s_{ia}\leq D_M+\Big(\prod_{j\ne i}r_j\Big)
                 \Big(D_M+T_M^{\frac2{N+1}}\Big).
 \label{app:ordinary-tucker-bounds}
\end{equation}
These bounds include zero columns and core entries and bound every original factor and the core. The original feasible sets are closed finite unions of nonnegative coordinate cones, and both controls and losses are polynomial. Feasibility of zero, continuity and semialgebraicity therefore verify properness, compact finite sublevels and the local concave KL property.

For the parameters, use the loss curvatures in Eq.~(\ref{app:loss-curvature}) and the same radius set in Eq.~(\ref{app:radius-set}). For ordinary CP, set $m=N$ and use $A_m=A_N$ as the reference factor. The control curvatures are exactly the first two CP lines of Eq.~(\ref{app:first-core-curvature}) with this substitution: $i<m$ becomes $i<N$, $s_{mr}$ becomes $s_{Nr}$ and $m-1$ becomes $N-1$.
For ordinary Tucker, use the Tucker factor and core lines of Eq.~(\ref{app:first-core-curvature}) over all modes: replace $i<N$ by $i\leq N$ in both the factor range and the core sum, with $\sqrt{u_{ia}}=\|\mathcal G_{i,a}\|_F$ for the core slice $k_i=a$.
With $g_j=\nabla_j(H^{(0)}+C^{(0)})(x)$ and these substituted control curvatures, Eq.~(\ref{app:safe-first-core}) gives $C^{(0)}_{\mathrm{safe}}(x,j)$. The update is $x_j^+\in P_{E_j}(x_j-g_j/C^{(0)}_{\mathrm{safe}}(x,j))$, explicitly given by Eq.~(\ref{eq:hardprojection}). Lemma~\ref{lem:safe-parameter} verifies condition~(\ref{cond:c3}); the candidate $t=1$ gives uniform parameter bounds on the compact sublevel.

The CP transformations in Eq.~(\ref{app:cp-scale-generators}) preserve the original loss and every support. With $t_{ir}=s_{ir}-s_{Nr}$,
\begin{equation}
 \nabla C^{(0)}=\mu\sum_{i<N,r}t_{ir}d_{ir}\in\Sscale(x).
 \label{app:ordinary-force}
\end{equation}
For Tucker, let $d_{ia}$ scale $a_{ia}$ by $e^t$ and its entire original core slice by $e^{-t}$, for every mode $i$. These transformations preserve the reconstruction and every original support. With $t_{ia}=s_{ia}-u_{ia}$,
\begin{equation}
\begin{aligned}
 (\nabla_{A_i}C^{(0)})_a&=\mu t_{ia}a_{ia},\qquad
 (\nabla_{\mathcal G}C^{(0)})_{\boldsymbol k}=-\mu g_{\boldsymbol k}\sum_i t_{ik_i},\\
 \nabla C^{(0)}&=\mu\sum_{i,a}t_{ia}d_{ia}\in\Sscale(x).
\end{aligned}
 \label{app:ordinary-tucker-gradient}
\end{equation}
Choose
\begin{equation}
 \Heff(x)=\Sscale(x)^\perp,\qquad W(x)=\Sscale(x).
 \label{app:ordinary-comparison}
\end{equation}
The transformations above preserve the original constraints and reconstruction, so Theorem~\ref{app:fiber-preservation} gives the effective comparison directions in Eq.~(\ref{eq:effective}). Equations~(\ref{app:ordinary-force})--(\ref{app:ordinary-tucker-gradient}) give $\nabla C^{(0)}(x)\in W(x)$, verifying condition~(\ref{cond:c4}). When $\V(x)=\{0\}$, the inclusion $\Sscale(x)\subseteq\V(x)$ gives $\Sscale(x)=\{0\}$ and the full-space convention applies. Theorem~\ref{thm:controlled-global} now gives finite length and whole-sequence convergence.

Lemma~\ref{lem:scale-orthogonality} gives $\partial J^{(0)}\subseteq\Sscale(x)^\perp$ for all original limiting normals, including singular and zero configurations and unsaturated supports. Thus the original subgradient and control gradient are orthogonal, and Eq.~(\ref{eq:full-pythagoras}) bounds both by the controlled certificate. Theorem~\ref{thm:full} gives full original criticality and summability of the full original residual and control gradient. Both controls are nonnegative and homogeneous of degree four, so Corollary~\ref{cor:homogeneous-control} gives the remaining summability assertions and $C^{(0)}(x^*)=0$. The defining squares give the displayed column balances; summing over columns gives the factor and core norm equalities.
\end{proof}

These constructions give original-variable finite length and full criticality without preassigned boundedness or original-objective coercivity. For related problems, the same results apply through their original reconstruction-preserving transformations, a balancing control with compact sublevels, and the block and KL conditions established in Section~\ref{sec:convergence}. Theorem~\ref{thm:full} and Corollary~\ref{cor:homogeneous-control} then give full original criticality and the corresponding summability conclusions, and Theorem~\ref{thm:original-rates} gives the original-variable and original-objective rates. The following applications use the graph controls and explicit block parameters above.

\subsection{Numerical experiments}
\label{subsec:exp-commonsettings}
All experiments are carried out in MATLAB R2021b on a computer with an Intel Core i7-10850H CPU at 2.70 GHz and 16 GB RAM, using one computational thread. The code is available at \url{https://github.com/Weifeng-Yang/CBPL}.

We use the direct safe parameters in Eqs.~(\ref{app:safe-first-core}) and (\ref{app:safe-last}) with $\beta=1.01$, taking the computable upper bound $\norm{L}_1$ for $\lambda_{\max}(L)$. Each block has one projected update, with $\sigma_j^k=\frac{1}{c_j^k}$ in Eq.~(\ref{e12}).

Each input is scaled to Frobenius norm 10, with its signs retained. We construct a symmetric five-nearest-neighbor graph on the samples with weights $\exp(-\frac{d_{uv}^2}{\tau})$, where $\tau$ is the median positive squared distance among the selected neighbor edges; the larger directed weight is retained when symmetrizing. Its Laplacian $L$ is scaled by the maximum weighted degree. The projection $P$ subtracts the mean within each actual connected component. We fix $\alpha=\mu=\epsilon=0.01$, $\delta=10^{-4}$ and use the finite sets in Eqs.~(\ref{app:theta-set}) and (\ref{app:radius-set}). The first factors and independent core use element budgets of $30\%$ of their entries, rounded down. The last factor uses its full entry budget, $s_N=n_Nr_{\mathrm{last}}$. A uniform random feasible initialization is scaled so that nonzero first-factor columns have unit norm, last-factor columns have $h_r=1$, and the independent core has unit Frobenius norm. The six fixed seeds are $20261026$, $20261028$, $20261029$, $20261030$, $20261032$ and $20261034$, in the order of Table~\ref{tab:results}.

Each model--data pair is run once from its fixed feasible initialization for 60 seconds. The curves retain the recorded iteration times; only the endpoint at 60 seconds is interpolated between its two neighboring observations.

We evaluate the first-order stationarity residual block by block. At the
complete iterate $x_{(B)}^k$, let $g_i^k=\nabla_i(H+C)(x_{(B)}^k)$.
For each factor or core block $x_i$, write $s$ for its prescribed entry
budget; a block without a sparsity restriction uses its full entry count.
The original normal-cone formula in Eq.~(\ref{app:normal-formula}) gives,
for every entry $j$,
\begin{equation}
 (w_i^k)_j=
 \begin{cases}
 (g_i^k)_j,&(x_i^k)_j>0,\\
 \min\{(g_i^k)_j,0\},&(x_i^k)_j=0,\quad\|x_i^k\|_0<s,\\
 0,&(x_i^k)_j=0,\quad\|x_i^k\|_0=s.
 \end{cases}
 \label{eq:numerical-coordinate-certificate}
\end{equation}
Thus positive entries retain their gradient components, and zero entries
retain exactly the components that cannot be canceled by an original
regular normal. Concatenating these block arrays gives $w^k$ and
\begin{equation}
\begin{aligned}
 R_k&=\left(\sum_{i=1}^{B}\|w_i^k\|_F^2\right)^{\frac12}\\
 &=\dist\big(0,\nabla(H+C)(x_{(B)}^k)+\widehat N_{\E}(x_{(B)}^k)\big).
\end{aligned}
\label{eq:numericalcertificate}
\end{equation}
Here $w^k\in\widehat\partial\Phi(x_{(B)}^k)$ is the minimum-norm regular
certificate. The equality $R_k=0$ is the constrained first-order
stationarity condition for $J+C$.
The constructed certificate in Theorem~\ref{thm:subgradient} is regular,
so $R_k\leq\|p_{x^k}\|$. Applying Eq.~(\ref{eq:transfer}) to $w^k$ gives
\begin{equation}
\reff(x_{(B)}^k)\leq R_k\leq\|p_{x^k}\|\longrightarrow0,
\qquad k\geq1.
\label{eq:numerical-original-residual}
\end{equation}
Thus $R_k$ also bounds the effective residual of the original objective.
Theorem~\ref{thm:exact-balance-block-criticality} further shows that it
bounds the joint original full residual of all factors without graph
regularization and the independent core. We also record $\norm{LA_N^k}_F$, all factor and core norms,
the displacement $\norm{x_{(B)}^{k+1}-x_{(B)}^k}$ and the accumulated path length
$\sum_{j<k}\norm{x_{(B)}^{j+1}-x_{(B)}^j}$.

\subsubsection{Graph-regularized sparse nonnegative matrix factorization with $\ell_0$-constraints ($\ell_0$-GSNMF)}

Nonnegative matrix factorization represents each sample by a nonnegative combination of basis vectors, and $\ell_0$-constraints directly control the number of nonzero entries in the factors. Graph regularization couples the sample representations through $A_2$ so that neighboring samples have similar representations. Theorem~\ref{app:scale-collapse} identifies the graph scale degeneracy of this model. The $\ell_0$-GSNMF model is
\begin{equation}
\begin{aligned}
\min_{A_1,A_2}\quad &\frac12\norm{X-A_1A_2^\top}_F^2+
\frac{\alpha}{2}\operatorname{tr}(A_2^\top LA_2),\\
\text{s.t.}\quad & A_1\geq0,\quad \norm{A_1}_0\leq s_1,\quad A_2\geq0,\quad \norm{A_2}_0\leq s_2,
\end{aligned}
\label{eq:exp-gsnmf}
\end{equation}
where $X\in\R^{m\times n}$, $A_1\in\R^{m\times R}$, $A_2\in\R^{n\times R}$ and $L\in\R^{n\times n}$.

For Eq.~(\ref{e11}), take $H$ as the objective in Eq.~(\ref{eq:exp-gsnmf}) and $F_i=\delta_{E_i}$ for the two displayed constraint sets. The block gradients are
\begin{equation}
\begin{aligned}
\nabla_1H&=(A_1A_2^\top-X)A_2,\\
\nabla_2H&=(A_1A_2^\top-X)^\top A_1+\alpha LA_2.
\end{aligned}
\label{eq:exp-nmf-grad}
\end{equation}
Use $C=C_I$ from Proposition~\ref{prop:allbudgets} with $N=2$. The CP formulas in Eqs.~(\ref{app:cp-control-grad}), (\ref{app:loss-curvature}) and (\ref{app:first-core-curvature}) then give the control gradients and curvatures; the sole first-factor control gradient and curvature are zero. Set $g_j=\nabla_j(H+C)(x)$ at the current intermediate block iterate $x$. Choose $c_1=C_{\mathrm{safe}}(x,1)$ from Eq.~(\ref{app:safe-first-core}) and $c_2=C_{\mathrm{safe}}(x,2)$ from Eq.~(\ref{app:safe-last}) or (\ref{eq:sparse-last-safe}), according to its budget. In the order $A_1,A_2$, the updates are
\begin{equation}
A_1^+=\mathcal H_{s_1}([A_1-\frac{g_1}{c_1}]_+),\qquad
A_2^+=\mathcal H_{s_2}([A_2-\frac{g_2}{c_2}]_+).
\label{eq:exp-nmf-update}
\end{equation}
Here $\mathcal H_s$ retains the $s$ largest entries, as in Eq.~(\ref{eq:hardprojection}).

For this graph implementation, the arithmetic cost of one sweep is
\begin{equation}
 \begin{aligned}
 O\!\big(&\operatorname{nnz}(X)R+(m+n)R^2+R^3+\operatorname{nnz}(L)R\\
 &+mR\log(1+mR)+nR\log(1+nR)\big).
 \end{aligned}
 \label{eq:exp-nmf-cost}
\end{equation}
The data products in Eq.~(\ref{eq:exp-nmf-grad}) give the first term,
the Gram products and curvature eigenvalues give the next two terms,
and graph multiplication gives the fourth. The last two terms bound
the hard-thresholding projections by sorting. Property~\ref{prop:computational-cost}
accounts for the control and safe parameters within the same order.

We use the complete lp\_ship12l matrix ($1151\times5533$) from the SuiteSparse Matrix Collection \cite{davis2011suitesparse} and the complete BASEHOCK matrix ($1993\times4862$) from the feature-selection data repository.\footnote{\url{https://jundongl.github.io/scikit-feature/datasets.html}} The ranks are $R=20$ and $R=50$, respectively.

\begin{figure}[!htbp]
\centering
\includegraphics[width=\linewidth,height=0.6\textheight,keepaspectratio]{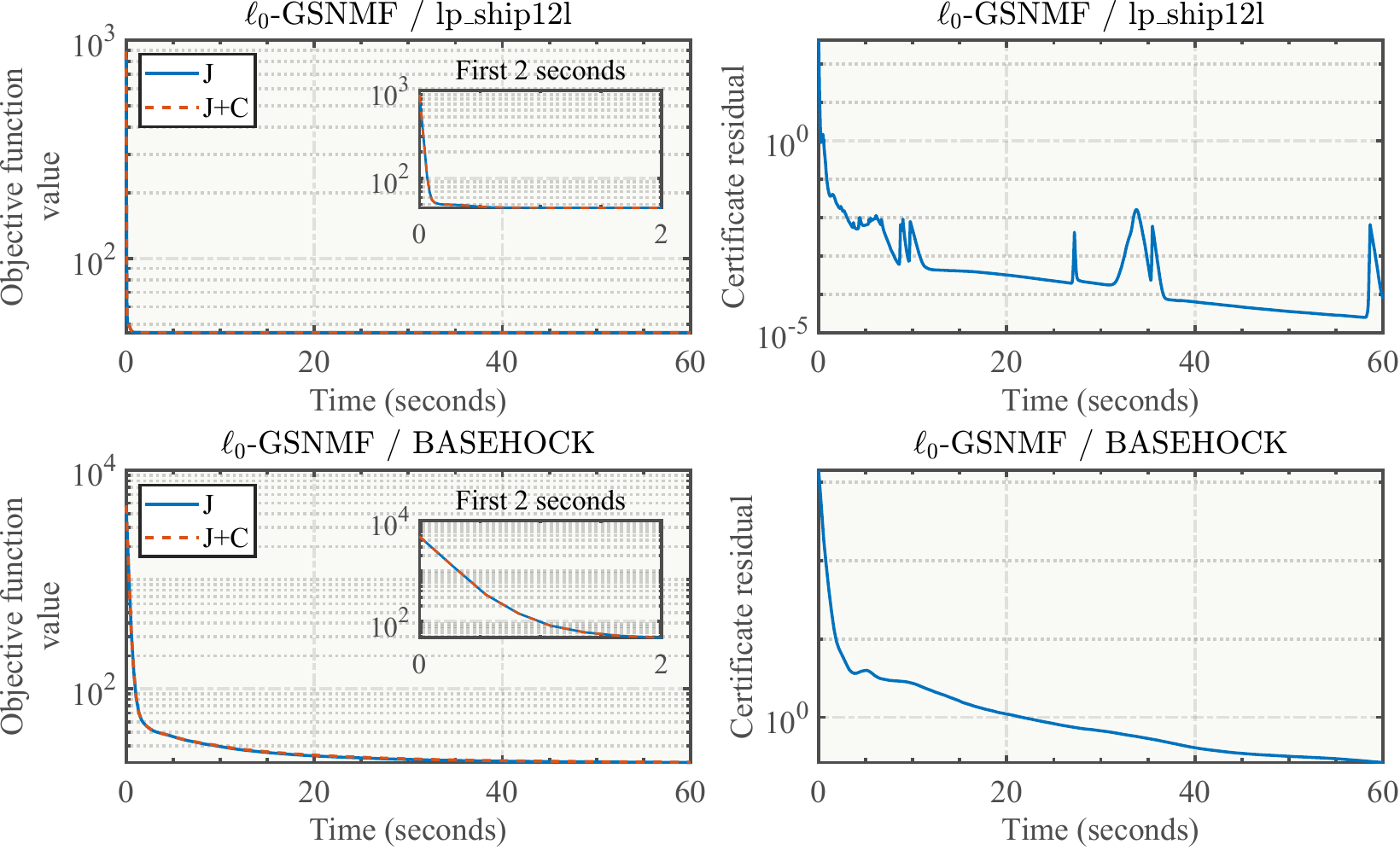}
\caption{CBPL for $\ell_0$-GSNMF on lp\_ship12l and BASEHOCK, one dataset per row. Left: $J$ and $J+C$, with an inset showing the first two seconds. Right: the certificate $R_k$ on a logarithmic vertical axis.}
\label{fig:convergence}
\end{figure}

Figure~\ref{fig:convergence} shows the decrease of $J+C$ and the stabilization of $J$ on both matrices. The right-hand curves bound the effective residual of the original objective by Eq.~(\ref{eq:numerical-original-residual}) and the full original block residual of $A_1$ by Theorem~\ref{thm:exact-balance-block-criticality}. Their decrease, together with the positive graph actions in Table~\ref{tab:results}, shows first-order progress for the original matrix model while retaining graph action in $A_2$.

\subsubsection{Graph-regularized sparse nonnegative CP decomposition with $\ell_0$-constraints ($\ell_0$-GSNCP)}

Nonnegative CP decomposition extracts latent features from multi-way data through a sum of nonnegative rank-one components. The $\ell_0$-constraints control sparsity in every factor, and graph regularization acts on the sample representations in $A_N$. Its reconstruction-preserving scales and graph degeneracy are established in Section~\ref{subsec:application-theory}. The $\ell_0$-GSNCP model is
\begin{equation}
\begin{aligned}
\min_{\{A_i\}_{i=1}^{N}}\quad &
\frac12\left\|\mathcal{X}-\sum_{r=1}^{R}a_{1r}\otimes\cdots\otimes a_{Nr}\right\|_F^2+
\frac{\alpha}{2}\operatorname{tr}(A_N^\top LA_N),\\
\text{s.t.}\quad & A_i\geq0,\quad \norm{A_i}_0\leq s_i\quad (i=1,\ldots,N),
\end{aligned}
\label{eq:exp-gsncp}
\end{equation}
where $\mathcal{X}\in\R^{n_1\times\cdots\times n_N}$, $A_i\in\R^{n_i\times R}$, $a_{ir}$ is column $r$ of $A_i$, and $L\in\R^{n_N\times n_N}$.

For Eq.~(\ref{e11}), take $H$ as the objective in Eq.~(\ref{eq:exp-gsncp}) and $F_i=\delta_{E_i}$ for its mode-$i$ constraints.
Let $X^{(i)}$ be the mode-$i$ unfolding and let $B_i$ be the Khatri--Rao product of all factors except $A_i$, ordered so that $Z^{(i)}=A_iB_i^\top$. Then $B_i\in\R^{(\prod_{j\ne i}n_j)\times R}$ and
\begin{equation}
\nabla_iH=A_i(B_i^\top B_i)-X^{(i)}B_i+
\begin{cases}0,&i<N,\\ \alpha LA_N,&i=N.\end{cases}
\label{eq:exp-cp-grad}
\end{equation}
Use the CP control $C=C_I$ from Proposition~\ref{prop:allbudgets}. Its gradients are given in Eq.~(\ref{app:cp-control-grad}), and the loss and control curvatures are given in Eqs.~(\ref{app:loss-curvature}) and (\ref{app:first-core-curvature}). With $g_i=\nabla_i(H+C)(x)$ at the current intermediate block iterate, choose $c_i=C_{\mathrm{safe}}(x,i)$ from Eq.~(\ref{app:safe-first-core}) for $i<N$ and from Eq.~(\ref{app:safe-last}) or (\ref{eq:sparse-last-safe}) for $i=N$, according to its budget. The cyclic updates are
\begin{equation}
A_i^+=\mathcal H_{s_i}([A_i-\frac{g_i}{c_i}]_+)\quad(i=1,\ldots,N).
\label{eq:exp-cp-update}
\end{equation}
The factors are updated in the order $A_1,\ldots,A_N$, with each quantity evaluated at the corresponding intermediate block iterate.

For fixed tensor order $N$, one sweep with dense tensor contractions costs
\begin{equation}
 O\!\left(NR\prod_i n_i+NR^2\sum_i n_i+NR^3
 +\operatorname{nnz}(L)R
 +\sum_i n_iR\log(1+n_iR)\right).
 \label{eq:exp-cp-cost}
\end{equation}
The unfolded-data products give $NR\prod_i n_i$; the Gram products,
factor products and curvature eigenvalues give the next two terms.
Graph multiplication and hard thresholding give the remaining terms.
By Property~\ref{prop:computational-cost}, the control and direct safe
parameters preserve this per-sweep order.

We use BreastMNIST \cite{yang2023medmnist,aldhabyani2020breast} and microPNW \cite{ni2023pnw}. BreastMNIST uses 100 fixed images, with $8\times8$ pixel block averages of the $224\times224$ source images, giving $28\times28\times100$. microPNW retains 100 frames, 20 uniformly selected bins, three channels and all 50 records, giving $100\times20\times3\times50$. The sample mode is last in each tensor, and both ranks are $R=20$.

\begin{figure}[!htbp]
\centering
\includegraphics[width=\linewidth,height=0.6\textheight,keepaspectratio]{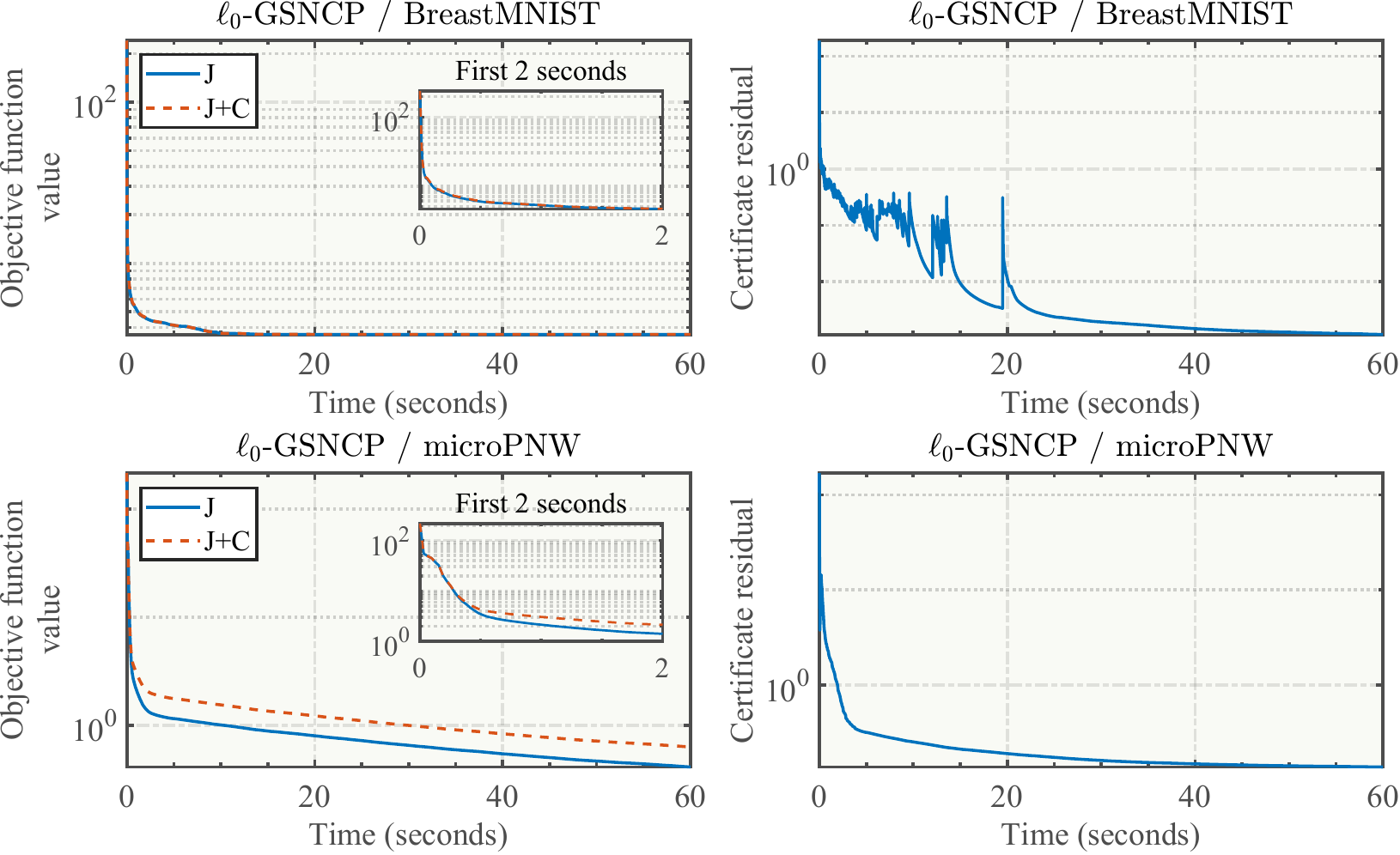}
\caption{CBPL for $\ell_0$-GSNCP on BreastMNIST and microPNW, one dataset per row. Left: $J$ and $J+C$, with an inset showing the first two seconds. Right: the certificate $R_k$ on a logarithmic vertical axis.}
\label{fig:convcp}
\end{figure}

Figure~\ref{fig:convcp} shows decreases in $J+C$ and $R_k$ on both tensors. Here $R_k$ bounds the original effective residual in Eq.~(\ref{eq:numerical-original-residual}) and the joint original full block residual of all factors $A_i$, $i<N$, by Theorem~\ref{thm:exact-balance-block-criticality}. The positive graph actions in Table~\ref{tab:results} show that the last factor retains graph action as these residual bounds decrease.

\subsubsection{Graph-regularized sparse nonnegative independent-core matrix tri-factorization with $\ell_0$-constraints ($\ell_0$-GSNMTF)}

Nonnegative matrix tri-factorization uses an independent core matrix to describe interactions between two sets of latent factors. The $\ell_0$-constraints control sparsity in the factors and the core, and graph regularization couples the sample representations in $A_2$. The factor--core scale structure is described in Section~\ref{subsec:application-theory}. We apply CBPL with the core as an original block to solve the $\ell_0$-GSNMTF model
\begin{equation}
\begin{aligned}
\min_{G,A_1,A_2}\quad &\frac12\norm{X-A_1GA_2^\top}_F^2+
\frac{\alpha}{2}\operatorname{tr}(A_2^\top LA_2),\\
\text{s.t.}\quad &G\geq0,\quad\norm{G}_0\leq s_G,\\
&A_i\geq0,\quad\norm{A_i}_0\leq s_i\quad(i=1,2),
\end{aligned}
\label{eq:exp-gsnmtf}
\end{equation}
where $X\in\R^{m\times n}$, $G\in\R^{r_1\times r_2}$, $A_1\in\R^{m\times r_1}$, $A_2\in\R^{n\times r_2}$ and $L\in\R^{n\times n}$. This experiment uses the two-mode independent-core Tucker reconstruction in Eq.~(\ref{eq:tuckermodel}).

For Eq.~(\ref{e11}), take $H$ as the objective in Eq.~(\ref{eq:exp-gsnmtf}), with $F_G=\delta_{E_G}$ and $F_i=\delta_{E_i}$ for its core and factor constraints. The three block gradients are
\begin{equation}
\begin{aligned}
\nabla_GH&=(A_1^\top A_1)G(A_2^\top A_2)-A_1^\top XA_2,\\
\nabla_1H&=A_1G(A_2^\top A_2)G^\top-XA_2G^\top,\\
\nabla_2H&=A_2G^\top(A_1^\top A_1)G-X^\top A_1G+\alpha LA_2.
\end{aligned}
\label{eq:exp-nmtf-grad}
\end{equation}
Use the independent-core control $C=C_I$ in Proposition~\ref{prop:allbudgets} with $N=2$ and $\mathcal G=G$. In Eq.~(\ref{app:tucker-control-grad}), $s_{1a}=\norm{a_{1a}}^2$ and $u_{1a}=\sum_bG_{ab}^2$, so each mode-$1$ core slice is a row of $G$. The last-factor gradient is given in Eq.~(\ref{app:cp-control-grad}). For the loss curvatures in Eq.~(\ref{app:loss-curvature}), the two factor unfolding matrices are $B_1=GA_2^\top$ and $B_2=G^\top A_1^\top$. Use the Tucker control curvatures in Eq.~(\ref{app:first-core-curvature}) and set $g_j=\nabla_j(H+C)(x)$ at the current intermediate block iterate. Equation~(\ref{app:safe-first-core}) gives $c_G=C_{\mathrm{safe}}(x,G)$ and $c_1=C_{\mathrm{safe}}(x,1)$; Eq.~(\ref{app:safe-last}) or (\ref{eq:sparse-last-safe}) gives $c_2=C_{\mathrm{safe}}(x,2)$ according to its budget. In the order $G,A_1,A_2$, the explicit updates are
\begin{equation}
\begin{aligned}
G^+&=\mathcal H_{s_G}([G-\frac{g_G}{c_G}]_+),\\
A_1^+&=\mathcal H_{s_1}([A_1-\frac{g_1}{c_1}]_+),\qquad
A_2^+=\mathcal H_{s_2}([A_2-\frac{g_2}{c_2}]_+),
\end{aligned}
\label{eq:exp-nmtf-update}
\end{equation}
Hard thresholding of $G$ acts on its entries, as in Eq.~(\ref{eq:hardprojection}).

One sweep of the three original blocks has arithmetic cost
\begin{equation}
 \begin{aligned}
 O\!\big(&\operatorname{nnz}(X)(r_1+r_2)+mr_1^2+nr_2^2
 +(m+n)r_1r_2\\
 &+r_1r_2(r_1+r_2)+r_1^3+r_2^3+\operatorname{nnz}(L)r_2\\
 &+mr_1\log(1+mr_1)+nr_2\log(1+nr_2)
 +r_1r_2\log(1+r_1r_2)\big).
 \end{aligned}
 \label{eq:exp-nmtf-cost}
\end{equation}
The products in Eq.~(\ref{eq:exp-nmtf-grad}) give the data, Gram and
core-product terms; curvature eigenvalues give $r_1^3+r_2^3$.
The remaining terms count graph multiplication and the three
hard-thresholding projections. The control scans cost
$O(mr_1+nr_2+r_1r_2)$ by Property~\ref{prop:computational-cost},
so the independent-core control preserves the original per-sweep order.

We use the same complete lp\_ship12l and BASEHOCK matrices specified above, with core sizes $r_1\times r_2=20\times20$ and $30\times20$, respectively.

\begin{figure}[!htbp]
\centering
\includegraphics[width=\linewidth,height=0.6\textheight,keepaspectratio]{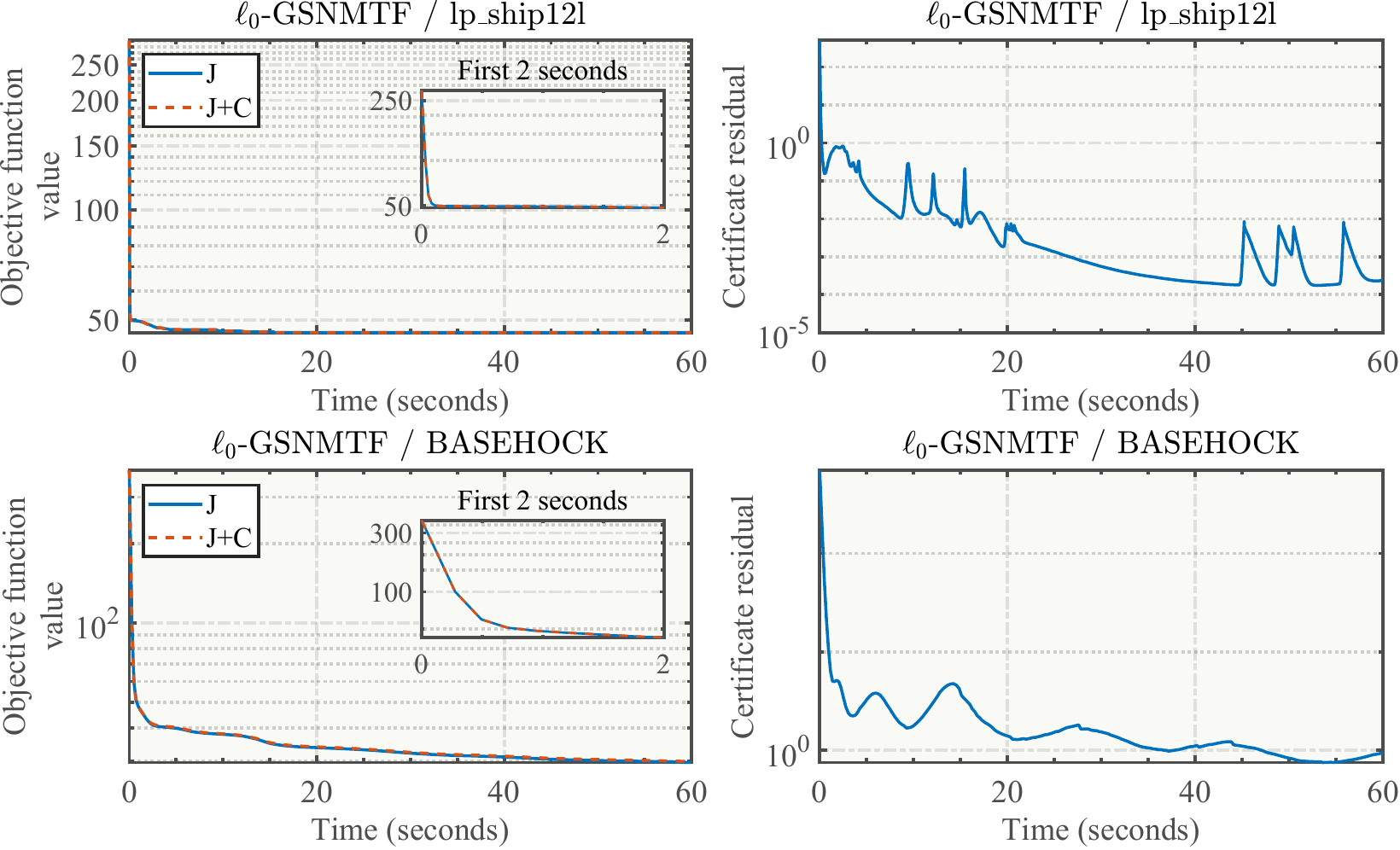}
\caption{CBPL for $\ell_0$-GSNMTF on lp\_ship12l and BASEHOCK, one dataset per row. Left: $J$ and $J+C$, with an inset showing the first two seconds. Right: the certificate $R_k$ on a logarithmic vertical axis.}
\label{fig:convnmtf}
\end{figure}

Figure~\ref{fig:convnmtf} shows the decrease of $J+C$ and the reduction of $R_k$ with the independent core $G$ updated as an original block. By Theorem~\ref{thm:exact-balance-block-criticality}, $R_k$ bounds the joint original full block residual of $A_1$ and $G$; Eq.~(\ref{eq:numerical-original-residual}) also bounds the effective residual of the whole original model. Thus the stationarity measure includes the independent core, and the positive graph actions in Table~\ref{tab:results} show that graph action in $A_2$ is retained.

\subsubsection{Convergence and scale control}
Table~\ref{tab:results} reports all six runs. Their initial certificates range from approximately $1.69\times10^2$ to $1.43\times10^3$, and the reductions range from approximately $2.9$ to $6.7$ orders of magnitude.

\begin{table}[H]
\centering
\caption{Results at 60 seconds. $R_k$ is the stationarity residual in Eq.~(\ref{eq:numericalcertificate}), and $\|LA_N\|_F$ is the graph action. Each row is one run of CBPL.}
\label{tab:results}
\small
\setlength{\tabcolsep}{4pt}
\resizebox{\linewidth}{!}{%
\begin{tabular}{lcrrrrr}
\toprule
Data & Rank/core & Iterations & $J$ & $J+C$ & $R_k$ & $\|LA_N\|_F$\\
\midrule
\multicolumn{7}{c}{$\ell_0$-GSNMF}\\
lp\_ship12l & $20$ & 4382 & 4.574e+01 & 4.574e+01 & 7.635e-05 & 1.547e-02 \\
BASEHOCK & $50$ & 200 & 2.118e+01 & 2.128e+01 & 2.637e-01 & 3.698e-02 \\
\midrule
\multicolumn{7}{c}{$\ell_0$-GSNCP}\\
BreastMNIST & $20$ & 18604 & 3.618e+00 & 3.620e+00 & 1.149e-03 & 1.182e+00 \\
microPNW & $20$ & 1648 & 4.091e-01 & 6.274e-01 & 1.373e-01 & 9.245e-01 \\
\midrule
\multicolumn{7}{c}{$\ell_0$-GSNMTF}\\
lp\_ship12l & $20\times 20$ & 3433 & 4.615e+01 & 4.615e+01 & 2.368e-04 & 1.689e-02 \\
BASEHOCK & $30\times 20$ & 244 & 2.953e+01 & 2.977e+01 & 9.426e-01 & 5.150e-02 \\
\bottomrule
\end{tabular}}
\end{table}

\begingroup
\raggedbottom
The scale-path experiment shows how the original graph term vanishes while the reconstruction is unchanged, and how the control opposes this scale escape. The CBPL trajectories then show the retained graph action and the scales of all factors and the independent core during optimization.

Figure~\ref{fig:scalingpath} evaluates a scale path through the $\ell_0$-GSNMF reconstruction, using the original objective and the proposed control. Its reconstruction is exactly unchanged. The graph term decreases by a factor $t^2$, so the original objective rewards this loss of graph action. At the same time, one factor grows as $\frac{1}{t}$ and the other shrinks as $t$. The function $J+C$ increases along the collapsing end of this path, displaying the mechanism that prevents the original scale escape.

\begin{figure}[H]
\centering
\includegraphics[width=\linewidth]{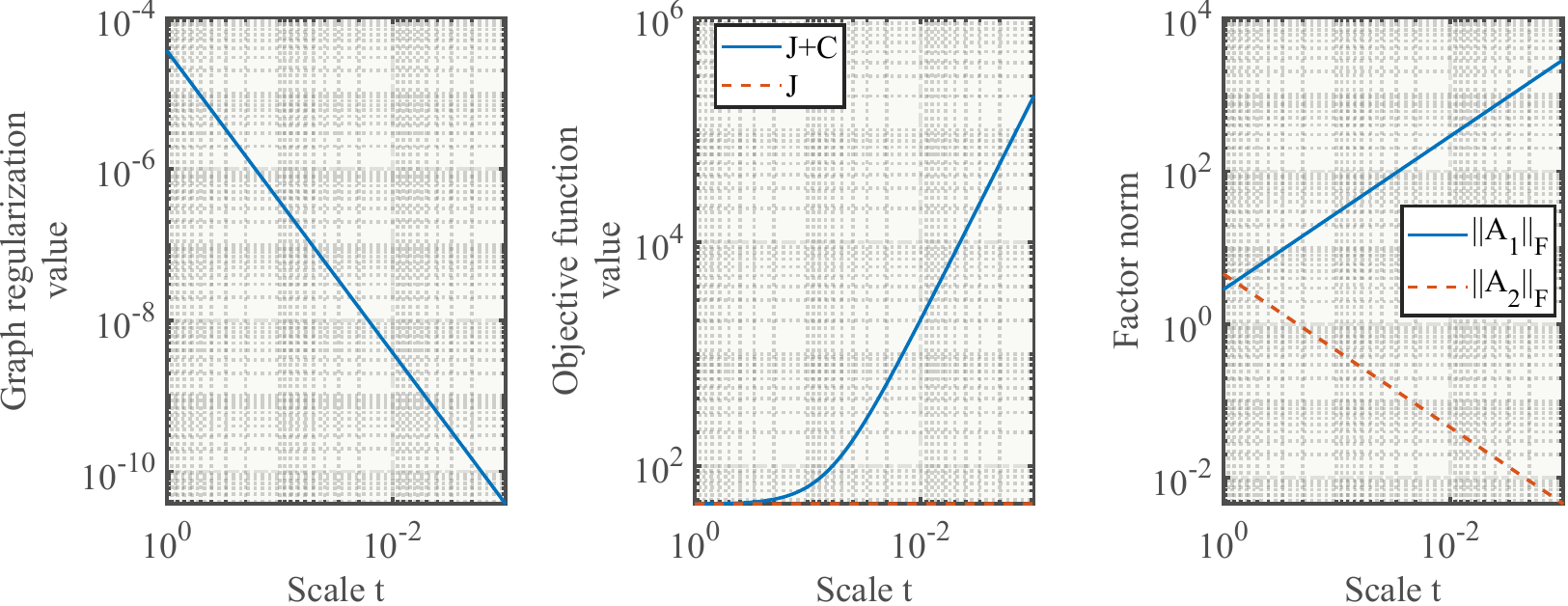}
\caption{The reconstruction-preserving $\ell_0$-GSNMF path $A_1(t)=\frac{A_1}{t}$, $A_2(t)=tA_2$, with $t$ decreasing from $1$ to $10^{-3}$. Left: the graph regularization term. Middle: the values of $J$ and $J+C$. Right: the factor norms.}
\label{fig:scalingpath}
\end{figure}

\begin{figure}[H]
\centering
\includegraphics[width=\linewidth]{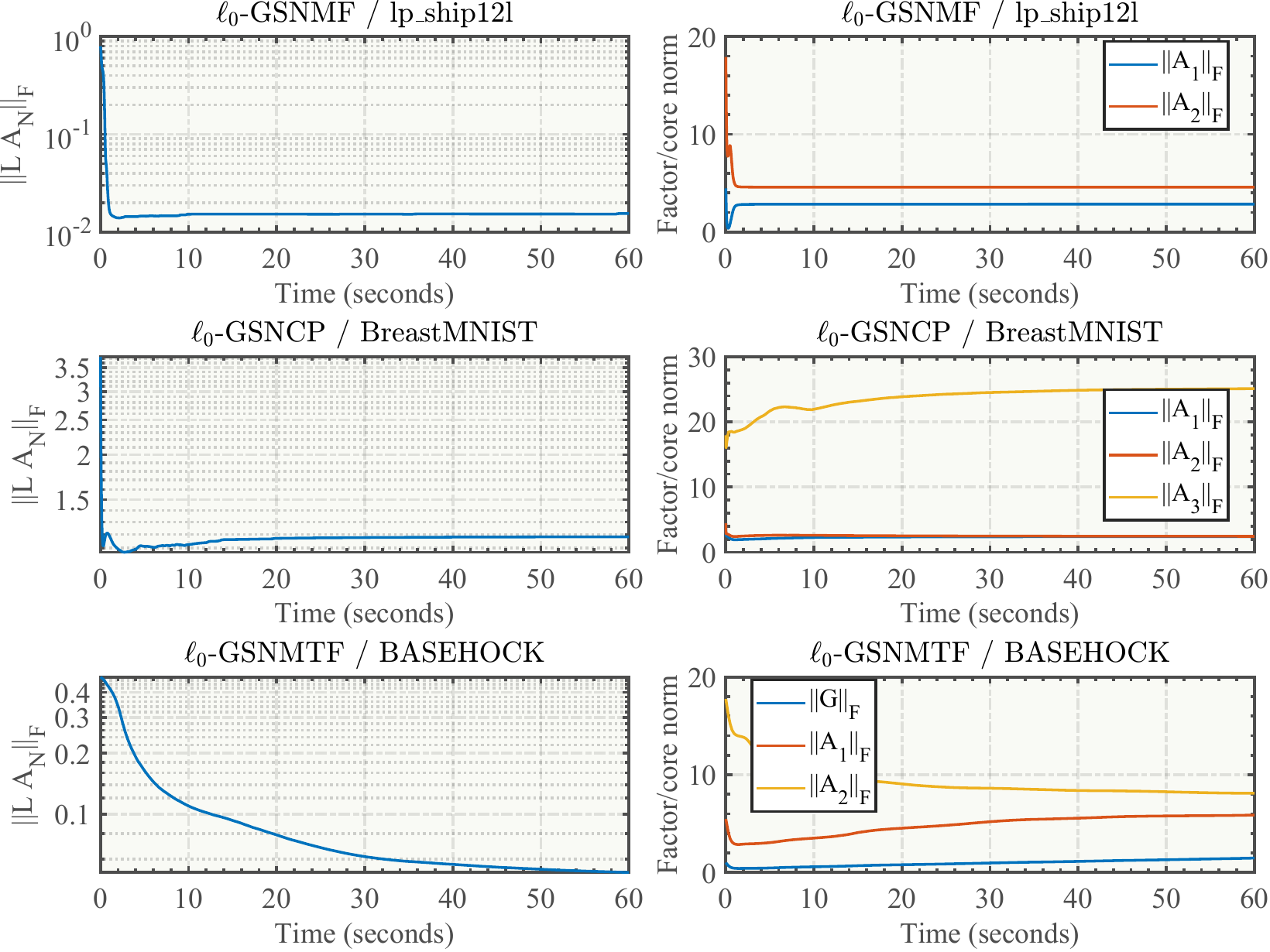}
\caption{Graph action and original-variable scales along the CBPL iterates. The rows show $\ell_0$-GSNMF on lp\_ship12l, $\ell_0$-GSNCP on BreastMNIST and $\ell_0$-GSNMTF on BASEHOCK. Left: $\|LA_N\|_F$. Right: the Frobenius norms of all factors and the independent core.}
\label{fig:structure}
\end{figure}

Figure~\ref{fig:structure} follows actual CBPL iterates for the three runs identified in its caption. The graph-action curves in the left column remain positive throughout the window, and the right column displays the norms of all factors and the independent core. Across all six runs, the smallest observed graph actions range from $1.39\times10^{-2}$ to $1.07$, and the maximum observed norms of the complete original block vector range from approximately $18.5$ to $29.3$. These results show how the same updates retain graph action and control factor and core scales, corresponding to the automatic compactness and uniform noncollapse results in Theorems~\ref{thm:compactness} and \ref{thm:graph-noncollapse}.

\endgroup

\section{Conclusion}
\label{sec:conclusion}
In this paper, we proposed CBPL for a class of multiblock nonconvex and nonsmooth optimization problems without assuming bounded iterates or coercivity of the original objective. The control potentials establish boundedness while retaining the original objective, constraints and variables. We established a general convergence framework giving finite length, whole-sequence convergence to an effective critical point of the original objective, and finite total variation of the original objective and all block functions. We also derived separate convergence rates for the original variables and objective values, including finite termination and linear or sublinear convergence. The original effective residuals are summable and have an $o(k^{-1})$ best-iterate rate. Scale orthogonality further gives full original criticality and stronger objective-value rates for ordinary nonnegative matrix and tensor factorizations, their $\ell_0$-constrained forms and related problems. Homogeneous balancing controls give summable absolute original-objective errors and an $o(k^{-1})$ best-iterate bound.

We established a complete degeneracy theory proving that scale criticality is equivalent to graph collapse and full original criticality eliminates graph action. The controls prevent this collapse, the effective directions retain genuine reconstruction changes, and the limits satisfy exact balance and full original block criticality on factors without graph regularization and on independent cores. Applications to $\ell_0$-GSNMF, $\ell_0$-GSNCP and $\ell_0$-GSNMTF give the practical, simple, and effective schemes for these applications, combining guaranteed convergence with explicit closed-form updates that can be applied directly to the original problem formulations. Numerical experiments illustrate the convergence and structural preservation of these schemes.
\bibliographystyle{siamplain}
\bibliography{references}

\end{document}